\documentclass[11pt]{amsart}

\newcommand{\ord}{\operatorname{ord}}
\newcommand{\Sym}{\operatorname{Sym}}
\newcommand{\tbf}[1]{\textbf{#1}}
\makeatletter

\newcommand{\nosubsection}[1]{%
  \par\addvspace{.7\linespacing\@plus.7\linespacing} 
  {\noindent\normalfont\bfseries #1} 
  \par\nopagebreak\addvspace{.5\linespacing} 
  \@afterindenttrue\@afterheading 
}

\def\subsection{\@startsection{subsection}{2}%
  \z@{.7\linespacing\@plus.7\linespacing}{.5\linespacing}%
  {\normalfont\bfseries}}

\def\subsubsection{\@startsection{subsubsection}{3}%
  \z@{.6\linespacing\@plus.6\linespacing}{.4\linespacing}%
  {\normalfont\bfseries}}
\makeatother

\usepackage{amssymb,amsthm,amsmath,amscd,mathtools}
\usepackage{mathrsfs}
\usepackage{xcolor}
\usepackage[all]{xy}
\usepackage{tikz-cd}
\usepackage{enumitem}
\usepackage[all]{xypic}
\usepackage{amssymb, amsmath, amscd, amsthm}
\usepackage{epstopdf}
\usepackage[bookmarks=true,bookmarksnumbered=true]{hyperref}
\usepackage[hmargin=2cm,vmargin=3cm]{geometry}
\usepackage[all]{xy}
\allowdisplaybreaks

\numberwithin{equation}{section}

\theoremstyle{plain}

\newtheorem{thm}{Theorem}[section]
\newtheorem{lem}[thm]{Lemma}

\newtheorem{fac}[thm]{Fact}
\theoremstyle{customproof}

\newtheorem*{customproof1}{Proof of Theorem~\ref{e}}

\newtheorem{prop}[thm]{Proposition}

\newtheorem{cor}[thm]{Corollary}

\newtheorem{main}{Theorem}

\theoremstyle{definition}
\newtheorem{defn}[thm]{Definition}
\newtheorem{rem}[thm]{Remark}

\newtheorem{con}[thm]{Conjecture}

\def\diag{\operatorname{diag}}
\def\disc{\operatorname{disc}}

\def\min{\operatorname{min}}
\def\max{\operatorname{max}}

\def\Hom{\operatorname{Hom}}
\def\ind{\operatorname{ind}}
\def\cind{\operatorname{c-ind}}
\def\Ind{\operatorname{Ind}}
\def\Irr{\operatorname{Irr}}

\def\Sym{\operatorname{Sym}}
\def\Re{\operatorname{Re}}

\def\beq{\begin{equation}}
\def\eeq{\end{equation}}
\def\beqn{\begin{equation*}}
\def\eeqn{\end{equation*}}
\def\beqna{\begin{eqnarray}}
\def\eeqna{\end{eqnarray}}
\def\beqnan{\begin{eqnarray*}}
\def\eeqnan{\end{eqnarray*}}

\def\wt{\widetilde}
\def\vi{\varphi}
\def\tbf{\textbf}
\def\mc{\mathcal}

\def\bs{\backslash}
\def\J{\mathrm{J}}
\def\M{\mathrm{M}}
\def\U{\mathrm{U}}
\def\T{\mathrm{T}}
\def\P{\mathrm{P}}
\def\B{\mathrm{B}}
\def\an{\mathrm{an}}

\def\As{\mathrm{As}}
\def\GL{\mathrm{GL}}

\def\SO{\mathrm{SO}}
\def\SL{\mathrm{SL}}
\def\Mp{\mathrm{Mp}}
\def\Sp{\mathrm{Sp}}
\def\SO{\mathrm{SO}}
\def\A{\mathbb{A}}
\def\As{\mathscr{A}}
\def\WD{\mathit{WD}}

\def\AA{\mathbb{A}}
\def\CC{\mathbb{C}}

\def\ZZ{\mathbb{Z}}
\def\RR{\mathbb{R}}

\def\ss{\subseteq}

\def\S{\mathrm{S}}

\def\e{\epsilon}

\def\bs{\backslash}

\def\1{\Eins}

\makeatletter
\def\varddots{\mathinner{\mkern1mu
    \raise\p@\hbox{.}\mkern2mu\raise4\p@\hbox{.}\mkern2mu
    \raise7\p@\vbox{\kern7\p@\hbox{.}}\mkern1mu}}
\makeatother

\newcommand{\BIGOP}[1]{\mathop{\mathchoice%
{\raise-0.22em\hbox{\huge $#1$}}
{\raise-0.05em\hbox{\Large $#1$}}{\hbox{\large $#1$}}{#1}}}

\newcommand{\BIGboxplus}{\mathop{\mathchoice%
{\raise-0.35em\hbox{\huge $\boxplus$}}%
{\raise-0.15em\hbox{\Large $\boxplus$}}{\hbox{\large $\boxplus$}}{\boxplus}}}
\title{Fourier-Jacobi periods and the non-tempered Gan--Gross--Prasad conjecture for $\Mp_{2n} \times \Sp_{2m}$}
\author{Jaeho Haan}
\address{Department of Mathematical Sciences, KAIST, 291 Daehak-ro, Yuseong-gu, Daejeon, 34141, South Korea}
\email{jaehohaan@gmail.com}

\begin{document}
\begin{abstract}
In this paper, we establish one direction of the non-tempered global Gan--Gross--Prasad (GGP) conjecture for the symplectic-metaplectic pairs $\Sp_{2n} \times \Mp_{2m}$ and $\Mp_{2n} \times \Sp_{2m}$. Our study focuses on two distinct families of non-tempered global $A$-parameters spanning all coranks, where standard relative trace formula methods remain unavailable.

Our approach proceeds along two distinct lines of attack. First, we treat the case where both members of the pair are non-tempered, utilizing regularized Fourier-Jacobi periods of residual Eisenstein series and establishing a reciprocal non-vanishing theorem. Second, we address the setting where only the metaplectic member is non-tempered and its central $L$-value vanishes, relying heavily on the global theta correspondence and explicit seesaw identities.

Along the way, we establish one direction of the tempered GGP conjecture for these pairs in arbitrary coranks, and prove a  dichotomy for the global associated $L$-packet of the metaplectic parameter $[1] \boxplus M'$. We show that this packet lies entirely in the residual spectrum or entirely in the cuspidal spectrum, a behavior determined uniformly by the non-vanishing of the central $L$-value $L(1/2, M')$.
\end{abstract}

\maketitle
\section{\textbf{Introduction}}

\nosubsection{Motivation and context}
Special values of automorphic $L$-functions lie at the heart of modern number theory and arithmetic geometry. In the automorphic setting, a guiding principle is that distinguished period integrals detect central (or near-central) values of suitable $L$-functions. A particularly influential manifestation of this principle is the Gan--Gross--Prasad (GGP) conjecture. Originally formulated by Gross and Prasad \cite{GP92, GP94} for orthogonal groups, the conjecture relates the non-vanishing of period integrals on special orthogonal groups to the non-vanishing of central critical values of Rankin--Selberg $L$-functions.

Together with Gan \cite{Gan2}, they later extended this conjecture to encompass all classical and metaplectic groups. This vast framework naturally splits into two domains: Bessel periods, which are associated with orthogonal or Hermitian unitary groups, and Fourier--Jacobi periods, which are attached to symplectic-metaplectic or skew-Hermitian unitary groups.

The original conjectures of \cite{Gan2} concerned representations lying in tempered $L$-packets. In \cite{Gan1}, the framework was elegantly reformulated using Arthur's classification parameters (global $A$-parameters). In this modernized formulation, the central Rankin--Selberg value is replaced by a canonical extended $L$-value, $L(z,M,N)\big|_{z=0}$, defined purely in terms of the $A$-parameters of the two groups. Furthermore, the non-vanishing of a period is predicted to force the pair of parameters $(M,N)$ to be structurally "relevant." We refer to this generalized statement as the \emph{non-tempered GGP conjecture}.

The present paper focuses on Fourier--Jacobi periods for the symplectic-metaplectic group pairs
\[
(G_n,G_m)\in\bigl\{(\Sp(W_n),\Mp(W_m)),\ (\Mp(W_n),\Sp(W_m))\bigr\},
\qquad n\ge m.
\]
Our primary objective is to prove one direction of the non-tempered GGP conjecture for two specific, arbitrarily high-corank families of non-tempered $A$-parameters: one where both members of the pair are non-tempered, and another where only the metaplectic member is non-tempered.

\nosubsection{Fourier--Jacobi periods}
We first recall the periods in question. Let $W_n$ and $W_m$ be symplectic spaces over a number field $F$ of dimensions $2n$ and $2m$, respectively, with $n\ge m$. Let $X$ and $X^*$ be isotropic subspaces of $W_n$ such that $X+X^*$ decomposes as a sum of $r = n-m$ hyperbolic planes, and assume $W_m$ is the orthogonal complement of $X+X^*$ in $W_n$.

For $k \in \{n,m\}$, let $\Sp_{2k}$ denote the isometry group of $W_k$ and $\Mp_{2k}(\A)$ its metaplectic double cover. We view $\Sp_{2m}$ as the subgroup of $\Sp_{2n}$ acting trivially on the orthogonal complement of $W_m$. Fixing a complete flag of $X$, let $N_{n,r}$ be the unipotent radical of its stabilizer in $\Sp_{2n}$. The group $\Sp_{2m}$ naturally acts on $N_{n,r}$ by conjugation. We then form the semidirect product $H_{n,r} = N_{n,r}\rtimes\Sp_{2m}$.

There exists a global generic Weil representation $\nu_{\psi^{-1},W_m}$ of $(N_{n,r}\rtimes\Mp_{2m})(\A)$, realized on the Schwartz space $\mc S(\A^m)$. For any Schwartz function $f$, the associated theta series $\Theta_{\psi^{-1},W_m}(f)$ provides an $H_{n,r}(F)$-invariant genuine automorphic realization of $\nu_{\psi^{-1},W_m}$.

Let $\pi_1\boxtimes\pi_2$ be an irreducible automorphic representation of $\Mp_{2n}(\A)\times\Mp_{2m}(\A)$ where exactly one factor is genuine. For automorphic forms $\vi_1\in\pi_1$, $\vi_2\in\pi_2$, and a Schwartz function $f\in\nu_{\psi^{-1},W_m}$, the $\psi$-\emph{Fourier--Jacobi period} is defined as
\[
\mc{FJ}_\psi(\vi_1,\vi_2,f)
=
\int_{[H_{n,r}]}
\vi_1(u\wt g)\,\vi_2(\wt g)\,
\Theta_{\psi^{-1},W_m}(f)(u,\wt g)\,du\,dg.
\]
Here, $\wt g$ represents any preimage of $g$ in $\Mp_{2m}(\A)$. Because the product of two genuine forms is non-genuine, the integral is well-defined independent of this choice. 

If one of the representations $\pi_1$ or $\pi_2$ is cuspidal, this integral converges absolutely. However, the non-tempered parameters we treat in this paper frequently surface in the residual spectrum, meaning neither factor is cuspidal. In such cases, the integral diverges and requires regularized integration. We adopt the regularization scheme introduced in \cite{H}, which safely extends the classical period integral via mixed truncation. We retain the notation $\mc{FJ}_\psi$ for this regularized period.

\nosubsection{The global GGP conjecture}
Let us precisely formulate the conjecture using global $A$-parameters. Recall that a discrete global $A$-parameter $Y$ defines an associated global $L$-parameter $\phi_Y$. We denote by $\Pi_{\phi_Y}^{\mathrm{aut}}$ the set of irreducible discrete automorphic representations whose local components lie in the local $L$-packet of $\phi_{Y,v}$ at every place $v$. This set is referred to as the \emph{associated global $L$-packet} of $Y$.

\begin{con}[{\cite[Conjecture~9.1]{Gan1}}]\label{conn}
Let $(G_n,G_m)\in\{(\Sp(W_n),\Mp(W_m)),(\Mp(W_n),\Sp(W_m))\}$ and let $M\times N$ be a discrete global $A$-parameter of $G_n\times G_m$. Let $\pi_1\boxtimes\pi_2$ be an irreducible discrete automorphic representation of $G_n(\A)\times G_m(\A)$ with $A$-parameter $M\times N$. Then the following assertions hold:
\begin{enumerate}
\item If the Fourier-Jacobi period $\mc{FJ}_\psi$ is nonzero on $\pi_1\boxtimes\pi_2\boxtimes\nu_{\psi^{-1},W_m}$, then $(M,N)$ is a relevant pair.
\item Assume that $(M,N)$ is a relevant pair. Then the following are equivalent:
\begin{enumerate}
\item[(1)] There exists $(\pi_1',\pi_2')\in\Pi_{\phi_M}^{\mathrm{aut}}\times \Pi_{\phi_N}^{\mathrm{aut}}$ such that $\mc{FJ}_\psi$ is nonzero on $\pi_1'\boxtimes\pi_2'\boxtimes\nu_{\psi^{-1},W_m}$.
\item[(2)] The extended central $L$-value satisfies $L(z,M,N)\big|_{z=0}\ne0$.
\end{enumerate}
\end{enumerate}
\end{con}

When $M$ and $N$ are tempered parameters, the condition $L(z,M,N)\big|_{z=0}\ne0$ simplifies exactly to the non-vanishing of the standard central value $L(\frac12,M\times N)\ne0$. Thus, Conjecture \ref{conn} elegantly subsumes the original tempered GGP conjecture.

\nosubsection{Known results and the non-tempered challenge}
For tempered representations, the field has witnessed immense progress over the past decade. The relative trace formula has proven extraordinarily successful for unitary groups, resolving the tempered conjecture almost entirely \cite{BPCZ22,BPC25,BPLZZ21,W1,W2,Xu14,Xu16,BLX}. For classical groups other than unitary groups, one direction of the tempered conjecture has been established by relating the original period integrals to the regularized periods of certain residual Eisenstein series on larger ambient groups. This powerful idea was first pioneered by Ginzburg, Jiang, and Rallis \cite{GJR} for the symplectic-metaplectic pairs $\Mp_{2m} \times \Sp_{2n}$. However, Jiang and Zhang \cite[p.~742]{JZ} pointed out a gap in the proof of \cite[Theorem~5.1]{GJR} when $m \ne n$. Thereafter, Jiang and Zhang \cite{JSZ, JZ} bypassed this obstruction by developing a general Rankin-Selberg integral theory, successfully proving the direction for orthogonal and Hermitian unitary groups (see \cite[p.~742]{JZ}). This approach holds great potential to be adapted for symplectic-metaplectic groups.

By contrast, the non-tempered spectrum remains largely unexplored. Although the relative trace formula has proven highly successful in the tempered setting, its application to the non-tempered case remains elusive. It is worth noting, however, that significant progress has recently been made on the non-tempered GGP conjecture for general linear groups. For example, Boisseau \cite{Boi} resolved the split case $\GL_n \times \GL_{n+1}$ via residues of Rankin-Selberg zeta integrals.

However, for orthogonal, symplectic, and unitary groups, results have been heavily restricted to low-rank or highly specific cases. For corank one and small groups, specifically $(\SO_5,\SO_4)$, Gan and Gurevich \cite{GG09} and Gurevich and Szpruch \cite{GS15} proved non-tempered cases using seesaw identities in theta correspondence. Similarly, for $\U_3\times\U_2$, Gelbart, Rogawski, and Soudry \cite{GRoS97} settled parameters of the form $(\mathbb{I}_{\GL_1}\oplus M_0,[1])$. 

Cases of arbitrarily high corank have been explored in a handful of instances, such as \cite{HK}, though these are primarily restricted to situations where the smaller group has rank at most two (e.g., $(\SO_{2n},\SO_3)$ or $(\Sp_{2n},\Mp_2)$). Even more crucially, across the existing literature for classical groups, essentially all proven results heavily rely on the assumption that at least one of the parameters is tempered. The scenario where \emph{both} members of the pair are non-tempered has remained largely out of reach.

Beyond its natural place in Arthur's classification, the non-tempered conjecture is forced upon us in the symplectic-metaplectic setting for a very concrete reason. If $\pi$ is an irreducible globally generic cuspidal automorphic representation of $\Sp_{2n}(\A)$, its global $A$-parameter is automatically tempered. However, metaplectic groups defy this expectation. By \cite[Theorem~11.2]{GRS}, a globally generic cuspidal representation $\wt\pi$ of $\Mp_{2n}(\A)$ \emph{frequently possesses a non-tempered $A$-parameter}. When it does, its parameter necessarily takes the shape
\begin{equation}\label{gennontemp}
M=[1]\boxplus M',
\qquad
M'\ \text{a discrete tempered symplectic global } A\text{-parameter}.
\end{equation}
Consequently, even completing the GGP conjecture for the most well-behaved class of representations---globally generic cuspidal representations---demands that we venture into the non-tempered spectrum. 

\nosubsection{Main results}
Our main result is the following.  We prove the $(1) \Rightarrow (2)$ direction of the non-tempered GGP conjecture for the first case in which \emph{both} members of the pair are non-tempered, and both groups are of arbitrary rank and arbitrary corank.

\begin{main}[Theorem~\ref{non}, Remark~\ref{conb}]\label{b}
Let $M'$ and $N'$ be discrete tempered global $A$-parameters of different types, and let $\sigma$ be an irreducible unitary cuspidal automorphic representation of $\GL_a(\A)$ of the same type as $N'$ but not contained in $N'$. Define
\[
M=\bigl(\sigma\boxtimes[1]\bigr)\boxplus M',
\qquad
N=\bigl(\sigma\boxtimes[2]\bigr)\boxplus N'.
\]
Suppose that $M\times N$ is a discrete global $A$-parameter of $\Mp_{2n}\times\Sp_{2m}$ or of $\Sp_{2n}\times\Mp_{2m}$. Then $(M,N)$ is a relevant pair. Assume that $L(\frac12,M'\times\sigma)\ne0$. 

For any $\pi_1\in\Pi_{\phi_M}^{\mathrm{aut}}$ and $\pi_2\in\Pi_{\phi_N}^{\mathrm{aut}}$,
\[
\mc{FJ}_\psi\bigl(\pi_1,\pi_2,\nu_{\psi^{-1},W_m}\bigr)\ne0
\quad\Longrightarrow\quad
L(z,M,N)\big|_{z=0}\ne0.
\]
\end{main}

The non-vanishing hypothesis $L(\frac12,M'\times\sigma)\ne0$ plays a pivotal structural role. It guarantees that the global $L$-packet of $M$ is realized entirely in the residual spectrum of the larger group. Our proof exploits this residual realization directly, requiring no genericity assumptions on the representations $\pi_1$ or $\pi_2$.

What happens when this central value vanishes? The residual realization collapses, and the associated global $L$-packet is instead realized in the cuspidal spectrum. We fully resolve this complementary, central-vanishing scenario for the basic non-tempered parameter mentioned in \eqref{gennontemp}.

\begin{main}[Theorem~\ref{non0}]\label{c}
Let $n\ge1$, and let $M'$ be a discrete tempered symplectic global $A$-parameter of dimension $2n-2$. Define $M=[1]\boxplus M'$ as a discrete global $A$-parameter for $\Mp_{2n}$. Let $N$ be a discrete tempered global $A$-parameter for $\Sp_{2n}$. Assume that the central value vanishes:
\[
L\Bigl(\tfrac12,M'\Bigr)=0.
\]
Then for any $\pi_1\in\Pi_{\phi_M}^{\mathrm{aut}}$ and $\pi_2\in\Pi_{\phi_N}^{\mathrm{aut}}$,
\[
\mc{FJ}_\psi\bigl(\pi_1,\pi_2,\nu_{\psi^{-1},W_n}\bigr)\ne0
\quad\Longrightarrow\quad
(M,N)\ \text{is relevant and}\quad
L(z,M,N)\big|_{z=0}\ne0.
\]
\end{main}

Combining Theorem \ref{c} with the prior results in \cite{Y} and the structural description \eqref{gennontemp}, we obtain the full generic case of Conjecture \ref{conn} (in the $(1)\Rightarrow(2)$ direction), completely independent of any temperedness assumption.

\begin{cor}[Corollary~\ref{non1}]\label{gencor}
Let $\wt\pi$ and $\pi$ be irreducible globally generic cuspidal automorphic representations of $\Mp_{2n}(\A)$ and $\Sp_{2n}(\A)$, with $A$-parameters $M$ and $N$ respectively. If $\mc{FJ}_\psi(\wt\pi,\pi,\nu_{\psi^{-1},W_n})\ne0$, then $(M,N)$ is a relevant pair and $L(z,M,N)\big|_{z=0}\ne0$.
\end{cor}

The two hypotheses $L(\frac12,M')\ne0$ and $L(\frac12,M')=0$ appearing in
Theorems~\ref{b} and~\ref{c} are not merely convenient: they separate the
two possible spectral realizations of the associated global $L$-packet. We record this dichotomy as an independent result:

\begin{main}[Corollary~\ref{central-value-packet-dichotomy}]\label{d}
Let $n\ge1$, let $M'$ be a discrete tempered symplectic global $A$-parameter of dimension $2n-2$, and let $M=[1]\boxplus M'$. Then:
\begin{enumerate}
\item The global $L$-packet is non-empty: $\Pi_{\phi_M}^{\mathrm{aut}}\ne\varnothing$.
\item If $L(\frac12,M')=0$, then the entire packet is cuspidal: $\Pi_{\phi_M}^{\mathrm{aut}}\subset\Irr_{\mathrm{cusp}}(\Mp_{2n})$.
\item If $L(\frac12,M')\ne0$, then the entire packet is residual: $\Pi_{\phi_M}^{\mathrm{aut}}\subset\Irr_{\mathrm{res}}(\Mp_{2n})$.
\end{enumerate}
\end{main}

Finally, as the main input to Theorem \ref{b}, we establish the underlying tempered GGP conjecture for arbitrary coranks. 

\begin{main}\label{e}
Let $(G_n,G_m)$ be as in Conjecture \ref{conn} with $n\ge m$, and let $M,N$ be discrete tempered global $A$-parameters of different types for $G_n$ and $G_m$ respectively. For any $\pi_1\in\Pi_{\phi_M}^{\mathrm{aut}}$ and $\pi_2\in\Pi_{\phi_N}^{\mathrm{aut}}$,
\[
\mc{FJ}_\psi\bigl(\pi_1,\pi_2,\nu_{\psi^{-1},W_m}\bigr)\ne0
\quad\Longrightarrow\quad
L\Bigl(\tfrac12,M\times N\Bigr)\ne0.
\]
\end{main}

This result extends \cite[Theorem 1.1]{Y} from $m=n$ to arbitrary corank $r = n-m > 0$. As a side benefit, Theorem \ref{e} successfully removes a previously necessary assumption in Atobe's non-vanishing criterion for generalized Miyawaki lifts \cite[Theorem 5.5]{At0}.

\nosubsection{Strategy of the proofs}
Our proofs divide naturally according to the central value of the underlying parameters. 

\textbf{Step 1: The tempered case via residual Eisenstein series.}
The proof of Theorem \ref{e} builds upon the strategy introduced by Ginzburg, Jiang, and Rallis \cite{GJR}. While their pioneering work encountered a technical gap for arbitrary coranks ($n \neq m$), our approach successfully circumvents this obstruction by carefully interchanging the two residual Eisenstein series and employing the mixed truncation regularization from \cite{H}. This interchange does not eliminate the difficulty; rather, it transforms the obstruction encountered in \cite{GJR} into a vanishing statement for Fourier--Jacobi periods attached to an Eisenstein series (Proposition~\ref{l1}), presenting a more formidable challenge than before.

In \cite[Lemma~9.4]{H}, the skew-Hermitian analogue of Proposition~\ref{l1} was established via the computation of Jacquet modules. Such a computation cannot be transplanted here. While one immediate reason is its incompatibility with covering groups, the more critical issues are a recently discovered gap in that proof and a case that it does not cover at all (see Remark~\ref{gapH} for details). Consequently, we present a novel and complete proof that also repairs the argument in \cite{H}. Our approach relies on a delicate local analysis, conducted in Propositions~\ref{key} and~\ref{key3} for most cases, and in Proposition~\ref{theta-fj-aux} for the remaining case.

Beyond this issue, translating the framework of \cite{H} to the symplectic-metaplectic setting is technically more demanding than the skew-Hermitian case due to the intrinsic complexities of metaplectic covering groups. To streamline the exposition, we omit arguments that are strictly analogous to those in \cite{H}, providing complete proofs and explanations only where the presence of the covering group causes genuine deviations. These components ultimately yield Theorem~\ref{e}, alongside a \emph{reciprocal nonvanishing theorem}, which asserts that the period is non-zero on the original pair if and only if it is non-zero on the ambient residual pair.

\textbf{Step 2: Vogan's residual realization for the double non-tempered case.}
Theorem \ref{b} deals with arbitrary automorphic members of the packet. The bridge between the specific residual representations above and arbitrary packet members is provided by Proposition~\ref{voganres}. Using the endoscopic classification and component group calculations, we show that every member of the non-tempered packet occurs as a direct summand of the residual Eisenstein series $\mc E^{d-1}(\sigma,\rho)$, \emph{provided} $L(1/2, M'\times\sigma) \ne 0$. Multiplicity one identifies the given packet members with the corresponding summands of the residual pair, so that the period on the latter restricts to the given one. We then apply reciprocal nonvanishing and Theorem \ref{e} to convert the period non-vanishing into $L(1/2, M' \times N') \ne 0$, which cleanly matches the required order of the extended central $L$-value $L(z,M,N)$ at $z=0$.

\textbf{Step 3: Theta correspondence for the central-vanishing case.}
When $L(1/2, M') = 0$, Step 2 breaks down as the residual realization vanishes. Instead, we utilize the global descent construction of \cite{GRS} to construct a generic cuspidal representation $\tau$ on $\SO_{2n-1}$ with parameter $M'$. The vanishing of $L(1/2, M')$ forces the first non-zero theta lift of $\tau$ along the metaplectic tower $\{\Mp(W_k)\}_{k \ge 1}$ to be $\Theta_{\psi, V_{2n-1}, W_n}(\tau)$, which is a non-zero, globally generic cuspidal form inside our target packet. Yamana's criterion in \cite{Y0} ensures that every member of the packet arises through this construction. This theta realization enables us to apply the global seesaw identity. Writing
$V_{2n}=V_{2n-1}\oplus\langle-1\rangle$, the seesaw
\[
\xymatrix{
  \mathrm O(V_{2n}) \ar @{-}[d]\ar @{-}[dr] &
  \Mp(W_n)\times_{\{\pm1\}}\Mp(W_n) \ar @{-}[dl]\ar @{-}[d] \\
  \mathrm O(V_{2n-1})\times \mathrm O(V_1) & \Sp(W_n)
}
\]
converts the Fourier--Jacobi period of $\wt\pi\boxtimes\pi$ into a
Bessel-type period on $\SO(V_{2n})\times\SO(V_{2n-1})$ involving $\tau$
and a theta lift of $\pi$.  Analysing the resulting orthogonal period and
computing $\ord_{z=0}L(z,M,N)$ as before gives Theorem~\ref{c}.

\nosubsection{Organization of the paper}
Section~\ref{sec:prelim} fixes notation, recalls the local theta
correspondence, and collects what we need from Arthur's classification:
global $A$-parameters, their packets, and the properties of
representations with tempered $A$-parameter. Section~\ref{sec:jacquet}
carries out the local analysis of the Jacquet modules attached to
Fourier--Jacobi characters. Section~\ref{sec:residual} reviews the
residual Eisenstein series and the poles of the $L$-functions governing
them. Section~\ref{sec:reciprocal} establishes the reciprocal
nonvanishing theorem and deduces Theorem~\ref{e}. Finally,
Section~\ref{sec:nontemp} treats the non-tempered conjecture:
Subsection~\ref{FJR} gives the residual realization of an associated
global $L$-packet and Theorem~\ref{b}, and Subsection~\ref{gen} treats
the central-vanishing case, yielding Theorems~\ref{c} and~\ref{d}.

\section{\textbf{Preliminary}}\label{sec:prelim}
\subsection{Notations} We introduce some notation that will be used throughout the paper.

Let $F$ be a number field. Let $\A$ denote the adele ring of $F$ and $\A_{\text{fin}}$ the ring of finite adeles. For a place $v$ of $F$, we write $F_v$ for the completion of $F$ at $v$ and $|\cdot|_{v}$ for the normalized absolute value on $F_v$. We denote the trivial character by $\tbf{1}$ in both local and global contexts.

Let $q$ be the cardinality of the residue field of $F_v$. For an algebraic group $G$ over $F$ and for a $F$-algebra $R$, write $G(R)$ for the group of $R$-points of $G$. Denote $G(F) \bs G(\A)$ by $[G]$. Let $\psi$ be a non-trivial additive character of $F \bs \A$ and for $\alpha \in F$, put $\psi_{\alpha}(x):=\psi(\alpha x)$ for $x \in \A$ and $\eta_{\alpha}$ the quadratic character of $\A^{\times}/F^{\times}$ associated to the quadratic extension of $F(\sqrt{\alpha})/F$ by the global class field theory. For $x\in \A^{\times}$, put $|x|:=\prod_v |x|_{v}$ and for a positive integer $k$, let $\mu_k$ be the group of $k$-th root of unity in $\CC$. Denote by $( \  , \ )_v:F_v \times F_v \to \{\pm1\}$ the local Hilbert symbol and by $( \  , \ )\coloneqq \prod_v ( \  , \ )_v : \A \times \A \to \A^{\times}$ the global Hilbert symbol. \\ \indent Occasionally, we view $|\cdot|$ and $|\cdot|_{v}$ as a character of $\GL_n(\A)$ or $\GL_n(F_v)$ by composing with $\det$, respectively. For a character $\chi$ of $\GL_1(F_v)$ and $d \ge 1$, put $\chi(d):=\chi (\det_{\GL_d(F_v)})$ a character of $\GL_d(F_v)$. %and put  $\ol{\chi}(d)\coloneqq \chi  |\cdot|^{-\frac{d}{2}} \times \chi   |\cdot|^{-\frac{d}{2}-1}  \times \cdots \times \chi  |\cdot|^{\frac{d}{2}-1} \times \chi  |\cdot|^{\frac{d}{2}}$ a character of $(\GL_1(F_v))^{\times (d+1)}$.

Let $W_{k}$ be a $2k$-dimensional $F$-vector space with a symplectic structure $\left<\cdot ,\cdot \right>$ and write $\Sp_{2k}$ for its symplectic group. We also write $\mathcal{H}_{W_{k}}\coloneqq W_{k} \oplus F$ the Heisenberg group associated to $(W_{k}, \left<\cdot ,\cdot \right>)$ and view it as an algebraic group over $F$. When $(W, \left< \cdot , \cdot \right>)$ is a symplectic space and $W'$ is a subspace of $W$, we view $W'$ is also symplectic space with the inherited symplectic form $\left< \cdot , \cdot \right>$ of $W$.

Fix an arbitrary positive integer $k$ and maximal totally isotropic subspaces $X$ and $X^*$ of $W_k$, which are in duality with respect to $\left<\cdot , \cdot \right>$.
Fix a complete flag in $X$
\(
0=X_0 \subset X_1 \subset \cdots \subset X_k=X,
\)
and choose a basis $\{e_1,e_2,\cdots,e_k\}$ of $X$
such that $\{e_1,\cdots,e_j\}$ is a basis of $X_j$ for $1\le j \le k$.
Let $\{f_1,f_2,\cdots ,f_k\}$ be the basis of $X^*$ which is dual to the fixed basis $\{e_1,e_2,\cdots,e_k\}$ of $X$,
i.e., $\left<e_i,f_j\right>=\delta_{ij}$ for $1\le i,j \le n$,
where $\delta_{i,j}$ denotes the Kronecker delta.
We write $X^*_j$ for the subspace of $X^*$ spanned by $\{f_1, f_2,\cdots ,f_j\}$
and $W_{k-j}$ for the orthogonal complement of $X_j + X^*_j$ in $W_k$. Then $(W_{k-j}, \left<\cdot ,\cdot \right>)$ is
a symplectic $F$-space of dimension $2(k-j)$.
\par

For each $1\le j \le k$, denote by $\P_{k,j}$ the parabolic subgroup of $\Sp_{2k}$ stabilizing $X_j$,
by $\U_{k,j}$ its unipotent radical, and by $\M_{k,j}$ the Levi subgroup of $\P_{k,j}$ stabilizing $X_j^*$.
Then $\M_{k,j} \simeq \GL(X_j) \times \Sp_{2k-2j}$.
(Here, we regard $\GL(X_j)$ as the subgroup of $\M_{k,j}$ that acts as the identity map on $W_{k-j}$.)
Denote by $N_{k,j}$ (resp. $\mathcal{N}_{j}$) the unipotent radical of the parabolic subgroup of $\Sp_{2k}$ (resp., $\GL(X_j)$)
stabilizing the flag $\{0\}=X_0 \subset X_1 \subset \cdots \subset X_j$.
If we regard $\mathcal{N}_{j}$ as a subgroup of $\M_{k,j}\simeq \GL(X_j) \times \Sp_{2k-2j}$,
it acts on $\U_{k,j}$ and so $N_{k,j}=\U_{k,j} \rtimes \mathcal{N}_{j}$. When $j=0$, $\mathcal{N}_0$ denotes the trivial group.

For a sequence $\mathbf{m}=(m_1,m_2,\ldots,m_t)$ of positive integers whose sum is $k$, we denote by $\mc{P}_{\mathbf{m}}$ the standard parabolic subgroup of $\GL_{k}$
whose Levi subgroup is isomorphic to $\GL_{m_1} \times \cdots \times \GL_{m_t}$. Especially, when $\mc{m}=(1,1,\cdots,1,1)$, we write $\mc{B}_k$ for  $\mc{P}_{(1,1,\cdots,1,1)}$. Note that $\mc{N}_k$ is the unipotent radical of $\mc{B}_k$.

Let $\B_k$ be the minimal parabolic subgroup of $\Sp_{2k}$ that stabilizes the above complete flag of $X$, and $\B_k=\T_k\U_k$ represents its Levi decomposition. We call any $F$-parabolic subgroup $\P=\M\U$ of $\Sp_{2k}$ \emph{standard} if it contains $\B_k$.
Fix a maximal compact subgroup $K=\prod_v K_v$ of $\Sp_{2k}(\A)$ such that $\Sp_{2k}(\A)=\B_k(\A)K$, and for every standard parabolic subgroup $\P=\M\U$ of $\Sp_{2k}$,
\( \P(\A)\cap K=(\M(\A) \cap K)(\U(\A)\cap K),\)
where $\M(\A) \cap K$ remains maximal compact in $\M(\A)$ (see \cite[I.1.4]{Mo}). Let $\M(\A)^1$ be the intersection of the kernels of the homomorphisms $|\chi|:\M(\A)\mapsto(\RR_{+})^{\times}$, where $\chi$ ranges over the lattice $X^*(\M)$ of $F$-rational characters of $\M$.

We give the Haar measure on $\U(\A)$ so that the volume of $[\U]$ is 1 and give the Haar measure on $K$ so that the total volume of $K$  is $1$.
We also choose Haar measures on $\M(\A)$ for all Levi subgroups $\M$ of $\Sp_{2k}(\A)$ compatibly with respect to the Iwasawa decomposition.

For each place $v$, the covering group $\wt{\GL}_{k}(F_v)$ of $\GL_k(F_v)$ is defined by $\GL_k(F_v) \times \{\pm1\}$ with group law
\[
(g_1,\e_1)\cdot(g_2,\e_2)=(g_1g_2,\e_1\e_2\cdot(\det g_1, \det g_2)_v), \quad \quad g_1,g_2 \in \GL_k(F_v).
\] Let $\Mp_{2k}(F_v)$ denote the metaplectic double covering of $\Sp_{2k}(F_v)$. Specifically,
\(\Mp_{2k}(F_v) = \Sp_{2k}(F_v) \times \{\pm1\}\)
as a set, with a group law given by
\[(g_1,\epsilon_1) \cdot (g_2,\epsilon_2) = (g_1g_2, \epsilon_1\epsilon_2 \cdot c(g_1,g_2)),\]
 where $c$ is the Rao's 2-cocycle on $\Sp(W_n)$, whose values lie in $\{\pm1\}$. (For its definition, see \cite{R}.)

We denote the double covering map as $pr: \Mp_{2k}(F_v) \to \Sp_{2k}(F_v)$. Similarly, let $\Mp_{2k}(\A)$ be the metaplectic double covering of $\Sp_{2k}(\A)$, compatible with $\Mp_{2k}(F_v)$ for all places $v$. Again, we use the notation $pr$ for the double cover map from $\Mp_{2k}(\A)$ to $\Sp_{2k}(\A)$. Notably, $\U_k
(\A)$ canonically lifts into $\Mp_{2k}(\A)$ (see \cite[Appendix I]{Mo}), and Weil's result \cite{We} ensures that $\Sp_{2k}(F)$ lifts uniquely into $\Mp_{2k}(\A)$.

When considering the unipotent subgroup $\U$ of any parabolic subgroup of $\Sp_{2k}$ (or a subgroup $\S$ of $\Sp_{2k}(F)$), we maintain the same notation $\U$ (or $\S$) to denote their image under this lifting. Furthermore, for a subgroup $\J$ of $\Sp_{2k}$, we use $\widetilde{\J}(\A)$ (resp. $\widetilde{\J}(F_v)$) to represent $pr^{-1}(\J(\A))$ (resp. $pr^{-1}(\J(F_v))$). Therefore, $\wt{\P}_{k,a}=\wt{\M}_{k,a} \cdot \U_{k,a}$, where
\(\wt{\M}_{k,a}\simeq \wt{\GL}_{a} \times_{\{\pm1\}}\Mp_{{2k-2a}}.
\) When $\J$ is the trivial group, $\widetilde{\J}(\AA)$ (resp. $\widetilde{\J}(F_v)$) equals $\{\pm1\}=\{1,-1\}$.

Lastly, a function $f$ on $\Mp_{2k}(\AA)$ (resp. $\Mp_{2k}(F_v)$) is considered genuine if $f(\epsilon \cdot g) = \epsilon \cdot f(g)$ for $\epsilon \in \{1,-1\}$ and $g \in \Mp_{2k}(\AA)$ (resp. $\Mp_{2k}(F_v)$).

 Let $\delta_{\P}$ denote the modulus function of $\P(\A)$ (resp. $\widetilde{\P}(\A)$) (see \cite[I.2.17]{Mo}). We use $\Ind$, $\ind$, and $\cind$ for normalized, unnormalized, and compactly supported induction, respectively.

We also use the standard shorthand for normalized parabolic induction. If $\rho_i$ is a representation of $\GL_{a_i}$ and $\pi_0$ is a representation of a smaller symplectic group, then
\[
\rho_1\times\cdots\times\rho_t
=\Ind_{\mc P}^{\GL_{a_1+\cdots+a_t}}
(\rho_1\boxtimes\cdots\boxtimes\rho_t),
\qquad
\rho_1\times\cdots\times\rho_t\rtimes\pi_0
=\Ind_{P}^{\Sp}
(\rho_1\boxtimes\cdots\boxtimes\rho_t\boxtimes\pi_0),
\]
where the ambient groups and standard parabolics are determined by the inducing data. The same notation is used for metaplectic groups, with the general-linear factors regarded as genuine representations of the corresponding covers. Under the fixed $\psi$-dependent identification, we write $\rho\rtimes\widetilde\pi_0$ for $\rho_\psi\rtimes\widetilde\pi_0$. These conventions apply both locally and globally. We retain the explicit $\Ind$ notation when the inducing subgroup or the ambient group must be emphasized.

For an automorphic representation $\rho$ of $\M(\A)$ (resp. $\widetilde{\M}(\A)$) and $z \in \CC$, we define $I(z,\rho)$ as the induced representation:
\[ I(z,\rho) = \text{Ind}_{\P(\A)}^{\Sp_{2k}(\A)} (\delta_{\P}^{z} \cdot \rho) \quad \text{(resp. } I(z,\rho) = \text{Ind}_{\widetilde{\P}(\A)}^{\Mp_{2k}(\A)} (\delta_{\P}^{z} \cdot \rho) \text{)}. \]

This induced representation has a realization in the $\CC$-valued function space as follows:

Let \(\As_{\P}(\Sp_{2k})\) be the space of automorphic forms on
\(\U(\A)\P(F)\backslash\Sp_{2k}(\A)\), and let
\(\As_{\P}(\Mp_{2k})\) be the analogous space of genuine automorphic
forms on \(\U(\A)\P(F)\backslash\Mp_{2k}(\A)\). These forms are
smooth, \(K\)-finite (respectively, \(\widetilde K\)-finite), and
\(\mathfrak z\)-finite. Here \(\mathfrak z\) is the center of the
universal enveloping algebra of the complexified Lie algebra
\(\mathfrak g_{\CC}\) of
\(\Sp_{2k}(F\otimes_{\mathbb Q}\mathbb R)\). The two spaces are
modules for, respectively,
\[
\Sp_{2k}(\A_{\mathrm{fin}})
\times\left(\mathfrak g_{\CC},\prod_{v\mid\infty}K_v\right)
\quad\text{and}\quad
\Mp_{2k}(\A_{\mathrm{fin}})
\times\left(\mathfrak g_{\CC},
\prod_{v\mid\infty}\widetilde K_v\right).
\]

When $\P=\Sp_{2k}$, we simply write $\As(\Sp_{2k})$ (resp. $\As(\Mp_{2k})$) to denote $\As_{\P}(\Sp_{2k})$ (resp. $\As_{\P}(\Mp_{2k})$).

For an automorphic representation $\rho$ of $\M(\A)$ (resp. a genuine automorphic representation $\rho$ of $\widetilde{\M}(\A)$) and $z \in \CC$, we define $\As_{\P}^{\vert \cdot \vert^z \cdot \rho}(\Sp_{2k})$ as the subspace of functions $\phi \in \As_{\P}(\Sp_{2k})$ (resp. $\phi \in \As_{\P}(\Mp_{2k})$) such that for all $k \in K$ (resp. $\widetilde{K}$), the function $m \mapsto \delta_{\P}(m)^{-\frac{1}{2}-z} \cdot \phi(mk)$ belongs to the space of $\rho$.

Although this paper is mainly concerned with symplectic and
metaplectic groups, we shall also use orthogonal groups and the local
theta correspondence for orthogonal--metaplectic dual pairs.
For a non-degenerate quadratic space $V$ over $F$, we write
$\mathrm O(V)$ for its full isometry group and $\mathrm{SO}(V)$ for
its special orthogonal subgroup.  When only the dimension is being
emphasized, we write $V_d$ for a $d$-dimensional quadratic space; its
isometry class, discriminant, and Witt tower will be specified whenever
they are relevant.  If $V_d$ is split, we write $V_d^{\mathrm{spl}}$,
and abbreviate
\(
 \mathrm O_d=\mathrm O(V_d^{\mathrm{spl}}),
 \
 \mathrm{SO}_d=\mathrm{SO}(V_d^{\mathrm{spl}}).
\)
Throughout, the discriminant of $V$ is normalized by
\[
 \disc(V)=(-1)^{\dim V(\dim V-1)/2}\det(V)
 \quad\text{in }F^{\times}/F^{\times2},
\]
where $\det(V)$ is the determinant of the Gram matrix of the quadratic
form $Q(x)=\frac12(x,x)_V$, and $\chi_V$ denotes the quadratic character
attached to $\disc(V)$.  With this normalization $\disc$ is constant
along a Witt tower, and $\disc(V\oplus\langle a\rangle)=-a\,\disc(V)$
when $\dim V$ is odd.

\subsection{Local theta correspondence}

We now recall the local theta correspondence for the dual pair
$\mathrm O(V)\times\Mp(W)$.

Let $F$ be a non-archimedean local field of characteristic zero.
For a quadratic space $V$ and a symplectic space $W$ over $F$, let
\(
 \omega_{\psi,V,W}
\)
denote the Weil representation of
$\mathrm O(V)\times\Mp(W)$ associated with the fixed additive
character $\psi$ and the Leray splittings used throughout the paper.

For an irreducible representation $\tau$ of $\mathrm O(V)$, the
maximal $\tau$-isotypic quotient of the Weil representation is of the
form
\(
 \omega_{\psi,V,W}\twoheadrightarrow
 \tau\boxtimes\Theta_{\psi,V,W}(\tau).
\)
The representation $\Theta_{\psi,V,W}(\tau)$ of $\Mp(W)$ is called
the \emph{big theta lift} of $\tau$.  If it is nonzero, its unique
irreducible quotient, whose existence and uniqueness follow from
Howe duality, is denoted by
\(
 \theta_{\psi,V,W}(\tau)
\)
and is called the \emph{small theta lift}.  Thus there is a natural
surjection
\(
 \Theta_{\psi,V,W}(\tau)
 \twoheadrightarrow
 \theta_{\psi,V,W}(\tau).
\)
The theta lifts of genuine irreducible representations of $\Mp(W)$
to $\mathrm O(V)$ are defined similarly.  We use an upper-case
$\Theta$ for big theta lifts and a lower-case $\theta$ for small
theta lifts; the additive character and the two spaces will always
be displayed in the notation.  We refer to
\cite[Section~13.3 and Chapter~18]{GKT} for the definitions and Howe
duality.

\subsection{Discrete global $A$-parameters and their packets}\label{dgp}

\subsubsection{Discrete global $A$-parameters}
For a nonnegative integer $d$, let $[d]$ denote the unique $(d+1)$-dimensional irreducible representation of $\SL_2(\CC)$. When $d=-1$, we understand $[-1]$ to be 0.

For a non-negative integer $k$, a global $A$-parameter for $\SO_{2k}(\A)$ and $\Sp_{2k}(\A)$ (resp. $\SO_{2k+1}(\A)$ and $\Mp_{2k}(\A)$) is a formal finite direct sum
\begin{align}\label{art}
M=\bigoplus_{i=1}^{r} M_i \boxtimes [d_i],
\end{align}
 where $M_i$ is an irreducible cuspidal unitary automorphic representation of $\GL_{n_i}(\A)$. For brevity, we write $M_i \boxtimes [0]=M_i$ and $\tbf{1} \boxtimes [d]=[d]$.

If $M$ is decomposed as in (\ref{art}), we say that $M$ is tempered if $d_i=0$ for all $i$.
In this case, we simply write $M=\bigoplus_{i=1}^r M_i$.
We write $\Sym^2$ for the symmetric square representation and $\wedge^2$ for the exterior square representation.

We say that such an $M$ is a discrete global $A$-parameter of $\SO_{2k}(\A)$ and $\Sp_{2k}(\A)$ (resp. $\SO_{2k+1}(\A)$ and $\Mp_{2n}(\A)$) if $(M_i,d_i)$ satisfy
\begin{itemize}
\item $d_i$ is an integer greater than or equal to $-1$ such that
\(\sum_{i=1}^r n_i(d_i+1) = 2k \text{ and } 2k+1 \quad \text{(resp. } 2k)\)
\item if $d_i$ is even (resp. odd), then $L(z,\Sym^2(M_i))$, the symmetric square $L$-function of $M_i$ has a pole at $z=1$ (i.e., $M_i$ is orthogonal)
\item  if $d_i$ is odd (resp. even), then $L(z,\wedge^2(M_i))$, the exterior square $L$-function of $M_i$ has a pole at $z=1$
(i.e., $M_i$ is symplectic)
\item the central character $\omega_{M_i}$ of $M_i$ satisfies $\prod_{i=1}^r \omega_{M_i}^{d_i+1}=\tbf{1}$, the trivial character of $\GL_1(\A)$
\item if $(d_i,M_i)=(d_j,M_j)$, then $i=j$.
\end{itemize}

In this case, we say that $M$ is orthogonal (resp. symplectic).

\subsubsection{Local parameters, packets and component groups}
\label{locpack}

From now on, let $G_{k}$ denote either $\SO_{2k}$, $\Sp_{2k}$,
$\SO_{2k+1}$ or $\Mp_{2k}$, and let
$M=\bigoplus_{i=1}^{r}M_i\boxtimes[d_i]$ be a global $A$-parameter of
$G_{k}$.  For a place $v$ of $F$, the localization
$M_v=\bigoplus_{i=1}^{r}M_{i,v}\boxtimes[d_i]$ is a local
$A$-parameter, each $M_{i,v}$ being viewed as a representation of the
Weil--Deligne group $\WD(F_v)$ through the local Langlands
correspondence for general linear groups.

For a global or local \(A\)-parameter \(Y\), write \(\phi_Y\) for its
associated \(L\)-parameter, obtained by composing the Arthur parameter
with the diagonal map
\[
w\longmapsto
\left(w,\diag\left(|w|^{1/2},|w|^{-1/2}\right)\right).
\]
Concretely, if \(Y=\bigoplus_{i=1}^{r}Y_i\boxtimes[d_i]\), then
\begin{equation}\label{assocL}
\phi_Y
=
\bigoplus_{i=1}^{r}\bigoplus_{j=0}^{d_i}
Y_i|\cdot|^{\frac{d_i}{2}-j},
\end{equation}
so that \(\phi_Y=Y\) precisely when \(Y\) is tempered.

To a local \(A\)-parameter \(Y_v\) is attached a finite multiset
\(\Pi_{Y_v}\) of irreducible representations of \(G_{k}(F_v)\), the local
\(A\)-packet, and likewise a local \(L\)-parameter \(\phi_v\) has a local
\(L\)-packet \(\Pi_{\phi_v}\).  In the symplectic case, we use the
packets of \cite[Theorem~1.5.1 and Sect.~2.2]{Ae}; in the metaplectic
case, we use the \(\psi_v\)-normalized packets of
\cite[Definition~4.6.1]{Li}.  We fix a global additive character
\(\psi=\bigotimes_v\psi_v\) throughout the paper and suppress it from the
local and global metaplectic packet notation.  We also suppress the group
when it is clear from the context.

Each of these packets carries an internal parametrization: a member
\(\pi_v\in\Pi_{\theta_v}\) determines a character of the component group
\(\mathcal S_{\theta_v}\) of \(\theta_v\), and we write
\(\langle s,\pi_v\rangle_{\theta_v,G}\) for its value at
\(s\in\mathcal S_{\theta_v}\).  The component groups themselves, the
characters entering the multiplicity formula, and the comparison map
between the component groups of \(Y_v\) and of \(\phi_{Y_v}\) play no
role before Section~\ref{sec:nontemp}; we therefore postpone their
description to Subsection~\ref{compgrp}, where they are first used.
Until then only the packets are needed.

\subsubsection{Global $A$-packets and associated global $L$-packets}
\label{assoc}

Let \(M\) be a discrete global \(A\)-parameter of \(G_{k}\).  Arthur's
multiplicity formula attaches to \(M\) a character \(\epsilon_M^{G}\) of
the global component group \(\mathcal S_M\), described in
Subsection~\ref{compgrp}.  The \emph{global \(A\)-packet} of \(M\) is
\begin{equation}\label{globalApacket}
 \Pi_M
 =
 \left\{
 \pi=\bigotimes_v{}'\pi_v
 \ :\
 \pi_v\in\Pi_{M_v}\ \text{for every }v,\quad
 \prod_v\langle s_v,\pi_v\rangle_{M_v,G}
 =\epsilon_M^{G}(s)\ \text{ for all }s\in\mathcal S_M
 \right\},
\end{equation}
the members being counted with the multiplicities of the local packets;
the product is well defined because \(\pi_v\) is unramified and its
character trivial at almost every place.  By the endoscopic
classification, the discrete automorphic spectrum decomposes as
\[
 \mc{A}^2(G_{k})=\bigsqcup_{M}\Pi_M,
\]
the union being over the discrete global \(A\)-parameters of \(G_{k}\),
and every member occurs with multiplicity one.  This is
\cite[Theorem~1.5.2]{Ae} for symplectic and special orthogonal groups,
\cite{Ish} for pure inner forms of odd special orthogonal groups, and
\cite[Theorems~5.1.1 and~5.4.1 and Corollary~5.4.4]{Li} for metaplectic
groups.  This classification was originally conditional on the
stabilization of the twisted trace formula and on a package of local
results announced in \cite{Ae}, chief among them the local intertwining
relation; the latter has since been established in \cite{AGIKMS}, so
that the classification now rests on the twisted weighted fundamental
lemma alone.  Every result of this paper inherits that single
hypothesis, and nothing further.  We accordingly say that each
\(\pi\in\Pi_M\) \emph{has global \(A\)-parameter} \(M\).  Note that \(\Pi_M\) depends on the choice of
\(\psi\) when \(G_{k}=\Mp_{2k}\).

We write $\Irr_{\mathrm{cusp}}(G_{k})$ (resp. $\Irr_{\mathrm{res}}(G_{k})$) for the
set of irreducible (genuine) cuspidal (resp. residual) automorphic
representations of $G_{k}(\A)$, so that
\(
\mc{A}^2(G_{k})
=
\Irr_{\mathrm{cusp}}(G_{k})\sqcup\Irr_{\mathrm{res}}(G_{k}).
\)

The following unramified description of \(\Pi_M\) is the one we use
whenever a representation is handed to us by an explicit construction
and its global \(A\)-parameter has to be identified; this happens in the
proofs of Proposition~\ref{l1} and
Proposition~\ref{central-zero-associated-packet}.

\begin{rem}
\label{sat-rem}
For two automorphic representations \(\pi_1\) and \(\pi_2\) of
\(G_{k}(\A)\), we say that \(\pi_1\) and \(\pi_2\) are \emph{nearly
equivalent} if \(\pi_{1,v}\simeq\pi_{2,v}\) for almost all places \(v\)
of \(F\).  For almost all \(v\) the representation \(M_{i,v}\) is
unramified; let \(\{c_{1,i,v},\cdots,c_{t_i,i,v}\}\) be its set of
Satake parameters.  Then, for every \(\pi\in\Pi_M\) and almost every
\(v\), the Satake parameter of \(\pi_v\) (relative to \(\psi_v\) if
\(G_{k}=\Mp_{2k}\)) is
\begin{equation}\label{sat}
    \bigcup_{i=1}^{r} \bigcup_{j=1}^{t_i}\{c_{j,i,v}\cdot q^{-\frac{d_i}{2}},c_{j,i,v}\cdot q^{-\frac{d_i}{2}+1},\cdots,c_{j,i,v}\cdot q^{\frac{d_i}{2}-1},c_{j,i,v}\cdot q^{\frac{d_i}{2}}\}.
\end{equation}
In particular all elements of \(\Pi_M\) are nearly equivalent, and by
the uniqueness of isobaric decompositions together with strong
multiplicity one \cite{JS0}, the parameter \(M\) is determined by the
collection \eqref{sat}.
\end{rem}

For a discrete global \(A\)-parameter \(Y\) for a classical group
\(G_\ell\), let \(\Pi_{\phi_Y,G_\ell}^{\mathrm{aut}}\) denote the set of
irreducible discrete automorphic representations whose local components
belong to the local \(L\)-packets associated with \(\phi_{Y,v}\) at every
place \(v\). We call this set the \emph{associated global \(L\)-packet}
of \(Y\). When the group is clear, for both symplectic and metaplectic
groups we simply write \(\Pi_{\phi_Y}^{\mathrm{aut}}\).

At almost every place, the local \(L\)-packet \(\Pi_{\phi_{Y,v}}\)
contains a unique unramified member, and by \eqref{assocL} its Satake
parameter is the one displayed in \eqref{sat}. Hence
\[
\Pi_{\phi_Y}^{\mathrm{aut}}\subset\Pi_Y,
\]
and in particular any two members of an associated global \(L\)-packet
are nearly equivalent. If \(Y\) is tempered, then \(\phi_Y=Y\), the local
\(A\)-packet and the local \(L\)-packet attached to \(Y_v\) coincide at
every place, and consequently
\(
\Pi_{\phi_Y}^{\mathrm{aut}}=\Pi_Y;
\)
the reverse inclusion here is the assertion that the local components of
a discrete automorphic member of \(\Pi_Y\) lie in the corresponding local
\(A\)-packets, which is \cite[Theorem~1.5.2]{Ae} in the symplectic case
and \cite[Corollary~5.4.4(i)]{Li} in the metaplectic case.
It is the associated global \(L\)-packet, and not the global
\(A\)-packet, that we use throughout: for a tempered parameter the two
notions agree, while for a non-tempered parameter it is the former that
enters the Gan--Gross--Prasad conjecture.

\subsection{Representations with tempered global $A$-parameter}
\label{temp}

Throughout the paper we work with the members of the associated global
$L$-packet of a tempered global $A$-parameter, rather than with tempered
automorphic representations. The properties collected in this subsection
are the only ones we shall use, and they are what makes the former class
the convenient one.

We call an irreducible admissible representation \emph{almost tempered}
if every real exponent occurring in its canonical Langlands datum lies in
the open interval \((-\frac12,\frac12)\). This is the only meaning of
``almost tempered'' used in this paper.

\begin{lem}
\label{automatic-almost-tempered}
Let \(Y\) be a discrete tempered global \(A\)-parameter for a symplectic,
metaplectic, or special orthogonal group. The local components of every
irreducible automorphic representation in the discrete spectrum with
global \(A\)-parameter \(Y\) are unitary and almost tempered. The local
components of an irreducible unitary cuspidal automorphic representation
of a general linear group are also generic and almost tempered.
\end{lem}

\begin{proof}
Automorphic local factors are unitary. Let \(\sigma\) be an irreducible
unitary cuspidal automorphic representation of \(\GL_a(\A)\). Then
\(\sigma\) is globally generic, so that every \(\sigma_v\) is generic,
and the bound towards the generalized Ramanujan conjecture of Luo,
Rudnick and Sarnak \cite[Theorem~2]{LRS} places every real exponent of
\(\sigma_v\) strictly inside \((-\frac12,\frac12)\). Alternatively, the
same bound follows from the classification of the generic unitary duals
of general linear groups \cite{Tadic,Vogan}. This proves the second
assertion.

For the first assertion, write \(Y=\BIGboxplus_{i}Y_i\) and let \(\pi\)
be an irreducible automorphic representation in the discrete spectrum
with global \(A\)-parameter \(Y\). At every place \(v\), the local
component \(\pi_v\) lies in the local \(L\)-packet attached to
\(\phi_{Y,v}=\BIGboxplus_{i}Y_{i,v}\), and the real exponents occurring
in the canonical Langlands datum of \(\pi_v\) are among those of the
\(Y_{i,v}\). The preceding paragraph therefore bounds them by
\(\frac12\) in absolute value. This is the argument of
\cite[Lemma~B.8]{At0}, which applies verbatim to symplectic and
metaplectic groups by \cite{Ae,gi,Li}, and to special orthogonal groups
by \cite{Ae,Ish}.
\end{proof}

The next lemma is the reason why a tempered global \(A\)-parameter never
produces a residual representation; it is used repeatedly below.

\begin{lem}
\label{tempered-parameter-cuspidal}
Let \(Y\) be a discrete tempered global \(A\)-parameter for
\(\Sp_{2k}\), \(\Mp_{2k}\), or an odd special orthogonal group
\(\SO(V_{2k+1})\). Every irreducible automorphic representation in the
discrete spectrum with global \(A\)-parameter \(Y\) is cuspidal. In
particular, every member of the associated global \(L\)-packet
\(\Pi_{\phi_Y}^{\mathrm{aut}}=\Pi_Y\) is cuspidal.
\end{lem}

\begin{proof}
For symplectic groups and quasi-split odd special orthogonal groups,
the assertion follows from M{\oe}glin's description of the
noncuspidal discrete spectrum; see \cite[Section~1.3]{Moeglin08}.
The corresponding statement for pure inner forms of odd special
orthogonal groups follows from the inner-form classification in
\cite{Ish}. For metaplectic groups, it follows from the decomposition
of the genuine discrete spectrum by global Arthur parameters and the
multiplicity formula \cite[Theorems~5.1.1 and~5.4.1]{Li}. Indeed, a
noncuspidal discrete member necessarily has a nontrivial Arthur
\(\SL_2\)-block \(\tau\boxtimes[d]\) with \(d\ge1\), contradicting
the temperedness of \(Y\).
\end{proof}

\begin{prop}\label{3}
Let \(M\) and \(M'\) be discrete tempered global \(A\)-parameters for
\(\Sp_{2k}\) and \(\Mp_{2k}\), respectively, and let
\(\pi\in\Pi_{\phi_{M}}^{\mathrm{aut}}\) and
\(\pi'\in\Pi_{\phi_{M'}}^{\mathrm{aut}}\).  Then the local normalized
intertwining operators attached to \(\pi\) and \(\pi'\) are holomorphic
and non-vanishing for all \(z\) with \(\Re(z)\ge\frac12\).
\end{prop}

\begin{proof}
By Lemma~\ref{automatic-almost-tempered} the local components of \(\pi\)
and \(\pi'\) are unitary and almost tempered; at an unramified place
this says that the inducing characters of the corresponding principal
series are of the form \(\nu_i|\cdot|^{s_i}\) with \(\nu_i\) unramified
unitary and \(0\le s_i<\frac12\).  See also the description of the
automorphic discrete spectrum of symplectic groups \cite{Ae} and of
metaplectic groups \cite{gi,Li}.  The assertion is then a consequence of
\cite{Zhang}.
\end{proof}

\begin{rem}
Proposition~\ref{3} and
Lemmas~\ref{automatic-almost-tempered}--\ref{tempered-parameter-cuspidal}
make precise the sense in which our results are statements about
$A$-parameters rather than about representations.  They are also where
the two members of our pairs part company: by \cite[Theorem~11.2]{GRS} a
globally generic cuspidal automorphic representation of $\Mp_{2k}(\A)$
need not have a tempered global $A$-parameter, whereas a globally generic
cuspidal automorphic representation of $\Sp_{2k}(\A)$ always does.
\end{rem}

\section{\textbf{Jacquet module corresponding to Fourier-Jacobi characters}}\label{sec:jacquet}
In this section, we introduce the Jacquet module corresponding to a Fourier-Jacobi character and extract some properties of it for certain cases. This play a key role in proving Proposition~\ref{l1}, which asserts the vanishing of the Fourier-Jacobi periods when it involves a residual representation.

Fix a finite place $v$ of $F$ and a nontrivial additive character $\psi_v$ of $F_v$. Sometimes, we regard $\psi_v$ as a character of $\U_{k,k}(F_v)$ defined by $\psi_v(u)=\psi(\left<uf_1,f_k\right>)$. For any $\alpha \in F_v$, define $\psi_{\alpha,v}:F_v \mapsto \CC^{\times}$ as
\(\psi_{\alpha,v}(x)\coloneqq\psi_v(\alpha x), \quad \text{for } x\in F_v.
\)
Let $\Omega_{\alpha, W_{k-j}, v}$ represent the Weil representation of $\mathcal{H}_{W_{k-j}}(F_v) \rtimes \Mp_{2k-2j}(F_v)$ with respect to $\psi_{\alpha,v}^{-1}$. It is worth noting that $\Omega_{\alpha, W_{0}, v} = \psi_{\alpha,v}^{-1}.$

As this section is dedicated to the local context, we will henceforth omit the subscript $v$ and the field $F_v$ from our notation.

Since $\mathcal{N}_j \ss N_{k,j-1}$ and $N_{k,j-1}\backslash N_{k,j} \simeq \mathcal{H}_{W_{k-j}}$,
we can pull back  $\Omega_{\alpha,W_{k-j}}$ to $ N_{k,j}\rtimes \Mp_{2k-2j}$
and  denote it by the same symbol $\Omega_{\alpha,W_{k-j}}$.
When $j \ge 2$, we define a character $\lambda_{j}:\mathcal{N}_j \to \CC^{\times}$ by
\[
\lambda_{j}(n)=\psi( \left<ne_2,f_1\right>+\left<ne_3,f_2\right>+ \cdots + \left<ne_{j},f_{j-1}\right> ), \quad n \in \mathcal{N}_{j}.
\] When $j=0,1$, define $\lambda_{j}$ as the trivial character.
Put $\nu_{\alpha,W_{k-j}}=\lambda_{j}^{-1} \otimes \Omega_{\alpha,W_{k-j}}$
and denote $H_{k,j}=N_{k,j} \rtimes \Mp_{2k-2j}$.
%We can embed $H_{k,j}$ into $\Sp_{2k} \times \Mp_{{k-j}}$ by inclusion on the first factor and projection on the second factor.
Then $\nu_{\alpha,W_{k-j}}$ is a smooth representation of $H_{k,j}$.
%and up to conjugation of the normalizer of $H_{k,j}$ in $\Sp_{2k} \times \Mp_{{k-j}}$,
%it is uniquely determined by $\psi$ modulo $(F^{\times})^2$.
We shall denote by $\omega_{\alpha,W_{k-j}}$ the restriction of $\nu_{\alpha,W_{k-j}}$ to $\Mp_{2k-2j}$.
\par

For $0 \le l \le k-1$, we define characters $\psi_{l}$ of $N_{k,l+1}$,
which factors through the quotient map $n\colon N_{k,l+1} \to \U_{k,l+1} \backslash N_{k,l+1}\simeq \mathcal{N}_{l+1}$,
by setting
\(
\psi_{l}(u)=\lambda_{l+1}(n(u)).
\)
For an arbitrary $\alpha \in F^{\times}$, we also define a character $\psi_{l}^{(\alpha)}$ of $N_{k,l+1}$ by $\psi_{l}^{(\alpha)}(u)= \psi_{l}(u)\cdot \psi(\alpha\left<ue_{l+2},f_{l+1}\right>)$. (Here,we regard $e_{k+1}=f_1$ when $l=k-1$.)

For an irreducible genuine (resp. non-genuine) smooth representation $\pi'$ of $\Mp_{2k}$, write  $J_{\psi^{(\alpha)}}(\pi')$ for the Jacquet module of $\pi'$
with respect to the group $N_{k,k}$ and its character $\psi_{k-1}^{(\alpha)}$. Then it is a genuine (resp. non-genuine) representation of the group $\wt{\P}_{k}$. We say that $\pi'$ is $\psi^{(\alpha)}$-generic if
$J_{\psi^{(\alpha)}}(\pi')$ is non-zero. If $\pi'$ is $\psi^{(\alpha)}$-generic for some $\alpha \in F^{\times}$, we say that $\pi'$ is generic.
Thus a non-generic representation is $\psi^{(\alpha)}$-non-generic for
every $\alpha \in F^{\times}$, whereas the converse fails in general.  We
keep the two notions apart because Propositions~\ref{key}
and~\ref{key3} below require only $\psi^{(1)}$-non-genericity of their
inducing data, while every representation to which we apply them is in
fact non-generic.

We also write $J_{\psi_l}(\pi' \otimes \Omega_{\alpha,W_{k-l-1},v})$ for the Jacquet module of $\pi' \otimes \Omega_{\alpha,W_{k-l-1}}$ with respect to the group $N_{k,l+1}$ and its character $\psi_{l}$. Then it is a non-genuine (resp. genuine) representation of the group $\Mp_{2k-2l-2}$. We can write this as a two-step Jacquet module as follows (see \cite[Chapter 6]{GRS}.)

For $1\le l \le k$, let $C_{l-1}$ be the subgroup of $N_{k,l}$ corresponding to the center of the Heisenberg group $\mathcal{H}_{W_{k-l}}$ through the isomorphism
$ N_{k,l-1}\backslash N_{k,l} \simeq \mathcal{H}_{W_{k-l}}$. Note that $C_{l-1}(F) \simeq F$ and $C_{l-1}$ acts on $\Omega_{\alpha,W_{k-l}}$ by the character $\psi_{\alpha}^{-1}$. For $1 \le l \le k-1$, put $N_{k,l+1}^0=N_{k,l}C_{l} \ss N_{k,l+1}$ and define a character $\psi_{\alpha,l}^0:N_{k,l+1}^0 \to \CC^{\times}$ as $\psi_{\alpha,l}^0(nc)=\psi_{l-1}(n)\psi_{\alpha}(c)$ for $n \in N_{k,l}, \ c \in C_{l}$.

Let $J_{\psi_{\alpha,l}^0}(\pi')$ be the Jacquet module of $\pi'$ with respect to $N_{k,l+1}^0$ and $\psi_{\alpha,l}^0$. We also denote by $J_{\mc{H}_{W_{k-l-1}}/C_l}$ the Jacquet functor with respect to $W_{k-l-1}/C_l$ and the trivial character.

 Then we have
\beq \label{two}
J_{\psi_{l}}(\pi' \otimes \Omega_{\alpha,W_{k-l-1}})\simeq J_{\mc{H}_{W_{k-l-1}}/C_l}(J_{\psi_{\alpha,l}^0}(\pi')\otimes \Omega_{\alpha,W_{k-l-1}}).
\eeq
When $\alpha=1$, we simply write $\psi_{\alpha}, \ J_{\psi_{\alpha,l}^0}$, $\nu_{\alpha,W_{k-j}}$,  $\omega_{\alpha,W_{k-j}}$ and $ \Omega_{\alpha,W_{k-l-1}}$ as $\psi, \ J_{\psi_{l}^0}$, $\nu_{W_{k-j}}$, $\omega_{W_{k-j}}$ and $ \Omega_{W_{k-l-1}}$, respectively.

We can define a genuine character of $\GL_1$ by
\( \gamma_{\psi}\left(\beta,\zeta\right)=\zeta \gamma(\beta,\psi)^{-1}
\)
with
\(\gamma(\beta,\psi)=\gamma(\psi_{\beta})/\gamma(\psi)
\)
for $\beta \in F^{\times}$ and $\xi \in \{\pm1\}$, where $\gamma(\psi) \in \mu_8$ is a Weil constant associated with $\psi$ (see \cite{R}). Using this character, we obtain a bijection, which depends on the choice of the additive character $\psi$, between the set of equivalence classes of admissible representations of $\GL_j$ and that of genuine admissible representations of $\wt{\GL}_{j}$ via $\sigma \mapsto \sigma_{\psi}$, where
\(\sigma_{\psi} \left( g,\zeta \right) = \gamma_{\psi} \left( \det{g}, \zeta \right) \cdot\sigma(g).
\)
We also write
\[
\varepsilon_{-1}(g)=(\det g,-1)_F
\qquad (g\in\GL_j(F)).
\]
We also denote by $\mc{P}_{k,j}$ the  subgroup of $\GL_k$ consisting of $\left\{ \begin{pmatrix} \GL_{k-j} & x \\0 & \mc{N}_j\end{pmatrix}\right\}_{x \in \text{Mat}_{(k-j) \times j}}$. We write $\mc{N}_{j}$ in $\mc{P}_{k,j}$ as $\mc{N}_{k,j}$ (i.e., $\mc{N}_{k,j}=\mc{P}_{k,j} \cap (\text{Id}_{\GL_{k-j}}\times \GL_{j})$.)

For a smooth representation $\sigma$ of $\GL_j$ and a genuine smooth representation $\pi''$ of $\Mp_{2n-2j}$, the notation $\sigma\rtimes\pi''$ therefore means $\Ind_{\wt\P_{n,j}}^{\Mp_{2n}}(\sigma_\psi\boxtimes\pi'')$. The usual properties of parabolic induction and Jacquet functors in this setting are recalled in \cite{HM}.
\begin{lem} \label{Jac}
Let $n,m$ be positive integers such that $r=n-m > 0$. For a positive integer $a>0$, let $\mathcal{E}$ be a smooth representation of $\Mp_{2n+2a}$ of finite length and $\mc{I}$ a smooth representation of $\Mp_{2m+2a}$ of finite length. Assume that exactly one of $\mc{E}$ or $\mc{I}$ is genuine. Then \begin{align*}& \Hom_{N_{n+a,r} \rtimes \Sp_{2m+2a}}
\left(\mathcal{E} \otimes  \nu_{W_{m+a}}
\otimes \mc{I},
\CC \right)
\\& \simeq  \Hom_{\Sp_{2m+2a}} \left(J_{\psi_{r-1}}( \mathcal{E}\otimes \Omega_{W_{m+a}})
 \otimes \mc{I}  ,\CC \right).\end{align*}

\end{lem}

\begin{proof} By the Frobenius reciprocity,
 \begin{align*}
 & \Hom_{N_{n+a,r} \rtimes \Sp_{2m+2a}}
\left(\mathcal{E} \otimes  \nu_{W_{m+a}}
\otimes \mc{I},
\CC \right)
\\&=\Hom_{N_{n+a,r} \rtimes \Sp_{2m+2a}}
\left(\mathcal{E} \otimes  \lambda_r \otimes \Omega_{W_{m+a}}
\otimes \mc{I} ,
\CC \right)
\\& \simeq \Hom_{\Sp_{2m+2a}} \left(
J_{\psi_{r-1}}(\mathcal{E} \otimes \Omega_{W_{m+a}})
 \otimes \mc{I},\CC \right).
\end{align*}
\end{proof}

For each $2 \le t \le k$, we define characters $\lambda_{t}':\mathcal{N}_k \to \CC^{\times}$ and $\lambda_{t}'':\mathcal{N}_k \to \CC^{\times}$ as follows:
\begin{align*}
\lambda_{t}'(n)=\psi^{-1}( \left<ne_{k-t+2},f_{k-t+1}\right>+\left<ne_{k-t+3},f_{k-t+2}\right>+ \cdots + \left<ne_{k},f_{k-1}\right> ), \quad n \in \mathcal{N}_{k},\\
\lambda_{t}''(n)=\psi^{-1}( \left<ne_{k-t+2},f_{k-t+1}\right>+\left<ne_{k-t+3},f_{k-t+2}\right>+ \cdots + \left<ne_{k},f_{k-1}\right> + \left<ne_{k-t+1},f_{k-t} \right>), \quad n \in \mathcal{N}_{k}.
\end{align*}
When $t = 0$ or $t = 1$, we define $\lambda_{t}'$ and $\lambda_{t}''$ as the trivial characters.
For a smooth representation $\sigma$ of $\GL_k$ and $0 \le i \le k$, let $\sigma^{(i)}$ (reps. $\sigma_{(i)}$) be the smooth representation of $\GL_{k-i}$ (resp. $\mc{P}_{k-i-1,1}$) defined by the Jacquet module of $\sigma$ with respect to the character $\lambda_t'$ (resp. $\lambda_t''$) of $\mc{N}_k$ as in \cite[p.87]{GRS}. Note that $\sigma^{(i)}$ is nothing but the $i$-th Bernstein-Zelevinsky derivative of $\sigma$ if we use the unnormalized Jacquet functors in the definition of Bernstein-Zelevinsky derivative as in \cite{BZ}. Note that the Bernstein-Zelevinski derivative of $\chi(k)$ is given by
\[
\left( \chi (k) \right)^{(t)}
=
\begin{cases}
\chi(k-t)  &\text{ if } t \le 1, \\
0  & \text{ if } t \ge 2.
\end{cases}
\]

\begin{rem} \label{der} For characters $\mu_i$ of $\GL_1$, let $\sigma=\mu_1(m_1)\times\cdots\times\mu_t(m_t)$. We use the notation $A\stackrel{\text{s.s}}{\simeq}B$ to mean that $A$ and $B$ are isomorphic up to semi-simplication.

Using the Leibniz rule of the Bernstein-Zelevinsky derivative, it is not so difficult to check that for every integer $j \ge t+1$, $|\cdot|^{\frac{1-j}{2}}\cdot \sigma^{(j)} =0$ and for $0 \le j \le t$,
\begin{align}\label{der1}|\cdot|^{\frac{1-j}{2}}\cdot \sigma^{(j)} \stackrel{\text{s.s}}{\simeq}\bigoplus_{\substack{ i_1+\cdots + i_t=j \\ 0 \le i_1, \cdots, i_t \le 1 }} \text{Ind}_{\mc{P}_{(m_1-i_1,\cdots,m_t-i_t)}}^{\GL_{k-j}} (|\cdot|^{\frac{1-i_1}{2}}\cdot \mu_1(m_1-i_1) \boxtimes \cdots \boxtimes|\cdot|^{\frac{1-i_t}{2}}\cdot  \mu_t(m_t-i_t))
\end{align} (cf. \cite[(5.34)]{GRS}.)
By \cite[p. 452]{BZ0}, $\sigma_{(j)}$ has a finite $\mc{P}_{k-j-1,1}$-filtration whose subquotients are (compactly) induced from derivatives of $\sigma^{(j)}$. Note that for all $e \ge 1$, $(\sigma_{(j)})^{(e)}\simeq \sigma^{(e+j)}$ as a restriction of a $\mc{P}_{k-j-1,1}$-module to $\GL_{k-(e+j)}$. In particular, if $j \ge t$, then all  subquotients of $\sigma_{(j)}$ are zero and it leads to $\sigma_{(j)}=0$.

%When $j=t-1$, we know that there is only one non-zero  subquotient of $\tau_{(t-1)}$, that is, $(\tau_{(t-1)})^{(1)}$. Therefore, $\tau_{(t-1)} \simeq \tau^{(t)}$ as a $\GL_{k-t}-$module and trivial on $\mc{N}_{k-t-1,1}$ since the conjugation of $\mc{P}_{k-t-1,1}$  stabilizes $\mc{N}_{k-t-1,1}$.
\end{rem}

For two non-negative integers $t,s$, put $\delta_{t,s}=\begin{cases}1 \text{ if } t \ge s \\ 0 \text{ if } t<s. \end{cases}$ Choose an arbitrary $\alpha \in F$. The following theorem is a slight generalization of \cite[Theorem 6.1]{GRS}, reflecting $\alpha$, whose proof is almost the same.

\begin{thm}[{\cite[Theorem~6.1]{GRS}}] \label{ind1} Let $\sigma$ be a smooth representation of $\GL_j$ and $\wt{\tau}$ a smooth representation of $\Mp_{2n}$. Choose an arbitrary non-negative integer $l$. %If $0 \le l \le j-1$, we further assume that $\tau_{(l)}\begin{pmatrix} d & x \\0 & \mathrm{Id}_{l+1} \end{pmatrix}=\tau_{(l)}\begin{pmatrix} d &  \\0 & \mathrm{Id}_{l+1} \end{pmatrix}.$
Then
\begin{align*}
J_{\psi_{\alpha,l}^0}(\sigma\rtimes\wt\tau)
&\stackrel{\mathrm{s.s}}{\simeq}
\bigoplus_{\substack{l-n<t\le l\\0\le t\le j}}
\ind_{\wt{\P}_{j-t}'}^{\Mp_{2n+2j-(2l+2)}\ltimes\mc H_{W_{n+j-(l+1)}}}
\left(|\cdot|^{\frac{1-t}{2}}\sigma^{(t)}\boxtimes
J_{\psi_{\alpha,l-t}^0}(\wt\tau)\right)\\
&\quad\oplus\delta_{j-1,l}\,
J_{C_l,\psi_\alpha}\left(
\ind_{\wt{\P}_{j-l}''}^{\Mp_{2n+2j-(2l+2)}\ltimes\mc H_{W_{n+j-(l+1)}}}
\left(|\cdot|^{-\frac l2}\sigma_{(l)}\boxtimes\wt\tau\right)
\right).
\end{align*}
\end{thm}
\noindent (For the definitions of $\P_{j-t}'$ and $\P_{j-l}''$, we refer the reader to \cite[(6.4)]{GRS} and \cite[(6.8)]{GRS}.)

Using (\ref{two}), we obtain the following two theorems by combining special cases of Theorem~\ref{ind1} with \cite[Proposition 6.6]{GRS} and \cite[Proposition 6.7]{GRS}. (Here, we reflected some typos in \cite[Proposition 6.7]{GRS}.)

\begin{thm}\label{ind2}
Let $\sigma$ be a smooth representation of $\GL_j$ and $\wt\tau$ a smooth representation of $\Mp_{2n}$. For an integer $l\ge0$, assume either $j\le l$ or $\sigma_{(l)}=0$. Then
\[
J_{\psi_l}\left((\sigma\rtimes\wt\tau)\otimes
\Omega_{\alpha,W_{n+j-(l+1)}}\right)
\stackrel{\mathrm{s.s}}{\simeq}
\bigoplus_{\substack{l-n<t\le l\\0\le t\le j}}
|\cdot|^{\frac{1-t}{2}}\sigma^{(t)}\rtimes
J_{\psi_{l-t}}\left(\wt\tau\otimes
\Omega_{\alpha,W_{n+t-(l+1)}}\right).
\]
\end{thm}

\begin{thm}\label{ind3}
For an integer $j\ge1$ and a character $\mu$ of $\GL_1$, let $\sigma=\mu(j)$, and let $\wt\tau$ be a smooth representation of $\Mp_{2n}$. If $0\le l\le n+j-2$, then
\[
\begin{split}
J_{\psi_l}\left((\sigma\rtimes\wt\tau)\otimes
\Omega_{W_{n+j-(l+1)}}\right)
&\stackrel{\mathrm{s.s}}{\simeq}
\bigoplus_{\substack{l+1-n\le t\le l\\0\le t\le1}}
|\cdot|^{\frac{1-t}{2}}\sigma^{(t)}\rtimes
J_{\psi_{l-t}}\left(\wt\tau\otimes
\Omega_{W_{n+t-(l+1)}}\right)\\
&\quad\oplus\delta_{0,l}\,
\sigma^{(1)}\rtimes(\wt\tau\otimes\omega_{W_n}).
\end{split}
\]
\end{thm}

The same argument, with the linear and metaplectic members interchanged,
gives the following companion formula.  We record it because it will be
used explicitly in Proposition~\ref{key3}.

\begin{thm}\label{ind3op}
For an integer \(j\ge1\) and a character \(\mu\) of \(\GL_1\), let
\(\sigma=\mu(j)\), and let \(\tau\) be a smooth representation of
\(\Sp_{2n}\). If \(0\le l\le n+j-2\), then
\[
\begin{split}
J_{\psi_l}\left((\sigma\rtimes\tau)\otimes
\Omega_{W_{n+j-(l+1)}}\right)
&\stackrel{\mathrm{s.s.}}{\simeq}
\bigoplus_{\substack{l+1-n\le t\le l\\0\le t\le1}}
|\cdot|^{\frac{1-t}{2}}\sigma^{(t)}\rtimes
J_{\psi_{l-t}}\left(\tau\otimes
\Omega_{W_{n+t-(l+1)}}\right)\\
&\quad\oplus\delta_{0,l}\,
\sigma^{(1)}\rtimes(\tau\otimes\omega_{W_n}).
\end{split}
\]
\end{thm}

\begin{proof}
The proof is the same mirabolic-filtration argument as for
Theorem~\ref{ind3}.  Indeed, that argument is independent of the
genuineness parity of the input.  For a non-genuine inducing datum,
tensoring with the Weil representation, defined with respect to
\(\psi^{-1}\), changes the parity of every remaining Levi factor, and
its restriction identifies a linear general-linear datum \(\rho\)
with \(\rho_\psi\).
The Bernstein--Zelevinsky derivatives, modulus characters, and the
exceptional \(l=0\) term are otherwise unchanged, giving the displayed
formula.
\end{proof}

By applying Theorems~\ref{ind3} and~\ref{ind3op}, we establish two
vanishing propositions, Proposition~\ref{key} and
Proposition~\ref{key3}, both of which will be used in the subsequent
discussion.  These propositions do not appear in \cite{H}.
To set the stage for these, we first introduce some relevant notations and prove a lemma.

For a character $\chi$ of $\GL_1$, we denote by $e(\chi)$ a real number such that the product $\chi \cdot |\cdot|^{-e(\chi)}$ forms a unitary character.
For \(1\le a\le n\), let \(J_{\U_{n,a}}\) denote the
unnormalized Jacquet functor with respect to \(\U_{n,a}\) and the
trivial character.

\noindent Note that there is an exact sequence
\begin{align} \label{arrow}
0 \longrightarrow \cind_{\wt{Q}_{n,1}}^{\wt{\P}_{n,n}}((|\cdot|^{\frac{1}{2}})_{\psi}\boxtimes \psi^{-1}) {\longrightarrow} \omega_{W_n}\vert_{\wt{\P}_{n,n}}{\longrightarrow} (|\cdot|^{\frac{1}{2}})_{\psi} \longrightarrow 0
\end{align}
(see the proof of \cite[Theorem 16.1]{Gan2}.)

\begin{lem}\label{kl}
Assume $k\ge2$.  Let $\mu$ and
$\{\chi_i\}_{1\le i\le k}$ be unramified characters of $\GL_1$ such that
\(
-\frac12<e(\mu)\le0.
\)
For $z\in\CC$, put $\mu_z=\mu|\cdot|^z$, and let $\chi_{i,z}$ be either
$\chi_i$ or $\chi_i|\cdot|^z$.  We assume
\(
0\le e(\chi_i)<\frac12
\quad\text{whenever the family $\chi_{i,z}$ is static, i.e. }\chi_{i,z}=\chi_i.
\)
No restriction on $e(\chi_i)$ is needed when
$\chi_{i,z}=\chi_i|\cdot|^z$.  Assume in addition that at least one
family $\chi_{i,z}$ is moving.  Thus, together with $\mu_z$, at least
two inducing characters depend on $z$.  Define
\(
\pi_{k+1,z}
=
\chi_{1,z}\times\cdots\times\chi_{k,z}\times\mu_z.
\)
Let \(\mc E\) be a principal series representation of
\(\Sp_{2(k-2)}\).  Let \(\mc E_{k+1}\) be an irreducible
\(\psi^{(1)}\)-non-generic subquotient of
\(
\mu(3)\rtimes\mc E.
\)
Then
\[
\Hom_{\Sp_{2k+2}}
\left(
\Ind_{\wt\P_{k+1,k+1}}^{\Mp_{2k+2}}
\bigl((\pi_{k+1,z})_\psi\bigr)
\otimes\omega_{W_{k+1}}\otimes\mc E_{k+1},
\CC
\right)=0
\]
for $z$ outside a discrete subset of $\CC$.

The analogous assertion also holds with the linear and metaplectic
members interchanged.  More precisely, if
\(\widetilde{\mc E}\) is a genuine principal series of
\(\Mp_{2(k-2)}\) and
\(\widetilde{\mc E}_{k+1}\) is an irreducible
\(\psi^{(1)}\)-non-generic
subquotient of
\[
\mu(3)\rtimes\widetilde{\mc E},
\]
then
\[
\Hom_{\Sp_{2k+2}}
\left(
\Ind_{\P_{k+1,k+1}}^{\Sp_{2k+2}}
(\pi_{k+1,z})
\otimes\omega_{W_{k+1}}
\otimes\widetilde{\mc E}_{k+1},
\CC
\right)=0
\]
outside a discrete subset of \(\CC\).
\end{lem}

\begin{rem}\label{kl-rem}
The condition that two inducing characters move with \(z\) is used
essentially.

If one allowed \(j=1\), the non-dual \(t=1\) subquotient in the geometric-lemma filtration would transfer \(\mu|\cdot|^{-1}\).  After the modulus twist its
exponent would be
\[
\frac{k+1}{2}+e(\mu),
\]
which equals the upper endpoint of \eqref{length3-range} when
\(e(\mu)=0\).  In the present setting the derivative must remove both
moving characters, so \(j\ge2\); the strict inequality
\(e(\mu)>-\frac12\) then excludes this endpoint.
\end{rem}

\begin{proof}
By Frobenius reciprocity, the asserted Hom-space is isomorphic to
\[
\Hom_{\P_{k+1,k+1}}
\left(
\delta_{\P_{k+1,k+1}}^{-\frac12}
(\pi_{k+1,z})_\psi\otimes\omega_{W_{k+1}}
\otimes\mc E_{k+1},
\CC
\right).
\]
Recall that
\[
\delta_{\P_{k+1,k+1}}=|\cdot|^{k+2},
\qquad
\delta_{\P_{k+1,k+1-j}}=\delta_{Q_{k+1,j}}=|\cdot|^{k+j+2}.
\]
We apply the exact sequence
\begin{equation}\label{length3-arrow}
0\longrightarrow
\cind_{\wt Q_{k+1,1}}^{\wt\P_{k+1,k+1}}
\bigl((|\cdot|^{1/2})_\psi\boxtimes\psi^{-1}\bigr)
\longrightarrow
\omega_{W_{k+1}}|_{\wt\P_{k+1,k+1}}
\longrightarrow
(|\cdot|^{1/2})_\psi
\longrightarrow0.
\end{equation}
For the compactly induced term, Frobenius reciprocity reduces the
problem to the restriction of \(\pi_{k+1,z}\) to the mirabolic subgroup
\(\mc P_{k+1,1}\). By the Bernstein--Zelevinsky mirabolic filtration
\cite[Section~3.5]{BZ}, this restriction has a finite filtration whose
\(j\)-th successive quotient is
\[
\cind_{\mc P_{k+1,j}}^{\mc P_{k+1,1}}
\left(
|\det|^{j/2}(\pi_{k+1,z})^{(j)}
\boxtimes
\lambda_{k+1}|_{\mc N_{k+1,j}}
\right),
\qquad 1\le j\le k+1.
\]
Applying Frobenius reciprocity to each successive quotient and using
the modulus-character identities above, the corresponding Hom-space
reduces to
\begin{equation}\label{length3-jac}
\Hom_{\GL_{k+1-j}}
\left(
|\cdot|^{-\frac{k+j+1}{2}}(\pi_{k+1,z})^{(j)}
\otimes
J_{\psi^{(1)}_{j-1}}
\bigl(J_{\U_{k+1,k+1-j}}(\mc E_{k+1})\bigr),
\CC
\right).
\end{equation}
Here the restriction of \(\lambda_{k+1}^{-1}\) to the first
\(j-1\) simple-root coordinates, together with the rank-one character
\(\psi^{-1}\) in \eqref{length3-arrow} on the last coordinate, is
conjugate to \(\psi^{(1)}_{j-1}\).

Write
\[
J_{\U_{k+1,k+1-j}}(\mc E_{k+1})^{\mathrm{s.s.}}
\simeq
\bigoplus_{i\in I}A_i\boxtimes B_i
\]
as a $\GL_{k+1-j}\times\Sp_{2j}$-module, and let
\(
I_0=\{i\in I:J_{\psi^{(1)}_{j-1}}(B_i)\ne0\}.
\)
For $i\in I_0$, the representation $B_i$ is generic.  If $j=k+1$,
\eqref{length3-jac} is the full Whittaker Jacquet module of
$\mc E_{k+1}$ and hence vanishes.  We therefore assume
$1\le j\le k$.

Let $\Pi_z$ be an irreducible constituent of
$|\cdot|^{-(k+j+1)/2}(\pi_{k+1,z})^{(j)}$.  By the Leibniz rule for
Bernstein--Zelevinsky derivatives, $\Pi_z^\vee$ is a constituent of a
principal series whose inducing characters are selected from
\[
|\cdot|^{\frac{k+1}{2}}\chi_{1,z}^{-1},\ldots,
|\cdot|^{\frac{k+1}{2}}\chi_{k,z}^{-1},
|\cdot|^{\frac{k+1}{2}}\mu_z^{-1}.
\]
If one of the characters retained in $\Pi_z^\vee$ depends on $z$, its
cuspidal support cannot agree with the fixed support of $A_i$ except for
a discrete set of $z$.  Otherwise the derivative has removed both
moving inducing characters.  In particular, $j\ge2$, and every exponent
in the support of $\Pi_z^\vee$ lies in
\begin{equation}\label{length3-range}
\left(\frac{k}{2},\frac{k+1}{2}\right].
\end{equation}

We now apply the geometric lemma to
$J_{\U_{k+1,k+1-j}}(\mc E_{k+1})$.  Put
\(q=k+1-j\). The successive subquotients in the geometric-lemma filtration, which we refer to as cells, are first indexed are by
\(
\max\{0,3-j\}\le t\le\min\{q,3\},
\)
where $t$ is the number of coordinates of the $\GL_3$-block transferred
to the $\GL_{k+1-j}$-factor.

If $t=0$, then $j\ge3$ and the $\Sp_{2j}$-factor is a subquotient of
\[
\mu(3)\rtimes\sigma.
\]
The character $\mu(3)$ has zero Whittaker module.  By the heredity of
Whittaker models \cite{Rod}, the displayed induced representation has
zero Whittaker module.  Since the Whittaker functor is exact, none of its
irreducible subquotients is generic; hence no constituent from this cell
belongs to $I_0$.

Suppose $t\ge1$, and put $c_{k,j}=(k+j+2)/2$.  The normalized Jacquet
data of the character $\mu(3)$ must be read from the small-exponent end
of its segment.  More precisely, the \(t\)-cells are indexed by
\(t_1,t_2\ge0\) with \(t_1+t_2=t\), where \(t_1\) coordinates are
dualized and \(t_2\) coordinates retain their orientation.  The
transferred \(\GL_t\)-datum has supercuspidal support
\begin{equation}\label{length3-cells}
\left\{
\mu^{-1}|\cdot|^{1-r}:0\le r<t_1
\right\}
\ \cup\
\left\{
\mu|\cdot|^{t-2-s}:0\le s<t_2
\right\}.
\end{equation}
Thus \eqref{length3-cells} includes the mixed cells as well as the two
extreme orientations.  For example, when \(t=2\) and
\((t_1,t_2)=(1,1)\), it is
\(\{\mu^{-1}|\cdot|,\mu\}\).

After the modulus twist \( |\cdot|^{c_{k,j}}\), the support in
\eqref{length3-cells} contains a character of exponent strictly larger
than \((k+1)/2\).  Indeed, if \(t_1\ge1\), it contains
\(\mu^{-1}|\cdot|\), whose exponent after the twist is
\[
c_{k,j}+1-e(\mu)>\frac{k+1}{2}.
\]
If \(t_1=0\), the smallest exponent in the transferred support is
\[
c_{k,j}+e(\mu)-1
=\frac{k+j}{2}+e(\mu)
>\frac{k+1}{2},
\]
where the last strict inequality uses \(j\ge2\) and
\(e(\mu)>-\frac12\).
This is incompatible with \eqref{length3-range}.  Thus the complete
cuspidal supports of $A_i$ and $\Pi_z^\vee$ are different, and
\cite[Theorem~2.9]{BZ} gives
\(
\Hom_{\GL_{k+1-j}}(A_i,\Pi_z^\vee)=0
\)
outside a discrete set.  Hence every term arising from the left-hand
term of \eqref{length3-arrow} vanishes outside a discrete set.

The quotient in \eqref{length3-arrow} gives a Hom-space of the form
\[
\Hom_{\GL_{k+1}}
\left(
(|\cdot|^{-1}\delta_{\P_{k+1,k+1}})^{-1/2}\pi_{k+1,z}
\otimes J_{\U_{k+1,k+1}}(\mc E_{k+1}),
\CC
\right).
\]
The second factor is fixed and of finite length, whereas the central
character of $\pi_{k+1,z}$ varies nontrivially with $z$.  This term also
vanishes outside a discrete set, proving the linear assertion.

For the metaplectic assertion, use the same mirabolic filtration after
tensoring the two genuine factors.  All resulting Hom-spaces descend to
the corresponding linear Levi groups.  The metaplectic geometric lemma
\cite[Theorem~3.3 and Proposition~4.5]{HM} gives the same cells.  On
each dualized part of a transferred general-linear block, \(\rho\) is
replaced by \(\varepsilon_{-1}\rho^\vee\); the non-dual part is
unchanged.  The residual genuine metaplectic factor induced from the
untransferred part of
\((\mu(3))_\psi\) is unchanged.  Thus \(\varepsilon_{-1}\) occurs
only on the transferred general-linear factor.  Since it is quadratic
and unitary, it changes neither real exponents nor any of the support
comparisons above.  The preceding argument therefore proves the
metaplectic assertion as well.
\end{proof}

\begin{prop}\label{key}
Let \(n\ge2\).  Let $\mu$ and
$\{\chi_i\}_{1\le i\le n}$ be unramified characters of
$\GL_1$ with
\(
-\frac12<e(\mu)\le0.
\)
Let $0\le l\le n-3$.
For $z\in\CC$, put $\mu_z=\mu|\cdot|^z$, and let $\chi_{i,z}$ be either
$\chi_i$ or $\chi_i|\cdot|^z$.  Assume
\(
0\le e(\chi_i)<\frac12
\quad\text{for every static family }\chi_{i,z}=\chi_i;
\)
no exponent restriction is imposed on the base character of a moving
family $\chi_i|\cdot|^z$.  Assume that $\chi_{1,z}$ is moving.  We use
the recursive order
\(
\pi_{n+1,z}
=
\chi_{n,z}\times\cdots\times\chi_{2,z}
\times(\chi_{1,z}\times\mu_z),
\)
where $\times$ denotes normalized parabolic induction.  Thus
$\chi_{n,z},\ldots,\chi_{2,z}$ are the successive outer $\GL_1$-blocks,
while the two moving characters $\chi_{1,z}$ and $\mu_z$ remain in the
innermost principal series throughout the induction below.
Let $\mc E_{n-l}$ be an irreducible
$\psi^{(1)}$-non-generic subquotient of
\(
\mu(3)\rtimes\mc E,
\)
where $\mc E$ is a principal series representation of
$\Sp_{2(n-l-3)}$.  Then
\[
\Hom_{\Sp_{2n-2l}}
\left(
J_{\psi_l}
\left(
\Ind_{\wt\P_{n+1,n+1}}^{\Mp_{2n+2}}
\bigl((\pi_{n+1,z})_\psi\bigr)
\otimes\Omega_{W_{n-l}}
\right)
\otimes\mc E_{n-l},
\CC
\right)=0
\]
for $z$ outside a discrete subset of $\CC$.
\end{prop}

\begin{proof}
Denote the assertion by \(\mathsf V_n\), and argue by induction on
\(n\).  The range is empty when \(n=2\).  Assume \(\mathsf V_k\), and
fix \(0\le l_0\le k-2\).  Put
\[
\wt\tau_{t,z}
=
\Ind_{\wt\P_{t,t}}^{\Mp_{2t}}
\bigl((\pi_{t,z})_\psi\bigr).
\]
By normalized induction in stages, the recursive order in the statement
has \(\chi_{k+1,z}\) as the outer \(\GL_1\)-block.  Applying
Theorem~\ref{ind3} with this block gives, up to semisimplification,
\begin{align*}
J_{\psi_{l_0}}
(\wt\tau_{k+2,z}\otimes\Omega_{W_{k+1-l_0}})
\simeq{}&
\bigoplus_{0\le s\le\min\{1,l_0\}}
|\cdot|^{\frac{1-s}{2}}\chi_{k+1,z}(1-s)\rtimes
J_{\psi_{l_0-s}}
(\wt\tau_{k+1,z}\otimes\Omega_{W_{k+s-l_0}})
\\
&\oplus\delta_{0,l_0}
(\wt\tau_{k+1,z}\otimes\omega_{W_{k+1}}).
\end{align*}
When \(l_0=0\), the last summand has zero Hom against
\(\mc E_{k+1}\) by Lemma~\ref{kl}.  Notice that the nonempty range
forces \(k\ge2\), exactly as required there.

For the remaining summands, apply Frobenius reciprocity and write
\[
J_{\U_{k+1-l_0,1-s}}
\bigl(\mc E_{k+1-l_0}\bigr)^{\mathrm{s.s.}}
\simeq
\bigoplus_{i\in I}
\delta_{\P_{k+1-l_0,1-s}}^{1/2}X_i\boxtimes Y_i.
\]
It is enough to show, for every \(i\), that either
\begin{equation}\label{length3-last0}
\Hom_{\GL_{1-s}}
\left(
|\cdot|^{\frac{1-s}{2}}\chi_{k+1,z}(1-s)\otimes X_i,
\CC
\right)=0
\end{equation}
outside a discrete set, or that the corresponding Hom involving \(Y_i\)
vanishes by the induction hypothesis.

The geometric lemma has only \(t=0,1\), since the Jacquet parabolic has a
\(\GL_{1-s}\)-factor.  If \(s=1\), then
\(\U_{k+1-l_0,0}\) is trivial, so \(t=0\) and
\[
Y_i=\mc E_{k+1-l_0}.
\]
In particular, \(Y_i\) is \(\psi^{(1)}\)-non-generic, and
\(\mathsf V_k\) applies with \(l=l_0-1\).  If \(s=0\) and \(t=0\),
this cell can occur only when \(k-l_0\ge3\).  In that case \(Y_i\) is
an irreducible subquotient of
\[
\mu(3)\rtimes\mc E'.
\]
for a principal series \(\mc E'\).  It is non-generic, hence in
particular \(\psi^{(1)}\)-non-generic, by heredity \cite{Rod} and the
exactness of the Whittaker functor, since \(\mu(3)\) is a character of \(\GL_3\) with
zero Whittaker module.  Thus
\(\mathsf V_k\) applies with \(l=l_0\).

It remains to consider \(s=0\) and \(t=1\).  The normalized Jacquet
module of \(\mu(3)\) contributes
\[
X_i\simeq\mu|\cdot|^{-1}
\qquad\text{or}\qquad
X_i\simeq\mu^{-1}|\cdot|.
\]
When \(\chi_{k+1,z}\) is moving, \eqref{length3-last0} vanishes outside a
discrete set.  Suppose that it is static.  A nonzero
\eqref{length3-last0} would require
\[
e(X_i)=-\frac12-e(\chi_{k+1})
\in\left(-1,-\frac12\right].
\]
On the other hand, the two displayed possibilities for \(X_i\) have
exponents in
\[
\left(-\frac32,-1\right]
\qquad\text{and}\qquad
\left[1,\frac32\right),
\]
respectively.  Neither interval meets
\(\left(-1,-\frac12\right]\).  Hence the \(t=1\) cells vanish as well,
and there is no boundary term.  This proves \(\mathsf V_{k+1}\) and
completes the induction.
\end{proof}

The following proposition is the metaplectic analogue of Proposition~\ref{key}.

\begin{prop}\label{key3}
Let \(n\ge2\).  Let $\mu$ and
$\{\chi_i\}_{1\le i\le n}$ be unramified characters of
$\GL_1$ with
\(
-\frac12<e(\mu)\le0.
\)
Let $0\le l\le n-3$.
For $z\in\CC$, put $\mu_z=\mu|\cdot|^z$, and let $\chi_{i,z}$ be either
$\chi_i$ or $\chi_i|\cdot|^z$.  Assume
\(
0\le e(\chi_i)<\frac12
\quad\text{for every static family }\chi_{i,z}=\chi_i;
\)
no exponent restriction is imposed on the base character of a moving
family $\chi_i|\cdot|^z$.  Assume that $\chi_{1,z}$ is moving, and use
the recursive order
\(
\pi_{n+1,z}
=
\chi_{n,z}\times\cdots\times\chi_{2,z}
\times(\chi_{1,z}\times\mu_z).
\)
Here again $\times$ denotes normalized parabolic induction, with the
displayed recursive order.
Let $\mc E_{n-l}$ be an irreducible
$\psi^{(1)}$-non-generic subquotient of
\(
\mu(3)\rtimes\mc E,
\)
where $\mc E$ is a genuine principal series representation of
$\Mp_{2(n-l-3)}$.  Then
\[
\Hom_{\Sp_{2n-2l}}
\left(
J_{\psi_l}
\left(
\Ind_{\P_{n+1,n+1}}^{\Sp_{2n+2}}
(\pi_{n+1,z})
\otimes\Omega_{W_{n-l}}
\right)
\otimes\mc E_{n-l},
\CC
\right)=0
\]
for $z$ outside a discrete subset of $\CC$.
\end{prop}

\begin{proof}
Denote the assertion by \(\widetilde{\mathsf V}_n\), and argue by
induction on \(n\).  The range is empty when \(n=2\).  Assume
\(\widetilde{\mathsf V}_k\), and fix \(0\le l_0\le k-2\).
We use the same recursive order and the companion
Fourier--Jacquet formula in Theorem~\ref{ind3op}.  The base mirabolic
term for \(l_0=0\) vanishes by the metaplectic part of
Lemma~\ref{kl}.

For the other terms, the geometric lemma for metaplectic groups
\cite[Theorem~3.3 and Proposition~4.5]{HM} gives the same \(s\)- and
\(t\)-cells as in the proof of Proposition~\ref{key}.  After tensoring
the two genuine factors, the relevant Hom-spaces descend to the
corresponding linear Levi groups.  On each dualized part of the
transferred general-linear block, \(\rho\) is replaced by
\(\varepsilon_{-1}\rho^\vee\); the non-dual part is unchanged.  The
residual genuine metaplectic factor induced from the untransferred
coordinates of
\((\mu(3))_\psi\) is unchanged.  In particular,
\(\varepsilon_{-1}\) occurs only on the transferred general-linear
factor.  When \(s=1\), the group
\(\U_{k+1-l_0,0}\) is trivial, so the metaplectic factor is
\(\mc E_{k+1-l_0}\) itself and is
\(\psi^{(1)}\)-non-generic.  This cell is eliminated by
\(\widetilde{\mathsf V}_k\) with \(l=l_0-1\).  The \(s=0,t=0\) cell
can occur only when \(k-l_0\ge3\); heredity \cite{Rod} and exactness of
Whittaker models show that the relevant metaplectic subquotient induced from
\((\mu(3))_\psi\) remains non-generic, hence in particular
\(\psi^{(1)}\)-non-generic, so
\(\widetilde{\mathsf V}_k\) with \(l=l_0\) eliminates it.

It remains to treat \(s=0,t=1\).  The fixed quadratic character coming
from \(\gamma_\psi^2\) is precisely \(\varepsilon_{-1}\), and it does
not change real exponents.  The two transferred \(\GL_1\)-data are
\[
\mu|\cdot|^{-1}
\qquad\text{and}\qquad
\varepsilon_{-1}\mu^{-1}|\cdot|.
\]
As in the proof of Proposition~\ref{key}, a nonzero static term would
require the real exponent of the transferred datum to lie in
\(\left(-1,-\frac12\right]\).  The two displayed data instead have
exponents in
\[
\left(-\frac32,-1\right]
\qquad\text{and}\qquad
\left[1,\frac32\right),
\]
respectively.  Hence the static terms vanish, while a moving term
vanishes outside a discrete set of \(z\).  There is therefore no
boundary term, and the induction proves the proposition.
\end{proof}

%\begin{rem} Lemma \ref{Jac} can be used to control the size of the Hom space
%\[ \Hom_{N_{n+a,r} \rtimes \Sp_{2m+2a}}
%\left(\mathcal{E} \otimes  \nu_{\psi^{-1},W_{m+a}}
%\otimes \Ind_{\P_{m+a,a}}^{\Sp_{2m+2a}}(\sigma |\cdot|_E^s \boxtimes \pi) ,
%\CC \right)\]
%in terms of the Jacquet module of $\mathcal{E} \otimes  \nu_{\psi^{-1}, W_{m+a}}$. When $r>0$, the existence of $N_{n+a,r}$ in this Hom space becomes essential because without it the Hom space would have infinite dimension. This leads us to distinguish the cases of $r>0$ and $r=0$ separately.\end{rem}

%When $r=0$, by the Frobenius reciprocity, (\ref{vl}) is isomorphic to
%\begin{align*}
%&\Hom_{\Sp_{m+2a}}
%\left(
%\mathcal{E} \otimes  \Omega_{W_{m+2a}}
%\otimes \Ind_{\P_{m+2a,a}}^{\Sp_{m+2a}}(\sigma |\cdot|_E^s \boxtimes \pi) ,
%\CC \right)
%\\& \simeq \Hom_{\Sp_{m+2a}}
%\left(
%\mathcal{E} \otimes  \Omega_{W_{m+2a}}
%\otimes \Ind_{\P_{m+2a,a}}^{\Sp_{m+2a}}(\sigma |\cdot|_E^s \boxtimes \pi) ,
%\CC \right)
%\end{align*}

\section{\textbf{Residual representation}}\label{sec:residual}
In this section, we introduce the residual Eisenstein series representations and review some of their properties related to $L$-functions. Throughout the section we fix discrete tempered global $A$-parameters $M_{\pi}$ for $\Sp_{2k}$ and $M_{\pi'}$ for $\Mp_{2k}$, and we let
\[
\pi\in\Pi_{\phi_{M_{\pi}}}^{\mathrm{aut}},
\qquad
\pi'\in\Pi_{\phi_{M_{\pi'}}}^{\mathrm{aut}}
\]
be members of the corresponding associated global $L$-packets; both are cuspidal by Lemma~\ref{tempered-parameter-cuspidal}. Let $\sigma$ be an irreducible unitary cuspidal automorphic representation of $\GL_a(\A)$.

For $g\in \Mp_{2k+2a}(\A)$ (resp. $\Sp_{2k+2a}(\A)$),
write $g=m_gu_gk_g$ for $m_g\in \wt{\M}_{k+a,a}(\A)$ (resp. $m_g \in \M_{k+a,a}(\A)$), $\ u_g\in \U_{k+a,a}(\A), \ k_g\in \wt{K}_{k+a,a}$ (resp. $k_g \in K_{k+a,a}$).
Since $\wt{\M}_{k+a,a}\simeq \GL(X_a) \times_{\{\pm1\}} \Mp_{2k}$ (resp. $\M_{k+a,a}\simeq \GL(X_a) \times \Sp_{2k}$),
we decompose $m_g \simeq (m_1,\epsilon) \cdot (m_2,\epsilon)$ (resp. $m_g=m_1 \cdot m_2$),
where $m_1 \in \GL_a(\A) , m_2 \in \Sp_{2k}(\A), \epsilon \in \{\pm1\}.$
Define $d(g)=|\det{m_1}|_{\A}$.   We set $\sigma_{\psi}\left( g,\epsilon\right)=\epsilon\cdot \sigma(g) \prod_v \gamma(\det{g_v},\psi_v)^{-1}$ for $g \in \GL_a(\A)$ and $\epsilon \in \{\pm1\}.$ For every $a \in \A^{\times}$, the product $\prod_v \gamma(a_v,\psi_v)$ is well defined in view of \cite[Proposition A.11]{R}.
Note that $\sigma_{\psi} \boxtimes \pi'$ (resp. $\sigma \boxtimes \pi$) is a representation of $\wt{\M}_{k+a,a}(\A)$ (resp. $\M_{k+a,a}(\A)$). Let $\phi$ be an element in $\As_{\P_{k+a,a}}^{\sigma_{\psi} \boxtimes \pi'}(\Mp_{2k+2a}))$ (resp. $\As_{\P_{k+a,a}}^{\sigma \boxtimes \pi}(\Sp_{2k+2a})$). For any $z \in \CC$, write $\phi_z\coloneqq d^z \cdot \phi$. Then $\phi_z$ belongs to $\As_{\P_{k+a,a}}^{|\cdot|^z (\sigma_{\psi} \boxtimes \pi')}(\Mp_{2k+2a}))$ (resp. $\As_{\P_{k+a,a}}^{|\cdot|^z (\sigma \boxtimes \pi)}(\Sp_{2k+2a})$) for all $z\in \CC$ and call $\phi_z$ a flat section. Note that the restriction of a flat section $\phi_z$ to $\wt{K}_{k+a,a}$ (resp. $K_{k+a,a}$) is the restriction of $\phi$ to $\wt{K}_{k+a,a}$ (resp. $K_{k+a,a}$), which is independent of $z\in \CC$.  Furthermore, the Eisenstein series associated to $\phi_z$ is defined by
\[
E(g,\phi_z,0)=E(g,\phi,z)\coloneqq \sum_{\P_{k+a,a}(F) \bs \Sp_{2k+2a}(F)}\phi(\gamma g)d(\gamma g)^z.
\]
This series converges absolutely when $\Re(z)$ is sufficiently large and admits a meromorphic continuation to the whole complex plane. \cite[IV.1.8.]{Mo}

As is well known, the analytic properties of the Eisenstein series is governed by the global intertwining operator $M_z$. Here, we review the definition of $M_z$. Let $w_0$ be a Weyl element in $\Sp_{2k+2a}$ which takes $\U_{k+a,a}$ to its opposite $\U_{k+a,a}^{-}$. Then for an arbitrary $f \in \As_{\P_{k+a,a}}^{|\cdot|^z (\sigma_{\psi} \boxtimes \pi')}(\Mp_{2k+2a}))$ (resp. $\As_{\P_{k+a,a}}^{|\cdot|^z (\sigma \boxtimes \pi)}(\Sp_{2k+2a})$), $M_z(f)$ is defined by \[M_z(f)(\cdot)\coloneqq \int_{[\U_{k+a,a}]} f(w_0^{-1}u\cdot) du. \]

Then, $M_z : I(z,\sigma_{\psi} \boxtimes \pi') \to I(-z,w_0 (\sigma_{\psi} \boxtimes \pi'))$ (resp. $M_z:I(z,\sigma \boxtimes \pi) \to I(-z,w_0 (\sigma \boxtimes \pi))$) is a $\Mp_{2k+2a}(\A)$ (resp. $\Sp_{2k+2a}(\A)$)-invariant map. If $f=\boxtimes_v f_v$ is factorizable, then $M_z(f)=\prod_v M_{z,v}(f_v)$, where $M_{z,v}:\text{Ind}_{\wt{\P}(F_v)}^{\Mp_{2k+2a}(F_v)}\delta_{\P}^{z}\cdot (\sigma_{\psi,v} \boxtimes \pi_v') \mapsto \text{Ind}_{\wt{\P}(F_v)}^{\Mp_{2k+2a}(F_v)}\delta_{\P}^{-z} \cdot w_0(\sigma_{\psi,v} \boxtimes \pi_v')$ (resp. $M_{z,v} : \text{Ind}_{\P(F_v)}^{\Sp_{2k+2a}(F_v)}\delta_{\P}^{z}\cdot (\sigma_{\psi,v} \boxtimes \pi_v') \mapsto \text{Ind}_{\P(F_v)}^{\Sp_{2k+2a}(F_v)}\delta_{\P}^{-z} \cdot w_0(\sigma_{\psi,v} \boxtimes \pi_v')$) is the local intertwining operator. (Note that $M_{z,v}$ is associated to $\sigma_v,\pi_v' \ (\text{resp. } \pi_v),\psi_v \text{ and } w_0$.)

Define the local normalizing operator $N_{z,v}$ by multiplying $M_{z,v}$ with the normalzing factor
\[\alpha_v\coloneqq \frac{L_v(z+1,\sigma_{v} \times \pi_{v}')\cdot L_v(2z+1,\sigma_{v}, Sym^2)\cdot \epsilon_v(z,\sigma_{v} \times \pi_{v}',\psi_v)\cdot \epsilon_v(2z,\sigma_{v},Sym^2,\psi_v) }{L_v(z,\sigma_{v} \times \pi_{v}')\cdot L_v(2z,\sigma_{v},Sym^2)}\]
\[(\text{resp.} \ \alpha_v\coloneqq \frac{L_v(z+1,\sigma_{v} \times \pi_{v})\cdot L_v(2z+1,\sigma_{v}, \wedge^2)\cdot \epsilon_v(z,\sigma_{v} \times \pi_{v},\psi_v)\cdot \epsilon_v(2z,\sigma_{v},\wedge^2,\psi_v) }{L_v(z,\sigma_{v} \times \pi_{v})\cdot L_v(2z,\sigma_{v},\wedge^2)})\]
where the local $L$-factors and $\epsilon$-factors are defined through the localization of the global $A$-parameters of $\pi'$ and $\pi$ as described in \cite{Ae} and \cite{gi}. Then
\begin{equation}\label{norm1}
    M_z(f)=\frac{L_{\psi}(z,\sigma \times \pi')\cdot L(2z,\sigma,Sym^2)}{L_{\psi}(z+1,\sigma\times \pi')\cdot L(2z+1,\sigma, Sym^2)\cdot \epsilon(z,\sigma \times \pi',\psi)\cdot \epsilon(2z,\sigma,Sym^2,\psi) }\cdot \prod_v N_{z,v}(f_v)
\end{equation}
\begin{equation}\label{norm2}
    \big(\text{resp.} \ M_z(f)=\frac{L(z,\sigma \times \pi)\cdot L(2z,\sigma,\wedge^2)}{L(z+1,\sigma \times \pi)\cdot L(2z+1,\sigma, \wedge^2)\cdot \epsilon(z,\sigma \times \pi,\psi)\cdot \epsilon(2z,\sigma,\wedge^2,\psi) }\cdot \prod_v N_{z,v}(f_v)\big),
\end{equation}
where $L_{\psi}(z,\sigma \times \pi')$ \big(resp. $L(z,\sigma \times \pi)$\big) is the Rankin-Selberg $L$-function $L(z,\sigma\times M_{\pi'})$ \big(resp. $L(z,\sigma\times M_{\pi})$\big).

The following is Proposition~\ref{3}, which is an analog of \cite[Theorem~5.1]{JZ} and can be proved in the same way.

\begin{prop}[cf. {\cite[Theorem 4.1]{JZ1}}]\label{p1}Let $\pi'\in\Pi_{\phi_{M_{\pi'}}}^{\mathrm{aut}}$ (resp. $\pi\in\Pi_{\phi_{M_{\pi}}}^{\mathrm{aut}}$) and $\sigma$ be as above.
Choose an arbitrary factorizable $\phi=\otimes_v \phi_v \in \As_{\P_{k+a,a}}^{\sigma_{\psi} \boxtimes \pi'}(\Mp_{2k+2a}))$ (resp. $\As_{\P_{k+a,a}}^{\sigma \boxtimes \pi}(\Sp_{2k+2a})$). Then for each place $v$ of $F$, $N_{z,v}(\phi_{z,v})$ is   holomorphic and nonzero for $z\ge \frac{1}{2}$.
\end{prop}

%the above local normaized intertwining operator $N_{z,v}$ associated to $\sigma_{\psi} \times \pi'$ (resp. $\sigma \times \pi$) and $\psi$ is holomorphic and nonzero for $z\ge \frac{1}{2}$. That means, for an arbitrary $\phi \in \As_{\P_{k+a,a}}^{\sigma_{\psi} \boxtimes \pi'}(\Mp_{2k+2a}))$ (resp. $\As_{\P_{k+a,a}}^{\sigma \boxtimes \pi}(\Sp_{2k+2a})$), $N_{z,v}(\phi_z)$ is   holomorphic and nonzero for $z\ge \frac{1}{2}$.

The properties recorded in Proposition~\ref{3} and
Lemma~\ref{automatic-almost-tempered}, together with the shape of a
discrete tempered global $A$-parameter, play essential roles in the
following propositions.

\begin{prop} \label{p2}
Let \(M_{\pi}\) be a discrete tempered global \(A\)-parameter for
\(\Sp_{2k}\), let \(\pi\in\Pi_{\phi_{M_{\pi}}}^{\mathrm{aut}}\), and let
\(\sigma\) be an irreducible unitary cuspidal automorphic representation
of \(\GL_a(\A)\). For
\[
\phi\in
\As_{\P_{k+a,a}}^{\sigma\boxtimes\pi}(\Sp_{2k+2a}),
\]
the Eisenstein series \(E(\phi,z)\) has at most a simple pole at
\(z=\frac12\) and \(z=1\). As \(\phi\) varies, it has a pole at
\(z=\frac12\) if and only if
\[
L\left(\frac12,\sigma\times\pi\right)\ne0
\quad\text{and}\quad
L(z,\sigma,\wedge^2)\ \text{has a pole at \(z=1\)}.
\]
Furthermore, it has a pole at \(z=1\) if and only if
\(L(z,\sigma\times\pi)\) has a pole there.
\end{prop}

\begin{proof} By the Langlands theory of Eisenstein series (\cite{la2}), the analytic properties of the family of $E(\phi,z)$ are controlled by those of the family
\[E_{\P_{k+a,a}}(\phi,z)=\phi_z+M_z(\phi_z).
\]
Therefore, $E(\phi,z)$ has a pole at $z=z_0$ of order $r$ if and only if $M_z(\phi_z)$ has a pole at $z=z_0$ of order $r$. For a factorizable $\phi=\otimes_v \phi_v \in \As_{\P_{k+a,a}}^{\sigma \boxtimes \pi}(\Sp_{2k+2a})$, \eqref{norm2} and Proposition~\ref{p1} tell us that $M_z(\phi_z)$ has a pole at $z=z_0$ of order $r$ if and only if
\[
\frac{L(z,\sigma \times \pi)\cdot L(2z,\sigma,\wedge^2)}
{L(z+1,\sigma\times \pi)\cdot L(2z+1,\sigma, \wedge^2)
\cdot \epsilon(z,\sigma \times \pi,\psi)
\cdot \epsilon(2z,\sigma,\wedge^2,\psi)}
\]
has a pole at $z=z_0$ of order $r$.

By \cite{JSS}, \cite{Sha1} and \cite[Theorem~1.3]{Gr},  $L(z,\sigma \times \pi)$ and $L(z,\sigma,\wedge^2)$ converges absolutely and does not vanish for $\Re(z) \ge 1$. Furthermore, $L(z,\sigma \times \pi)$ and $L(z,\sigma,\wedge^2)$ are holomorphic on $0<\Re(z)<1$ and they have possible simple pole at $z=1$. By putting all these together, the proposition follows.
\end{proof}

\begin{prop}\label{p3}
Assume that $\sigma$ is isomorphic to one of the isobaric summands of $M_{\pi}$. Then there exists $\phi \in \As_{\P_{k+a,a}}^{\sigma \boxtimes \pi}(\Sp_{2k+2a})$ such that $E(\phi,z)$ has a pole at $z=1$.
\end{prop}
\begin{proof}Note that $L(z,\sigma \times \sigma)=L(z,\sigma, \wedge^2)\cdot L(z,\sigma, \Sym^2)$. It is known that neither $L(z,\sigma, \wedge^2)$ and $L(z,\sigma, \Sym^2)$ has a zero at $z=1$. Since $M_{\pi}$ is a discrete tempered global $A$-parameter for $\Sp_{2k}$, each of its isobaric summands is of orthogonal type by the definition recalled in Subsection~\ref{dgp}; hence $L(z,\sigma, \Sym^2)$ has a simple pole at $z=1$.
%and by \cite{GS}, $L(z,\sigma, \wedge^2)$ is nonzero at $z=0$.
Therefore, $L(z,\sigma \times \sigma)$ has a simple pole at $z=1$ and it implies that $\sigma=\sigma^{\vee}$. %\cite[Theorem 3.1]{Kim} asserts that $L(z,\sigma, \wedge^2)$ is holomorphic and nonvanishing at any real points $s>1$.
Then the proposition follows from Proposition~\ref{p2}.
\end{proof}

Since the proofs of the following propositions are almost same with the above propositions, we omit the proofs.
\begin{prop}\label{p4}
Let $M_{\pi'}$ be a discrete tempered global $A$-parameter for $\Mp_{2k}$, let $\pi'\in\Pi_{\phi_{M_{\pi'}}}^{\mathrm{aut}}$, and let $\sigma$ be an irreducible unitary cuspidal automorphic representation of $\GL_a(\A)$.
For $\phi' \in  \As_{\P_{k+a,a}}^{\sigma_{\psi} \boxtimes \pi'}(\Mp_{2k+2a})$,
the Eisenstein series $E(\phi',z)$ has at most a simple pole at $z=\frac{1}{2}$ and $z=1$.
Moreover, it has a pole at $z=\frac{1}{2}$ as $\phi'$ varies
if and only if
$L_{\psi}(z,\sigma\times \pi')$ is non-zero at $z=\frac{1}{2}$
and $L(z,\sigma, \Sym^2)$ has a pole at $z=1$.
\end{prop}
\begin{proof}It directly follows from the computation of the constant terms of the Eisenstein series $E(\phi',z)$ in \cite[Sec. 3.2]{GJR}.
\end{proof}
\begin{prop}[{\cite[Proposition 5.3]{Y}}] \label{p5}
Assume that $\sigma$ is isomorphic to one of the isobaric summands of $M_{\pi'}$. Then there exists $\phi' \in \As_{\P_{k+a,a}}^{\sigma_{\psi} \boxtimes \pi'}(\Mp_{2k+2a})$ such that $E(\phi',z)$ has a pole at $z=1$.
\end{prop}

\begin{rem}When $\pi'$ and $\pi$ are irreducible globally generic cuspidal representations of $\Mp_{2k+2a}(\A)$ and $\Sp_{2k+2a}(\A)$, respectively, then Proposition~\ref{p2}--\ref{p5} are proved in \cite[Proposition 3.2]{GJR}, \cite[Proposition~5.1, Proposition~5.3]{Y}.
\end{rem}

Thanks to Proposition \ref{p2} and  Proposition \ref{p3}, we can define the residues of the Eisenstein series to be the limits
\[
\mathcal{E}^{0}(\phi)=\lim_{z \to \frac{1}{2}}(z-\frac{1}{2})\cdot E(\phi,z), \quad \mathcal{E}^{1}(\phi)=\lim_{z \to 1}(z-1)\cdot  E(\phi,z), \quad
\quad \phi \in \As_{\P_{k+a,a}}^{\sigma \boxtimes \pi}(\Sp_{2k+2a}).
\]
Let $\mathcal{E}^{0}(\sigma,\pi)$ (resp. $\mathcal{E}^{1}(\sigma,\pi)$) be the residual representation of $\Sp_{2k+2a}(\A)$ generated by $\mathcal{E}^{0}(\phi)$ (resp. $\mathcal{E}^{1}(\phi)$) for $\phi \in \As_{\P_{k+a,a}}^{\sigma \boxtimes \pi}(\Sp_{2k+2a})$.

\noindent Thanks to Proposition \ref{p4} and  Proposition \ref{p5}, we can similarly define the residues of the Eisenstein series to be the limits
\[
\mathcal{E}^{0}(\phi')=\lim_{z \to \frac{1}{2}}(z-\frac{1}{2})\cdot E(\phi',z), \quad \mathcal{E}^{1}(\phi')=\lim_{z \to 1}(z-1)\cdot  E(\phi',z), \quad
\quad \phi' \in \As_{\P_{k+a,a}}^{\sigma_{\psi} \boxtimes \pi'}(\Mp_{2k+2a}).
\]
Let $\mathcal{E}^{0}(\sigma,\pi')$ (resp. $\mathcal{E}^{1}(\sigma,\pi')$) be the residual representation of $\Mp_{2k+2a}$ generated by $\mathcal{E}^{0}(\phi')$ (resp. $\mathcal{E}^{1}(\phi')$) for $\phi' \in \As_{\P_{k+a,a}}^{\sigma_{\psi} \boxtimes \pi'}(\Mp_{2k+2a})$.

\begin{rem}  \label{r1}
If either \(L(z,\sigma,\wedge^2)\) or
\(L(z,\sigma,\Sym^2)\) has a pole at \(z=1\), then
\(\sigma\simeq\sigma^\vee\), and hence its central character
\(\omega_\sigma\) is quadratic. In the exterior-square case,
\(\sigma\) is of symplectic type; consequently \(a\) is even and
\(\omega_\sigma=\mathbf 1\) (see \cite{JS}). In the
symmetric-square case, \(\sigma\) is of orthogonal type, and no
analogous parity or trivial-central-character conclusion holds in
general.
\end{rem}

\section{\textbf{Reciprocal Non-vanishing of the Fourier--Jacobi periods}}\label{sec:reciprocal}
In this section, we shall prove our main theorem. To do this, we should prove a reciprocal non-vanishing theorem of the Fourier--Jacobi periods which has its own value but also crucial to prove our main theorem.

Let $m,a$ be positive integers and let $r$ be a positive integer. Write $n=m+r$ and let $\left(W_{n+a},\left<\ ,\ \right>\right)$ be a  symplectic spaces over $F$ of dimension $2(n+a)$. Fix maximal totally isotropic subspaces $X$ and $X^*$ of $W_{n+a}$, in duality, with respect to $\left<\cdot , \cdot\right>$.
Fix a complete flag in $X$
\(
0=X_0 \subset X_1 \subset \cdots \subset X_{n+a}=X,
\)
and choose a basis $\{e_1,e_2,\cdots,e_{n+a}\}$ of $X$
such that $\{e_1,\cdots,e_i\}$ is a basis of $X_i$ for each $1\le i \le n+a$.
Let $\{f_1,f_2,\cdots ,f_{n+a}\}$ be the basis of $X^*$ which is dual to the fixed basis $\{e_1,e_2,\cdots,e_{n+a}\}$ of $X$,
i.e., $\left<e_i,f_j\right>=\delta_{ij}$ for $1\le i,j \le n+a$,
where $\delta_{i,j}$ denotes the Kronecker delta.
For each $1\le i \le n+a$, write $X^*_i$ for the subspace of $X^*$ spanned by $\{f_1, f_2,\cdots ,f_i\}$ and write $Z_m$ (resp. $Z_m^*$) be the subspace of $X$ (resp. $X^*$) generated by $e_{r+a+1},\cdots, e_{n+a}$ (resp. $f_{r+a+1},\cdots, f_{n+a}$). We also write $Y_a$ (resp. $Y_a^*$) for the subspace of $X$ (resp. $X^*$) spanned by $e_{r+1},\cdots, e_{r+a}$ (resp. $f_{r+1},\cdots, f_{r+a}$). Then $X$ (resp. $X^*$) has the polar decomposition $X=X_r+Y_a+Z_m$ (resp. $X^*=X_r^*+Y_a^*+Z_m^*$). For a $F$-dimensional vector space $V$, let $\mc{S}(V)(\A)$ be the Schwartz space on $V(\A)$. For an arbitrary positive integer $k$, the Schodinger model gives rise to the global Weil representation $\Omega_{\psi^{-1},W_k}$ of $\mathcal{H}_{W_k} (\A) \rtimes \Mp_{2k}(\A)$ on $\mc{S}(X_k^*)(\A)$. As in the local case, we denote by $\omega_{\psi^{-1},W_{k}}$ the restriction of $\nu_{\psi^{-1},W_{k}}$ to $\Mp_{2k}(\A)$. We also define a generic character $\lambda_r$ of $\mc{N}_{r}(\A)$ as in Sec. 3. Then $\nu_{\psi^{-1},W_{k}}=\lambda_r^{-1} \otimes \Omega_{\psi^{-1},W_k}$ is a smooth representation of $H_{k+r,r}(\A)\coloneqq N_{k+r,r}(\A) \rtimes \Mp_{2k}(\A).$
Via the natural inclusion maps $N_{n,r} \hookrightarrow N_{n+r,r}$, $\Mp_{2m} \hookrightarrow \Mp_{2m+2a}$, we regard elements of $N_{n,r}$ and $\Mp_{2m}$ as those of $N_{n+r,r}$ and $\Mp_{2m+2a}$, respectively. For $h \in \GL(Y_a^*)$, put $h^* \in \GL(Y_a)$ satisfying $\left<hx, y\right>=\left< x,h^*y\right>$. Write $\mc{S}=\mathcal{S}(Y_a^*+Z_m^*)(\A).$ Since $\mc{S} \simeq \mc{S}(Y_a^*(\A)) \otimes \mc{S}(Z_m^*(\A))$, we describe the (partial) $N_{n+a,r}(\A) \rtimes \Mp_{2m+2a}(\A)$-action of $\nu_{\psi^{-1},W_{m+a}}$ on $\mc{S}(Y_a^*(\A)) \otimes \mc{S}(Z_m^*(\A))$ as follows: (see \cite[pp. 58-59]{At}.)

For $f_1 \otimes f_2 \in \mc{S}(Y_a^*(\A)) \otimes \mc{S}(Z_m^*(\A))$ and $y \in Y_a^*(\A), z \in Z_m^*(\A),$
\begin{align}
&
\left(
\nu_{\psi^{-1},W_{m+a}}((n,1))(f_1\otimes f_2)
\right)(y,z)
\notag\\
&\qquad
=f_1(y)
\left(\nu_{\psi^{-1},W_m}((n,1))f_2\right)(z),
\qquad n\in N_{n,r}(\A),
\label{w1}\\
&
\left(
\nu_{\psi^{-1},W_{m+a}}((1,\wt g_0))(f_1\otimes f_2)
\right)(y,z)
\notag\\
&\qquad
=f_1(y)
\left(\nu_{\psi^{-1},W_m}((1,\wt g_0))f_2\right)(z),
\qquad \wt g_0\in\Mp_{2m}(\A),
\label{w2}\\
&
\left(
\nu_{\psi^{-1},W_{m+a}}\bigl((1,(h_0,\zeta))\bigr)
(f_1\otimes f_2)
\right)(y,z)
\notag\\
&\qquad
=(|\cdot|^{1/2})_{\psi^{-1}}(\det h_0,\zeta)
f_1(h_0^*y)f_2(z),
\notag\\
&\qquad
(h_0,\zeta)\in\wt{\GL}(Y_a)(\A),
\label{w3}\\
&
\left(
\nu_{\psi^{-1},W_{m+a}}\bigl((x+x',0)\bigr)
(f_1\otimes f_2)
\right)(y,z)
\notag\\
&\qquad
=\psi^{-1}\left(
\left<y,x\right>+\frac12\left<x',x\right>
\right)f_1(y+x')f_2(z),
\notag\\
&\qquad
x\in Y_a(\A),\qquad x'\in Y_a^*(\A).
\label{w4}
\end{align}
These partial actions will be used in the proof of Lemma~\ref{l3}.

For $f \in \mc{S}=\mathcal{S}(Y_a^*+Z_m^*)(\A)$,
 its associated theta function $\Theta_{\psi^{-1},W_{m+a}}(f)(\cdot)$ is defined by
\[
\big(\Theta_{\psi^{-1},W_{m+a}}(f)\big)(h)
\coloneqq \sum_{y \in \left(Y_a^*+Z_m^*\right)(F)} \left(\nu_{\psi^{-1},W_{m+a}}(h)f\right)(y),
\] where $h=((u,n),g) \in \left( (\U_{n+a,r} \rtimes \mathcal{N}_{r}) \rtimes \Mp_{2m+2a}\right)(\A)= (N_{n+a,r} \rtimes \Mp_{2m+2a})(\A)=H_{n+a,r}(\A)$. Note that this is slowly increasing function. When $f \in \mc{S}$ is a pure tensor $f_1 \otimes f_2\in \mc{S}(Y_a^*(\A)) \otimes \mc{S}(Z_m^*(\A))$, then \[\big(\Theta_{\psi^{-1},W_{m+a}}(f)\big)(h)
\coloneqq\sum_{y \in Y_a^*(F), z \in Z_m^* (F)} \left(\Omega_{\psi^{-1},Y_a+Y_a^*}(h)f_1\right)(y)\cdot \left(\nu_{\psi^{-1},W_{m}}(h)f_2\right)(z).
\]
For each $f \in \mc{S}$, $\Theta_{\psi^{-1},W_{m+a}}( f)$ is an automorphic form and the map  $f \to \Theta_{\psi^{-1},W_{m+a}}( f)$ gives an automorphic realization of  $\nu_{\psi^{-1},W_{m+a}}$. Since we have fixed $\psi$, we suppress $\psi$ from the notation and write $\nu_{\psi^{-1},W}$ and $\Theta_{\psi^{-1},W}(f)$ by $\nu_{W}$ and $\Theta_{W}(f)$, respectively. We also simply write $\Theta_{\P_{m+a,a}}( f)$ for the constant term of $\Theta_{W_{m+a}}( f)$ along $\P_{m+a,a}$.
\par
The following is an analogue of \cite[Lemma 9.1]{H} whose proof is identically same.
\begin{lem}[cf. {\cite[Lemma 6.2]{Y}}] \label{l0} For $h \in H_{n+a,r}(\A)$, $f \in \mathcal{S}$,
\[
\left(\Theta_{\P_{m+a,a}}(f)\right)(h)
=\sum_{z \in Z_m^*(F)} \left(\nu_{W_{m+a}}(h)f \right)(0,z).\]

\end{lem}

\begin{rem}\label{eval}Let $e$ be the identity element of $H_{n+a,r}(\A) = N_{n+a, r}(\A) \rtimes \Sp_{2m+2a}(\A)$. If $f \in \mc{S}$ is a pure tensor $f_1 \otimes f_2\in \mc{S}(Y_a^*)(\A) \otimes \mc{S}(Z_m^*)(\A)$, then
\[\left(\Theta_{\P_{m+a,a}}(f_1 \otimes f_2)\right)(e)
= f_1(0) \cdot \left( \sum_{z \in Z_m^* (F)} f_2(z) \right).
\]
Since every Schwartz function in $ \mc{S}(Z_m^*)(\A)$ can be obtained as evaluation at $0\in Y_a^*(\A)$ of some Schwartz function in $\mathcal{S}(Y_a^*+Z_m^*)(\A)$, we can regard theta functions in $\nu_{W_m}$ as the evaluation at $e$ of
the constant term of theta functions in $\nu_{W_{m+a}}$ along $\P_{m+a,a}$. This observation will be used in the proof of Lemma \ref{l3}.
\end{rem}

\begin{rem}\label{ind}From Lemma \ref{l0} and (\ref{w3})  we have
\[
\left(\Theta_{\mathcal{P}_{m+a,a}}(f)\right)(u,(h_0,\zeta) g)
= (|\cdot|^{\frac{1}{2}})_{\psi^{-1}}(\det(h_0),\zeta)
\cdot
\left(\Theta_{\mathcal{P}_{m+a,a}}(f)\right)(u,g).
\]

for $u \in N_{n+a,r}, (h_0,\zeta) \in \wt{\GL}(Y_a)$ and $g \in \Mp_{2m+2a}$.
Thus from Lemma \ref{l0} and $(\ref{w1}), (\ref{w2}), (\ref{w3})$, we see that the constant terms of theta functions $\{\Theta_{\P_{m+a,a}}(f)\}_{f\in  \mathcal{S}}$
belong to the induced representation
\[
\mathrm{ind}
_{\wt{\GL}_a(\A) \times_{\{\pm1\}} (N_{n,r}(\A) \rtimes \Mp_{2m}(\A))}
^{N_{n+a,r}(\A) \rtimes \Mp_{2m+2a}(\A)}
(|\cdot|^{\frac{1}{2}})_{\psi^{-1}} \boxtimes \nu_{W_m}.
\]
\end{rem}
From now on, we consider the Fourier-Jacobi periods involving a (residual) Eisenstein series. Since the (residual) Eisenstein series is not a cusp form, $\mc{FJ}_{\psi}$ here denotes the regularized Fourier-Jacobi periods defined in \cite[Definition 4.5]{H}. (Note that the regularized Fourier-Jacobi periods naturally extend the definition of the original Fourier-Jacobi period, so it justifies using the same notation. See \cite[Remark 4.6]{H}). Since we have fixed $\psi$, we omit it from the notation and simply write $\mc{FJ}$ in place of $\mc{FJ}_{\psi}$ for the remainder of this section.

Before setting up the local machinery for Lemma~\ref{fj-bessel-mixed-model} and Proposition~\ref{theta-fj-aux}, we explain what it is for.

\begin{rem}
\label{gapH}
Proposition~\ref{l1} below is the technical heart of the proof of
Theorem~\ref{e}, and it is the analogue, for symplectic--metaplectic
pairs, of \cite[Lemma~9.4]{H}.  We cannot, however, follow the approach
taken there.  The computation of the Jacquet module of the relevant
Fourier--Jacobi coefficient carried out in \cite{H} contains a gap, and
the case \(a=1\), which turns out to require a separate treatment, was not addressed.
We therefore give a different and complete proof here.  It is general
enough to apply verbatim to the skew-Hermitian pairs of \cite{H}, and it
thereby repairs \cite[Lemma~9.4]{H} as well.

The proof splits according to the size of \(a\).  When \(a\ge2\), the
required vanishing is obtained from a Jacquet module computation, which
is the content of Propositions~\ref{key} and~\ref{key3}.  When \(a=1\)
--- a case that does occur, namely for the pairs
\(\Mp_{2n}\times\Sp_{2m}\) --- the Jacquet module retains too little
information for that argument, and we proceed instead through the local
theta correspondence, by way of Proposition~\ref{theta-fj-aux} below.
\end{rem}

We now introduce the notation needed for
Lemma~\ref{fj-bessel-mixed-model} and Proposition~\ref{theta-fj-aux}. Throughout this passage and in those two
statements, \(N\) and \(M\) denote \emph{integers} with \(N>M>0\), and
not \(A\)-parameters. Let \(F\) be a non-archimedean local field of
characteristic zero and fix a polarization
\[
W_N=X_r\oplus W_M\oplus X_r^*,\qquad r=N-M>0.
\]
Let \(P(X_r)=M(X_r)N(X_r)\) be the maximal parabolic subgroup of
\(\Sp(W_N)\) stabilizing \(X_r\), where
\(M(X_r)\simeq\GL(X_r)\times\Sp(W_M)\).  Fix a complete flag in \(X_r\),
let \(U_r\subset\GL(X_r)\) be the corresponding maximal unipotent
subgroup with the generic character \(\lambda_r\), and put
\[
N_{N,r}=U_rN(X_r),\qquad
\widetilde H_{N,r}=N_{N,r}\rtimes\Mp(W_M).
\]
For the quadratic line \(L=\langle1\rangle\), set
\[
\varpi_K=\omega_{\psi,L,W_K}^{\vee}
\simeq\omega_{\psi^{-1},L,W_K}\qquad(K=N,M),
\]
where the isomorphism follows from
\cite[Propositions~9.31--9.33]{GKT}.  We use the standard
\(\psi\)-dependent identification of genuine representations of
metaplectic general linear groups with linear representations, and all
parabolic inductions are normalized.

For a genuine finite-length representation \(\Pi\) of \(\Mp(W_N)\), let
\(\mathfrak S_r(\Pi)\) denote the set of inertial classes of the
\(\GL_r\)-factors occurring in the supercuspidal supports of its
irreducible constituents.  We call an irreducible supercuspidal
representation \(\rho\) of \(\GL(X_r)\) \(\Pi\)-separated if
\[
[\rho],[\rho^\vee]\notin\mathfrak S_r(\Pi).
\]
Such a \(\rho\) always exists, since \(\mathfrak S_r(\Pi)\) is finite,
whereas \(\GL_r(F)\) has infinitely many supercuspidal inertial classes.

\begin{lem}
\label{fj-bessel-mixed-model}

Let \(\Pi\) be a genuine finite-length representation of \(\Mp(W_N)\)
contained in one Bernstein block, and let \(\sigma_0\) be a finite-length
representation of \(\Sp(W_M)\). For any \(\Pi\)-separated irreducible
supercuspidal representation \(\rho\) of \(\GL(X_r)\), regard
\(\rho\rtimes\sigma_0\) by inflation as a non-genuine representation of
\(\Mp(W_N)\). Then there is a natural isomorphism
\begin{equation}\label{general-fj-reduction}
\Hom_{\widetilde H_{N,r}}
 \left(\Pi\otimes\sigma_0\otimes\nu_{\psi^{-1},W_M},\CC\right)
\simeq
\Hom_{\Mp(W_N)}
 \left(\Pi\otimes(\rho\rtimes\sigma_0)\otimes\varpi_N,\CC\right).
\end{equation}
\end{lem}

The central element of \(\Mp(W_M)\) acts trivially on the tensor product
in the left-hand side of \eqref{general-fj-reduction}.  Hence that
Hom-space may equivalently be taken over
\(N_{N,r}\rtimes\Sp(W_M)\).

\begin{proof}
The assertion follows from the mixed-model argument in
\cite[Theorem~16.1]{Gan2}, with $\psi$ replaced by $\psi^{-1}$ and
with the genuine and non-genuine factors interchanged.  We briefly
indicate the required changes.

Under the standard $\psi$-dependent identification of genuine
representations of the metaplectic cover of $\GL(X_r)$ with linear
representations, the genuine character occurring in the open-orbit
part of the mixed model cancels with the corresponding genuine
factor in the inducing datum.  The modulus characters appearing in
\cite[Theorem~16.1]{Gan2}, where unnormalized induction is used, are
absorbed here by our normalized-induction convention.

The evaluation-at-zero quotient belongs to Bernstein components
whose $\GL_r$-support contains an unramified twist of $\rho$ or
$\rho^\vee$.  Hence its projection to the Bernstein block containing
$\Pi$ is zero by the choice of $\rho$.  The open-orbit term is treated
exactly as in loc.\ cit., using the Gelfand--Kazhdan description of the
restriction of the supercuspidal representation $\rho$ to the
mirabolic subgroup, followed by induction in stages and compact
Frobenius reciprocity.  This gives \eqref{general-fj-reduction}.

Finally, the argument is exact in $\sigma_0$ and is carried out after
projection to the Bernstein block of $\Pi$; hence it applies without
change to the finite-length representations considered here.
\end{proof}

\begin{prop}\label{theta-fj-aux}
Let $F$ be a non-archimedean local field of characteristic zero and let
$N>M\ge1$.  Let $V=V_{2M-2}$ be an even-dimensional quadratic space in
a fixed Witt tower, with discriminant character $\chi_V$, and let
$\tau$ be an irreducible representation of $\mathrm O(V)$ such that
\(
 \sigma_0\coloneqq\theta_{\psi^{-1},V,W_M}(\tau)\ne0.
\)
Put
\[
 V^+=V\oplus\mathbb H^{N-M}=V_{2N-2},\qquad
 V'=V^+\oplus L=V_{2N-1},\qquad L=\langle1\rangle.
\]
Then $\disc(V')=\disc(V)$, and the genuine character attached to $V'$
is $(\chi_V)_{\psi^{-1}}$ in our conventions.

Let $\widetilde\pi_0$ be an irreducible genuine subquotient of a
principal series
\[
 \Ind_{\widetilde B_{N-1}}^{\Mp(W_{N-1})}
 \left((\eta_1)_\psi\boxtimes\cdots\boxtimes
       (\eta_{N-1})_\psi\right),
 \qquad |e(\eta_i)|<\frac12.
\]
For $z\in\CC$, put
\[
 I_N(z)=\chi_V|\cdot|^z\rtimes\widetilde\pi_0.
\]
Set
\[
 \mathcal D=
 \left\{z\in\CC:|\cdot|^z=|\cdot|^{1/2}\right\}
 =\frac12+\frac{2\pi i}{\log q}\ZZ,
\]
where $q$ is the residue-field cardinality.  Then, for
$z\notin\mathcal D$,
\[
\Hom_{N_{N,N-M}\rtimes\Sp(W_M)}
\left(I_N(z)\otimes\sigma_0\otimes
\nu_{\psi^{-1},W_M},\CC\right)=0.
\]
\end{prop}

\begin{proof}
Put $r=N-M$.  We first show that, for every genuine finite-length
representation $\Pi$ of $\Mp(W_N)$ contained in one Bernstein block,
\begin{equation}\label{fj-implies-weil-hom}
\Hom_{\widetilde H_{N,r}}
 \left(\Pi\otimes\sigma_0\otimes
       \nu_{\psi^{-1},W_M},\CC\right)\ne0
\quad\Longrightarrow\quad
\Hom_{\Mp(W_N)}
 \left(\omega_{\psi^{-1},V',W_N},\Pi^{\vee}\right)\ne0.
\end{equation}

For a parabolic subgroup $S=M_SU_S$ of either member of the dual
pair, we write
\(
 R_S(\pi)=J_{U_S}^{\mathrm{norm}}(\pi)
\)
for the normalized Jacquet module of $\pi$ along $S$.

Choose \(\rho\) as in Lemma~\ref{fj-bessel-mixed-model}. Since only
finitely many additional inertial classes have to be avoided, we may
choose \(\rho\) to be non-self-dual and to satisfy the following
conditions:
\begin{enumerate}
\item no unramified twist of \(\chi_V\rho\) or
\((\chi_V\rho)^\vee\) occurs in the \(\GL_r\)-cuspidal support of any
Jacquet module of \(\tau\);
\item the supercuspidal support of every \(\GL_r\)-representation
occurring in a non-open stratum of the Kudla filtration of
\(R_P(\omega_{\psi^{-1},V^+,W_N})\), where
\(M_P\simeq\GL_r\times\mathrm O(V)\), is disjoint from the inertial
class of \(\chi_V\rho\).
\end{enumerate}

For $r\ge2$ the second condition is automatic.  Indeed, for
$1\le k\le r-1$ the $\GL_r$-representation in the $k$-th stratum
is induced from the proper Levi $\GL_{r-k}\times\GL_k$, while the
$k=0$ stratum is a determinant character whose supercuspidal
support lies on the Borel Levi.  Neither can belong to the
Bernstein component of the supercuspidal representation
$\chi_V\rho$.  When $r=1$, the closed stratum is a character of
$\GL_1$; it is the only non-open stratum, and condition~(2) excludes
precisely its inertial class.  By the standard rank-one irreducibility criterion for classical groups,
using the non-self-duality of $\chi_V\rho$ and the disjointness of its
inertial line from the Jacquet supports of $\tau$,
\(
 \Sigma_\rho=(\chi_V\rho)\rtimes\tau
\)
is irreducible.  Put $J_\rho=\rho\rtimes\sigma_0$; no irreducibility of
$J_\rho$ will be needed.

We record the relevant direction of the theta induction principle.
Take smooth vectors in \cite[Proposition~13.21]{GKT}, apply normalized
Frobenius reciprocity, and use \cite[Theorem~16.1]{GKT} with the roles of
the two members of the dual pair interchanged.  Exact Bernstein
projection and condition~(2) kill every non-open stratum.  The open
stratum, together with \cite[Lemma~16.7]{GKT}, gives
\begin{equation}\label{kudla-dual-identity}
 \Theta_{\psi^{-1},V^+,W_N}(\Sigma_\rho)^{\vee}
 \simeq
 \rho^{\vee}\rtimes
 \Theta_{\psi^{-1},V,W_M}(\tau)^{\vee}.
\end{equation}
Here the factor $\chi_V$ in $\Sigma_\rho$ cancels the Leray character on
the open stratum.  Indeed, since $W_N$ is split symplectic, we use the
canonical Leray splitting of $\mathrm O(V^+)$ with splitting data
$(\tbf1,\psi^{-1})$, so $\mu_{W_N}=\tbf1$ by
\cite[Proposition~12.3]{GKT}.  Moreover, in the open stratum $k=t=r$ one
has $t-k=0$, so no $|\det|^{\lambda_{t-k}}$-factor occurs; with normalized
Jacquet modules and normalized Frobenius reciprocity, no modulus character
remains.  The big theta lifts in
\eqref{kudla-dual-identity} have finite length by
\cite[Theorem~16.19]{GKT}.  Since
$\Theta_{\psi^{-1},V,W_M}(\tau)\twoheadrightarrow\sigma_0$, exactness of
normalized induction and smooth duality give
\begin{equation}\label{kudla-induced-pair}
 \Theta_{\psi^{-1},V^+,W_N}(\Sigma_\rho)
 \twoheadrightarrow J_\rho.
\end{equation}
Consequently the defining maximal $\Sigma_\rho$-isotypic quotient of the
Weil representation yields
\[
 \omega_{\psi^{-1},V^+,W_N}\twoheadrightarrow
 \Sigma_\rho\boxtimes J_\rho.
\]
If the first Hom-space in \eqref{fj-implies-weil-hom} is nonzero,
Lemma~\ref{fj-bessel-mixed-model} gives a nonzero map
$J_\rho\otimes\varpi_N\to\Pi^\vee$.  By block compatibility of the Weil
representation and tower compatibility of the Leray splittings
\cite[Proposition~9.19 and Lemma~12.7]{GKT},
\[
 \omega_{\psi^{-1},V',W_N}
 \simeq \omega_{\psi^{-1},V^+,W_N}\otimes\varpi_N.
\]
Composing these maps and applying a linear functional on $\Sigma_\rho$
which does not annihilate the image proves
\eqref{fj-implies-weil-hom}.

We next prove the vanishing consequence of the metaplectic Kudla
filtration needed below.  Let $W$ have dimension $2n$, let $V_o$ be an
odd-dimensional quadratic space of dimension $m\le2n-1$ in a fixed Witt
tower, and put
\(
 s_0=\frac{m-2n}{2}.
\)
Let
\(
 \Sigma=\Ind_{\widetilde B_n}^{\Mp(W)}
 (\xi_1\boxtimes\cdots\boxtimes\xi_n)
\)
be a genuine principal series.  If
\begin{equation}\label{critical-character-avoidance}
 \xi_i\ne\chi_{V_o,\psi^{-1}}|\cdot|^{s_0}
 \qquad(1\le i\le n),
\end{equation}
then
\begin{equation}\label{ordered-kudla-vanishing}
 \Hom_{\Mp(W)}
 \left(\omega_{\psi^{-1},V_o,W},\Sigma\right)=0.
\end{equation}

Indeed, write
$\Sigma=\xi_1\rtimes\Sigma'$, where
\[
 \Sigma'=\Ind_{\widetilde B_{n-1}}^{\Mp(W')}
 (\xi_2\boxtimes\cdots\boxtimes\xi_n),
 \qquad \dim W'=2n-2.
\]
Normalized Frobenius reciprocity identifies the Hom-space in
\eqref{ordered-kudla-vanishing} with
\[
 \Hom_{\widetilde{\GL}_1\times_{\mu_2}\Mp(W')}
 \left(R_{Q_1}(\omega_{\psi^{-1},V_o,W}),
       \xi_1\boxtimes\Sigma'\right).
\]
By \cite[Theorem~16.8]{GKT} with $t=1$, the normalized Jacquet
module $R_{Q_1}(\omega_{\psi^{-1},V_o,W})$ has a two-step Kudla
filtration.  Denote by $R^1$ its submodule corresponding to the
$k=1$ open stratum.  Then there is an exact sequence
\begin{equation}\label{kudla-line-exact}
0\longrightarrow R^1\longrightarrow
R_{Q_1}(\omega_{\psi^{-1},V_o,W})\longrightarrow
\chi_{V_o,\psi^{-1}}|\cdot|^{s_0}
 \boxtimes\omega_{\psi^{-1},V_o,W'}\longrightarrow0.
\end{equation}
Here the parameter denoted by $n$ in \cite[Theorem~16.8]{GKT} is
$\dim W=2n$, so its exponent is
$\lambda_1=(m-(2n+1)+1)/2=s_0$.  The last quotient in
\eqref{kudla-line-exact} contributes only when
$\xi_1=\chi_{V_o,\psi^{-1}}|\cdot|^{s_0}$, which is excluded by
\eqref{critical-character-avoidance}.  Hence every nonzero map restricts
nontrivially to $R^1$.

If $V_o$ is anisotropic, then $R^1=0$.  Otherwise the $k=1$ stratum in
\cite[Theorem~16.8]{GKT} shows that $\Mp(W')$ acts on $R^1$ only through
$\omega_{\psi^{-1},V_{m-2},W'}$.  Consequently,
\[
 \Hom_{\widetilde{\GL}_1\times_{\mu_2}\Mp(W')}
 \left(R^1,\xi_1\boxtimes\Sigma'\right)\ne0
 \quad\Longrightarrow\quad
 \Hom_{\Mp(W')}
 \left(\omega_{\psi^{-1},V_{m-2},W'},\Sigma'\right)\ne0.
\]
Since $(m-2)-2(n-1)=m-2n$ and the discriminant character is unchanged
along the Witt tower, the same argument removes
$\xi_2,\xi_3,\ldots$ successively with the same critical character.
Let $V_{o,\an}$ be the anisotropic kernel and put
\(
 b=\frac{m-\dim V_{o,\an}}{2}.
\)
After $b$ steps the quadratic space is anisotropic.  Moreover,
$m\le2n-1$ and $\dim V_{o,\an}\ge1$ imply
\(
 b\le\frac{(2n-1)-1}{2}=n-1.
\)
Thus the symplectic group has not been exhausted: at the last step
$\dim W^{(b)}=2(n-b)\ge2$, so $\xi_{b+1}$ is still present, $R^1=0$, and
the remaining top quotient is excluded by
\eqref{critical-character-avoidance}.  This proves
\eqref{ordered-kudla-vanishing}.  The bound $m\le2n-1$ is used precisely
here and cannot be omitted in general; compare the rank-one counterexample in
\cite[Lemma~19.9]{GKT} when $(n,m)=(1,3)$.

Apply this with $V_o=V'$ and $W=W_N$.  Here $m=2N-1=2n-1$ with
$n=N$, so the bound $m\le2n-1$ holds with equality and
$b=N-1=n-1$.  By the subrepresentation theorem,
$\widetilde\pi_0^\vee$ embeds into a genuine principal series of
$\Mp(W_{N-1})$ with the same supercuspidal support, up to inversion and
Weyl ordering.  Hence its inducing characters
$\xi_2,\ldots,\xi_N$ still have real exponents in
$(-\frac12,\frac12)$.  Induction in stages gives
\[
 I_N(z)^\vee\hookrightarrow
 \Sigma_N(z)=\Ind_{\widetilde B_N}^{\Mp(W_N)}
 \left((\chi_V|\cdot|^{-z})_{\psi^{-1}}
       \boxtimes\xi_2\boxtimes\cdots\boxtimes\xi_N\right).
\]
Here $\dim V'=2N-1$, so $s_0=-\frac12$, and
$\chi_{V',\psi^{-1}}=(\chi_V)_{\psi^{-1}}$.  The characters $\xi_i$ for $i\ge2$ cannot be critical, since
$|e(\xi_i)|<\frac12=|s_0|$, while the first character is
critical exactly when $z\in\mathcal D$.  Hence
\eqref{ordered-kudla-vanishing} applies to $\Sigma_N(z)$ for
$z\notin\mathcal D$.  Since
$\Hom_{\Mp(W_N)}(\omega_{\psi^{-1},V',W_N},-)$ is left exact, the above
embedding yields
\[
 \Hom_{\Mp(W_N)}
 \left(\omega_{\psi^{-1},V',W_N},I_N(z)^\vee\right)=0
 \qquad(z\notin\mathcal D).
\]
Finally, $I_N(z)$ has finite length and lies in one Bernstein block, so
\eqref{fj-implies-weil-hom} applies.  The preceding vanishing proves the
proposition.
\end{proof}

\begin{prop}\label{l1}
Let \(n,m,a\) be positive integers with \(r=n-m>0\), and let
$\sigma$ be an irreducible cuspidal automorphic representation of
$\GL_a(\A)$.  Let $\pi_1\boxtimes\pi_2$ be an irreducible automorphic
representation of $\Mp_{2n}(\A)\times\Mp_{2m}(\A)$ in which exactly one
factor is genuine; write $\pi'$ for the genuine factor and $\pi$ for the
non-genuine one.  Let $M_{\pi}$ and $M_{\pi'}$ be discrete tempered
global $A$-parameters for the two relevant groups, and assume that
\[
 \pi\in\Pi_{\phi_{M_{\pi}}}^{\mathrm{aut}},
 \qquad
 \pi'\in\Pi_{\phi_{M_{\pi'}}}^{\mathrm{aut}},
\]
so that $\pi$ and $\pi'$ are both cuspidal by
Lemma~\ref{tempered-parameter-cuspidal}.

Write
\[
 M_{\pi}=\sigma_1\boxplus\cdots\boxplus\sigma_s,
 \qquad
 M_{\pi'}=\sigma_1'\boxplus\cdots\boxplus\sigma_t',
\]
where the $\sigma_i$ are of orthogonal type and the $\sigma_j'$ are of
symplectic type.  If $\pi_1=\pi'$, assume that
$\sigma\simeq\sigma_i$ for some $i$; if $\pi_1=\pi$, assume that
$\sigma\simeq\sigma_j'$ for some $j$.  Then:
\begin{enumerate}
\item If $\pi_1=\pi'$ and $\pi_2=\pi$, then, for all
$\vi\in\mathcal E^1(\sigma,\pi)$,
$\phi'\in\As_{\P_{n+a,a}}^{\sigma_\psi\boxtimes\pi'}
(\Mp_{2n+2a})$, and $f\in\nu_{W_{m+a}}$,
\[
 \mathcal{FJ}(E(\phi',z),\vi,f)=0.
\]
\item If $\pi_1=\pi$ and $\pi_2=\pi'$, then, for all
$\vi'\in\mathcal E^1(\sigma,\pi')$,
$\phi\in\As_{\P_{n+a,a}}^{\sigma\boxtimes\pi}
(\Sp_{2n+2a})$, and $f\in\nu_{W_{m+a}}$,
\[
 \mathcal{FJ}(E(\phi,z),\vi',f)=0.
\]
\end{enumerate}
\end{prop}

\begin{proof}
By \cite[Proposition~6.3(iii)]{H}, the two regularized periods define,
respectively, equivariant functionals in
\begin{align}
\label{ob1}
&\Hom_{N_{n+a,r}(\A_{\mathrm{fin}})\rtimes
\Sp_{2m+2a}(\A_{\mathrm{fin}})}
\left(
(\sigma|\cdot|^z\rtimes\pi')
\otimes\nu_{W_{m+a}}\otimes\mathcal E^1(\sigma,\pi),\CC
\right),\\
&\label{ob2}
\Hom_{N_{n+a,r}(\A_{\mathrm{fin}})\rtimes
\Sp_{2m+2a}(\A_{\mathrm{fin}})}
\left(
(\sigma|\cdot|^z\rtimes\pi)
\otimes\nu_{W_{m+a}}\otimes\mathcal E^1(\sigma,\pi'),\CC
\right).
\end{align}

We reduce both cases to one unramified finite place.  Let \(S\) contain
the dyadic places and all places where the relevant data are ramified.
If \(a=1\), choose any \(v\notin S\).  If \(a\ge2\), choose an unramified place \(v\) such that \(\operatorname{supp}(\sigma_v)\) (i.e. supercuspidal support of
\(\sigma_v\)) contains a character \(\mu\), with
\(-\frac12<e(\mu)\le0\), which is either nonquadratic or occurs with
multiplicity at least two.  Such a place exists, since otherwise
\(a=2\) and
\(\operatorname{supp}(\sigma_v)=\{\mathbf1,\varepsilon_v\}\) for
almost all \(v\), so that
\(L^S(s,\sigma)=\zeta_F^S(2s)\), contradicting the entirety of
\(L(s,\sigma)\).

At this place, the local component of the residual representation is the
unique unramified Langlands quotient $J_v$ of
\[
 \sigma_v|\det|\rtimes\pi_v
\]
in the first case, and of its metaplectic analogue in the second.  This
rests on the irreducibility of the residual representation, which we
recall from \cite[proof of Lemma~6.1]{Y}.

The cuspidal support of the residues generating $\mathcal E^1(\sigma,\pi)$
consists only of $\sigma|\cdot|^{-1}\boxtimes\pi$ on $\P_{m+a,a}$, so these
residues are square integrable by \cite[Lemma~I.4.11]{Mo}.  Hence
$\mathcal E^1(\sigma,\pi)$ is a unitary, and therefore semisimple, quotient
of $I(1,\sigma\boxtimes\pi)$.  Now put the inducing data in standard order
at each place.  The exponents contributed by $\sigma_v|\det|$ are of the
form $1+x_{i,v}$ with $x_{i,v}\in(-\frac12,\frac12)$, hence exceed
$\frac12$, whereas the exponents in the canonical Langlands datum of the
other factor lie in $[0,\frac12)$, both bounds being furnished by
Lemma~\ref{automatic-almost-tempered}.  Every block coming from
$\sigma_v|\det|$ therefore precedes all the remaining blocks, so
$I_v(1)$ is a quotient of a standard module.  Consequently $I_v(1)$ has a
unique irreducible quotient $J_v$, the space of maps onto it is
one-dimensional, and the Langlands quotient is the unique semisimple
quotient of $I_v(1)$.  We conclude that
\[
 \mathcal E^1(\sigma,\pi)\simeq\bigotimes_v{}'J_v ,
\]
which is irreducible; in particular its component at our unramified place
$v$ is $J_v$.

If either global functional in \eqref{ob1} or \eqref{ob2} were nonzero, choosing a
factorizable residue vector and fixing all components away from such a
place $v$ would give a nonzero local Hom-space with residual factor
$J_v$.  We henceforth work locally at $v$, suppress $v$ from the
notation, and write $J$ for this spherical quotient.

\smallskip
\noindent\emph{Case (i): $\pi_1=\pi'$ and $\pi_2=\pi$.}
By Lemma~\ref{Jac}, it is enough to prove that
\begin{equation}\label{ob3}
\Hom_{N_{n+a,r}\rtimes\Sp_{2m+2a}}
\left(
(\sigma|\cdot|^z\rtimes\pi_1)
\otimes\nu_{W_{m+a}}\otimes J,\CC
\right)=0
\end{equation}
outside a discrete set of $z$.  Lemma~\ref{Jac} identifies
\eqref{ob3} with the corresponding $\Sp_{2m+2a}$-Hom-space involving
\[
 J_{\psi_{r-1}}
 \left(
 (\sigma|\cdot|^z\rtimes\pi_1)
 \otimes\Omega_{W_{m+a}}
 \right).
\]
After reordering, $\pi_1$ is the unique quotient of
\[
 \Ind_{\wt B_n}^{\Mp_{2n}}
 ((\chi_1')_\psi\boxtimes\cdots\boxtimes(\chi_n')_\psi),
 \qquad 0\le e(\chi_i')<\frac12,
\]
and $\pi_2$ is an unramified subquotient of a linear principal series
with the same exponent bound.

Assume first that $a\ge2$.  By Remark~\ref{r1}, $\sigma$ is self-dual.
Since $J$ is the Langlands quotient of
$\sigma_v|\det|\rtimes\pi_{2,v}$, its unramified $L$-parameter is
obtained from that of $\pi_2$ by adjoining $\sigma|\cdot|$ and
$\sigma^\vee|\cdot|^{-1}$.  Now $\pi_2$ has tempered $A$-parameter
$M_{\pi}=\sigma_1\boxplus\cdots\boxplus\sigma_s$, so that
$\phi_{\pi_{2,v}}=M_{\pi,v}$, and by hypothesis $\sigma\simeq\sigma_i$
for some $i$.  Using $\sigma^\vee\simeq\sigma$ we therefore obtain
\begin{equation}\label{length3-block}
 \phi_{J_v}
 =
 \underbrace{
 \bigl(\sigma_v|\cdot|\oplus\sigma_v\oplus\sigma_v|\cdot|^{-1}\bigr)
 }_{\text{the associated }L\text{-parameter of }\sigma\boxtimes[2]}
 \oplus
 \bigoplus_{j\ne i}\sigma_{j,v}.
\end{equation}
We refer to the underbraced summand as a \emph{length-three block}.

By our choice of \(v\), choose an inducing character
\(\mu_\sigma\) of \(\sigma\) as above.  At the level of Borel support the
three characters
$\mu_\sigma|\cdot|,\mu_\sigma,
 \mu_\sigma|\cdot|^{-1}$ occupy three slots and form
\(\mu_\sigma(3)\).  Thus, for a suitable principal series representation
\(\mc E\),
\begin{equation}\label{length3-realization-sp}
 J\ \text{is an irreducible non-generic subquotient of}\quad
\mu_\sigma(3)\rtimes\mc E.
\end{equation}
Both assertions in \eqref{length3-realization-sp} follow from a single
surjection.  Let \(I_B\) be the Borel principal series with the same
support as \(J\); the quotient map
\(
 \mu_\sigma|\cdot|\times\mu_\sigma\times
 \mu_\sigma|\cdot|^{-1}\twoheadrightarrow\mu_\sigma(3)
\)
induces
\(
 I_B\twoheadrightarrow\mu_\sigma(3)\rtimes\mc E.
\)
Since \(I_B^K\) is one-dimensional for a hyperspecial maximal compact
subgroup \(K\) and the target is spherical, exactness of
\(K\)-invariants places the unique spherical constituent of \(I_B\),
namely \(J\), in the target.  Since \(\mu_\sigma(3)\) is a character of
\(\GL_3\) and hence has zero Whittaker module, Rodier's heredity
theorem \cite{Rod} gives the target vanishing \(\psi^{(\alpha)}\)-Whittaker module for \emph{every}
\(\alpha\), and exactness of the Whittaker functor then makes \(J\)
non-generic.  In particular \(J\) is \(\psi^{(1)}\)-non-generic, which is
the form in which Proposition~\ref{key} will use it.
Write
$\sigma\simeq\sigma_1\times\mu_\sigma$ and form the ordered
$\GL_{n+a}$ principal series
\[
 \pi_z^{\mathrm{std}}=
\sigma_1|\cdot|^z\times\mu_\sigma|\cdot|^z
\times\chi_1'\times\cdots\times\chi_n' .
\]
Here the two pieces of the $\sigma|\cdot|^z$-block are adjacent, so
normalized induction in stages gives the required surjection from the
corresponding genuine Siegel principal series onto the induced
representation in \eqref{ob3}.

For the application of Proposition~\ref{key}, let
$\pi_z^{\mathrm{aux}}$ be the Borel principal series obtained from
$\pi_z^{\mathrm{std}}$ by moving the static
$\chi_i'$-block to the outside, while leaving one inducing character of
$\sigma_1|\cdot|^z$ together with
$\mu_\sigma|\cdot|^z$ as the innermost pair in the recursive order of
that proposition.  Every interchange used here crosses a moving
character with a static one.  The corresponding normalized rank-one
intertwining operators are therefore isomorphisms outside a discrete
set of $z$, and hence
\(
\pi_z^{\mathrm{std}}\simeq\pi_z^{\mathrm{aux}}
\)
there.  After expanding $\sigma_1$, the auxiliary series satisfies the
hypotheses of Proposition~\ref{key}: it has the two required moving
characters.  Proposition~\ref{key}, with $n$ there replaced by
$n+a-1$ and $l=r-1$, together with exactness of
$J_{\psi_{r-1}}$, proves the desired vanishing.  The required range is
$(n+a-1)-(r-1)=m+a\ge3$, since \(m\ge1\) by the standing assumptions
and \(a\ge2\) in the present case.

Now let $a=1$.  Then $\sigma$ is quadratic.  Let $V_{2m}$ be the
unramified even quadratic space with discriminant character
$\chi_V=\sigma$.  Since $\sigma$ is an isobaric summand of the global
$A$-parameter $M_{\pi}$ of $\pi_2$, the unramified local theta correspondence
\cite[Chapitre~VI]{MVW} gives an irreducible unramified
$\tau\in\Irr(\mathrm O(V_{2m}))$ such that
\(
 \pi_2\simeq\theta_{\psi^{-1},V_{2m},W_m}(\tau).
\)
The going-down high-rank formula
\cite[Theorem~6.1(2)]{BHtc}, valid for arbitrary irreducible
representations, applied from relative level $-1$ to relative level
$-3$, gives
\(
 J\simeq\theta_{\psi^{-1},V_{2m},W_{m+1}}(\tau).
\)
Indeed, both sides are the unique spherical quotient of
\(
 \sigma|\cdot|\rtimes\pi_2.
\)
Taking
$N=n+1$ and $M=m+1$, Proposition~\ref{theta-fj-aux} applies to
\(
 I_N(z)=\sigma|\cdot|^z\rtimes\pi_1
\)
and proves \eqref{ob3} outside its discrete exceptional set.

\smallskip
\noindent\emph{Case (ii): \(\pi_1=\pi\) and \(\pi_2=\pi'\).}
Here \(\sigma\) is of symplectic type, so \(a=2b\ge2\).  At any
unramified place, after replacing \(\mu_i\) by \(\mu_i^{-1}\) if
necessary, write
\[
 \sigma\simeq
 \mu_1\times\cdots\times\mu_b\times
 \mu_b^{-1}\times\cdots\times\mu_1^{-1},
 \qquad -\frac12<e(\mu_i)\le0.
\]
Since \(M_{\pi'}\) contains \(\sigma\), the same unramified
calculation as in \cite[proof of Lemma~6.1, pp.~69--70]{Y} shows that
the Borel support of the spherical residual constituent \(J\) contains,
in the sense of \eqref{length3-block}, the length-three blocks
\[
 \mu_i|\cdot|,\quad\mu_i,\quad\mu_i|\cdot|^{-1}
 \qquad(1\le i\le b).
\]
Fix \(\mu_1\), and let \(\mc E\) be the genuine Borel principal series
formed from the remaining support.  With the general-linear data
understood through the fixed \(\psi\)-dependent genuine identification,
the quotient
\[
 \mu_1|\cdot|\times\mu_1\times\mu_1|\cdot|^{-1}
 \twoheadrightarrow\mu_1(3)
\]
and the same \(I_B\)-argument as in Case~(i) show that \(J\) is an
irreducible subquotient of
\[
 \mu_1(3)\rtimes\mc E.
\]
Moreover, \((\mu_1(3))_\psi\) has zero Whittaker module.  Rodier's
heredity theorem \cite{Rod} and exactness of the Whittaker functor
therefore show that
\(J\) is non-generic, and in particular
\(\psi^{(1)}\)-non-generic.

The remainder of the proof is identical to Case~(i), with the linear
and metaplectic factors interchanged and Proposition~\ref{key3} used
in place of Proposition~\ref{key}.  The remaining \(a-1\) inducing
characters of \(\sigma|\cdot|^z\) provide the required moving family.
Taking \(n\mapsto n+a-1\) and \(l=r-1\), Proposition~\ref{key3} gives
the required vanishing.  Its range condition is
\[
 (n+a-1)-(r-1)=m+a\ge3,
\]
since \(m\ge1\) and \(a\ge2\).
\end{proof}
From now on, we simply write $\P$ (resp. $\P_a,\M_a, K_a$) for $\P_{n+a,a}$ (resp. $\P_{m+a,a},\M_{m+a,a}, K_{m+a}$). Note that $N_{n+a,r}=(\Hom(Y_a,X_r) \times \Hom(Y_a^*,X_r)) \cdot N_{n,r}$ and $N_{n+a,r} \cap \P= \Hom(Y_a^*,X_r) \cdot N_{n,r}$. Denote $\Hom(Y_a,X_r) \times \Hom(Y_a^*,X_r)$ by $V$ and let $dv$ be the Haar measure of $V$ such that $dn'=dvdn$, where $dn'$ (resp. $dn$) is the Haar measures of $N_{n+a,r}$ (resp. $N_{n,r}$).

Using Proposition~\ref{l1}, we can prove the following lemma.

\begin{lem}[cf. {\cite[Lemma~{6.3}]{Y}, \cite[Lemma~{9.6}]{H}}] \label{l2}
With the same notation and hypotheses as in Proposition~\ref{l1}, the
following holds.
\begin{enumerate}
\item When $\pi_1=\pi'$ and $\pi_2=\pi$,  then for all $\vi \in \mathcal{E}^1(\sigma,\pi),
\phi' \in \As_{\P_{n+a,a}}^{\sigma_{\psi} \boxtimes \pi'}(\Mp_{2n+2a})$
and $f\in \nu_{W_{m+a}}$,
\begin{align*}
\notag
&\mathcal{FJ}(\mathcal{E}^0(\phi'), \vi,f)
=
\\
& \int_{K_a}\int_{[V]} \left( \int_{\M_{a}(F)\bs \M_{a}(\A)^1}  \left( \int_{[N_{n,r}]}\phi'(n\wt{m}vk) \cdot \Theta_{\P_{a}}(f)\big((n,\wt{m})\cdot (v,k)\big) dn  \right) \varphi_{\P_{a}}(mk) dm \right) \cdot e^{\frac{H_{\P}(vk)}{2}}   dv dk.
\end{align*}

\item When $\pi_1=\pi$ and $\pi_2=\pi'$, then for all $\vi' \in \mathcal{E}^1(\sigma,\pi'),
\phi \in \As_{\P_{n+a,a}}^{ \sigma \boxtimes\pi}(\Sp_{2n+2a})$
and $f\in \nu_{W_{m+a}}$,
\begin{align*}
\notag
&\mathcal{FJ}(\mathcal{E}^0(\phi), \vi',f)
=
\\
& \int_{K_a}\int_{[V]} \left( \int_{\M_{a}(F)\bs \M_{a}(\A)^1}  \left( \int_{[N_{n,r}]}\phi(nmvk) \cdot \Theta_{\P_{a}}(f)\big((n,\wt{m})\cdot (v,k)\big) dn  \right) \varphi_{\P_{a}}'(\wt{m}k) dm \right) \cdot e^{\frac{H_{\P}(vk)}{2}}   dv dk.
\end{align*}
\end{enumerate}
\end{lem}

\begin{proof}
A proof is identical to that of \cite[Lemma 9.6]{H}. We omit the details.
\end{proof}

%The following lemma is essentially same with \cite[Lemma 9.7]{H}. %Recently, its proof has been corrected in \cite{GJBR}. However, their corrected proof is roundabout. We give a relatively short proof of it.

\begin{lem}[{\cite[Lemma 6.5]{Y}, \cite[Lemma 9.7]{H}}]\label{l3}
With the same notation and hypotheses as in Proposition~\ref{l1}, the
following holds.

\begin{enumerate}
\item When $\pi_1=\pi'$ and $\pi_2=\pi$, then the following is equivalent;
\begin{enumerate}
\item The Fourier--Jacobi period functional $\mc{FJ}(\pi',\pi,\nu_{W_m})$ is nonzero
\item There exist $\vi \in \mathcal{E}^1(\sigma,\pi),
\phi \in \As_{\P_{n+a,a}}^{\sigma_{\psi} \boxtimes \pi'}(\Mp_{2n+2a})$
and $f\in \nu_{W_{m+a}}$ such that
\beq\label{rec1} \int_{K_a}\int_{[V]} \left( \int_{\M_{a}(F)\bs \M_{a}(\A)^1}  \left( \int_{[N_{n,r}]}\phi(n\wt{m}vk) \cdot \Theta_{\P_{a}}(f)\big((n,\wt{m})\cdot (v,k)\big) dn  \right) \varphi_{\P_{a}}(mk) dm \right) \cdot e^{\frac{H_{\P}(vk)}{2}}   dv dk \ne 0.
\eeq
\end{enumerate}

\item When $\pi_1=\pi$ and $\pi_2=\pi'$, then the following is equivalent;
\begin{enumerate}
\item The Fourier--Jacobi period functional $\mc{FJ}(\pi',\pi,\nu_{W_m})$ is nonzero
\item  There exist $\vi' \in \mathcal{E}^1(\sigma,\pi'),
\phi \in \As_{\P_{n+a,a}}^{ \sigma \boxtimes\pi}(\Sp_{2n+2a})$
and $f\in \nu_{W_{m+a}}$ such that
\beq \label{req2}
 \int_{K_a}\int_{[V]} \left( \int_{\M_{a}(F)\bs \M_{a}(\A)^1}  \left( \int_{[N_{n,r}]}\phi(nmvk) \cdot \Theta_{\P_{a}}(f)\big((n,\wt{m})\cdot (v,k)\big) dn  \right) \varphi_{\P_{a}}(\wt{m}k) dm \right) \cdot e^{\frac{H_{\P}(vk)}{2}}   dv dk \ne 0.
\eeq

\end{enumerate}
\end{enumerate}
\end{lem}

\begin{proof} Though the proof is essentially same with \cite[Lemma 9.7]{H}, for its importance, we give the proof here.
Since the case $\pi_1=\pi  \text{ and } \pi_2=\pi'$ is similar, we only prove the case $\pi_1=\pi'  \text{ and } \pi_2=\pi$.
We first prove the $(a) \to (b)$ direction.

Put
\[
\Pi'= (\sigma  |\cdot|^{-1})_{\psi} \boxtimes \pi',
\quad
\Pi=\sigma |\cdot|^{\frac{1}{2}} \boxtimes \pi,
\quad
\Pi''=|\cdot|_{\psi^{-1}}^{\frac{1}{2}} \boxtimes \nu_{W_m}.
\]
We define a functional on $\Pi' \boxtimes \Pi \boxtimes \Pi''$ by
\[
l(\eta' \boxtimes \eta \boxtimes \eta'')
=\int_{\M_{a}(F)\bs \M_{a}(\A)^1}\eta(m) \left(\int_{N_{n,r}(F) \bs N_{n,r}(\A)}\eta'(nm)\eta''(nm)dn\right)dm
\]
\par

Let $(\Pi' \boxtimes \Pi \boxtimes \Pi'')^{\infty}$ be
the canonical Casselman-Wallach globalization of $\Pi' \boxtimes \Pi \boxtimes \Pi''$
realized in the space of smooth automorphic forms
without the $K_{\M_{n+a}} \times K_{\M_{m+a}} \times K_{\M_{m+a}}$-finiteness condition,
where $K_{\M_{i+a}}=K_{i+a} \cap \M_{i+a,a}(\A)$ for $i=n,m$
(cf. \cite{Ca}, \cite[Chapter 11]{Wa}).
Since cusp forms are bounded,
$l$ can be uniquely extended to a continuous functional on $(\Pi' \boxtimes \Pi \boxtimes \Pi'')^{\infty}$
and denote it by the same notation.
Our assumption enables us to choose $\eta' \in \Pi', \eta \in \Pi$ and $\eta^{\prime \prime} \in \Pi''$
so that $l(\eta' \boxtimes \eta \boxtimes \eta'')\ne 0$.
We may assume that $\eta', \eta$ and $\eta''$ are pure tensors.
By \cite{LS,BS}, the functional $l$ is a product of local functionals
$l_v \in \mathrm{Hom}_{\M_{a,v}}
((\Pi_{v}' \boxtimes \Pi_v \boxtimes \Pi^{\prime \prime}_v)^{\infty},\mathbb{C})$,
where we set
$(\Pi_v' \boxtimes \Pi_v \boxtimes \Pi_v^{\prime \prime})^{\infty}
=\Pi_v' \boxtimes \Pi_v \boxtimes \Pi_v^{\prime \prime}$
if $v$ is finite.
Then we have $l_v(\eta_v' \boxtimes \eta_v \boxtimes \eta_v'')\ne 0$.
\par

Choose a finite set \(S\) containing the archimedean places and every
finite place at which the groups, the representations, \(\psi\), or
the chosen pure tensors are ramified. For \(v\notin S\), take
\(\eta_v'\), \(\eta_v\), and \(\eta_v''\) to be the normalized
spherical vectors and normalize \(l_v\) by
\(
l_v(\eta_v'\boxtimes\eta_v\boxtimes\eta_v'')=1.
\)

Denote by \(e\) the identity element of \(\Sp_{2m+2a}\). Choose
\(\varphi\in\mathcal E^1(\sigma,\pi)\) and a pure tensor
\(
f=f_1\otimes f_2\in\nu_{W_{m+a}}
=\Omega_{Y_a+Y_a^*}\otimes\nu_{W_m}
\)
such that
\begin{enumerate}
\item
\(
\delta_{\P_a}^{-1/2}\varphi_{\P_a}
=\boxtimes_v\varphi_v
\)
and
\(
\Theta_{\P_a}(f)=\boxtimes_v f_v
=\boxtimes_v\bigl(f_{1,v}(0)\otimes f_{2,v}\bigr);
\)
\item
\(
\varphi_v(e)=\eta_v
\)
and
\(
f_v((1,e))=\eta_v''.
\)
\end{enumerate}
Such choices are available by Remarks~\ref{eval} and~\ref{ind}; they
may be taken spherical outside \(S\).

We now make the local construction at \(v\in S\). Let
\(\U_{n+a,a}^{-}\) be the unipotent radical of the parabolic opposite
to \(\P_{n+a,a}\), and choose a smooth compactly supported function
\(\alpha_v\) on \(\U_{n+a,a,v}^{-}\), supported sufficiently near
the identity. On the open cell define a smooth section
\[
\phi_{v,0}\in
\Ind_{\wt\P_{n+a,a,v}}^{\Mp_{2n+2a,v}}(\Pi_v')^\infty
\]
by
\[
\phi_{v,0}(\wt m u u_-)
=\delta_{\P_v}(m)^{1/2}\alpha_v(u_-)
(\Pi_v')^\infty(\wt m)\eta_v',
\]
where \(m\in\M_{n+a,a,v}\), \(u\in\U_{n+a,a,v}\), and
\(u_-\in\U_{n+a,a,v}^{-}\). Thus the section takes its values in
\(\Pi_v'\), while \(\varphi_v\) takes its values in \(\Pi_v\).

Since
\(
N_{n+a,r}\simeq
\bigl(\Hom(Y_a,X)\times\Hom(Y_a^*,X)\bigr)\rtimes N_{n,r},
\)
we regard both Hom spaces as subgroups of \(N_{n+a,r}\). By
\eqref{w4}, for \(p_1\in\Hom(Y_a,X)_v\),
\[
f_v((p_1,e))
=\bigl(\Omega_{Y_a+Y_a^*}(p_1)f_{1,v}\bigr)(0)\otimes f_{2,v}
=f_v((1,e))=\eta_v''.
\]
Moreover, \(\Hom(Y_a,X)_v\) lies in
\(\U_{n+a,a,v}^{-}\). Shrinking the support of \(\alpha_v\), we
obtain
\[
\int_{\Hom(Y_a,X)_v}
l_v\bigl(\phi_{v,0}(p_1)\boxtimes\eta_v\boxtimes\eta_v''\bigr)
e^{H_{\P}(p_1)/2}\,dp_1\ne0.
\]
By continuity there is a sufficiently small neighborhood \(N_v\) of
\(0\) in \(\Hom(Y_a^*,X)_v\) such that
\[
\int_{\Hom(Y_a^*,X)_v}\int_{\Hom(Y_a,X)_v}
l_v\bigl(
\phi_{v,0}(p_1p_2)\boxtimes\eta_v
\boxtimes f_v((p_1p_2,e))
\bigr)
e^{H_{\P}(p_1p_2)/2}\chi_{N_v}(p_2)\,dp_1\,dp_2
\ne0.
\]
Choose \(f_{1,v}\) so that, under the realization
\[
f_{1,v}(p_2)
=\bigl(\Omega_{Y_a+Y_a^*}(p_2)f_{1,v}\bigr)(0),
\]
its support is contained in \(N_v\). Hence
\begin{align}
\label{ii}
&\int_{\Hom(Y_a^*,X)_v}\int_{\Hom(Y_a,X)_v}
l_v\bigl(
\phi_{v,0}(p_1p_2)\boxtimes\eta_v
\boxtimes f_v((p_1p_2,e))
\bigr)
e^{H_{\P}(p_1p_2)/2}\,dp_1\,dp_2
\ne0.
\end{align}

For a local section
\(
\phi_v\in
\Ind_{\wt\P_{n+a,a,v}}^{\Mp_{2n+2a,v}}(\Pi_v')^\infty,
\)
put
\[
\mathcal I_v(\phi_v)
=\int_{K_{m+a,v}}\int_{\Hom(Y_a^*,X)_v}
\int_{\Hom(Y_a,X)_v}
l_v\bigl(
\phi_v(p_1p_2k)\boxtimes\varphi_v(k)
\boxtimes f_v((p_1p_2k,e))
\bigr)
e^{H_{\P}(p_1p_2k)/2}\,dp_1\,dp_2\,dk.
\]
The open-cell construction above, together with \eqref{ii}, allows
the support in the \(K_{m+a,v}\)-variable to be concentrated near
the identity, and therefore gives a smooth section on which
\(\mathcal I_v\) is nonzero. Since \(\mathcal I_v\) is continuous and
the \(K_{n+a,v}\)-finite vectors are dense in the local induced
representation, we may choose a \(K_{n+a,v}\)-finite section
\(\phi_v\) with \(\mathcal I_v(\phi_v)\ne0\). This completes the proof of the \((a)\to(b)\) direction.

The proof of $(b) \to (a)$ direction is almost immediate. From what we have seen in the above, we see that the Fourier--Jacobi period integral $\mathcal{FJ}(\pi', \pi,\nu_{W_m})$ is a partial inner period integral in (\ref{rec1}). Therefore, if $\mathcal{FJ}(\pi',\pi,\nu_{W_m})=0$, the integral (\ref{rec1}) is always zero.
\end{proof}

By combining Lemma \ref{l2} with Lemma \ref{l3}, we get the following reciprocal non-vanishing theorem.

\begin{thm}\label{rthm}
With the same notation and hypotheses as in Proposition~\ref{l1},
Then $\mathcal{FJ}(\pi_1,\pi_2,\nu_{W_m}) \ne 0$ is equivalent to $\mathcal{FJ}(\mathcal{E}^0( \sigma,\pi_1), \mathcal{E}^1(\sigma,\pi_2),\nu_{W_{m+a}}) \ne 0$.
\end{thm}

Now piecing together everything we have developed so far, we can prove our main theorem.

\begin{customproof1}
Since the case $\pi_1=\pi$ is similar, we only treat the case
$\pi_1=\pi'$; thus $M$ is the symplectic parameter carried by the genuine
factor $\pi_1=\pi'$, and $N$ is the orthogonal parameter carried by
$\pi_2=\pi$.  By Lemma~\ref{tempered-parameter-cuspidal}, $\pi_1$ and
$\pi_2$ are cuspidal.  Write
\[
N=\sigma_1 \boxplus \cdots \boxplus \sigma_t,
\]
where $\sigma_1,\dots,\sigma_t$ are distinct irreducible cuspidal
automorphic representations of general linear groups.  Since $N$ is a
discrete tempered global $A$-parameter of orthogonal type, each
$\sigma_i$ is of orthogonal type, so that the symmetric-square
$L$-function $L(z,\sigma_i,\Sym^2)$ has a pole at $z=1$.  On the other
hand, $\mathcal{E}^0(\sigma_i,\pi')$ is nonzero by Theorem \ref{rthm}.
Thus by Proposition \ref{p4}, we have $L\left(\frac{1}{2},M \times \sigma_i \right)\ne 0 $ for all $1 \le i \le t$
and we get
\[
L\left(\frac{1}{2},M \times N\right)
= \prod_{i=1}^tL\left(\frac{1}{2},M \times \sigma_i  \right)\ne 0.
\] \hfill\(\square\)
\end{customproof1}

\begin{rem}
\label{miyawaki}
We record here the application of Theorem~\ref{e} to Miyawaki lifting
announced in the introduction.  In~\cite{At0}, Atobe recast the theory of
Miyawaki lifting in representation-theoretic terms and proved a
nonvanishing criterion for the generalized Miyawaki
lift~\cite[Theorem~5.5]{At0} under the following three assumptions:

\begin{itemize}
\item (i) $\to$ (ii) direction of the tempered GGP conjecture for $\Mp_{2n+2}(\mathbb{A}) \times \Mp_{2n}(\mathbb{A})$
\item (ii) $\to$ (i) direction of the tempered GGP conjecture for $\Mp_{2n}(\mathbb{A}) \times \Mp_{2n}(\mathbb{A})$
\item the generic summand conjecture for the Fourier--Jacobi coefficients (see \cite[Hypothesis 5.3 (A)]{At0}).
\end{itemize}
The first of these is the corank-one case of Theorem~\ref{e}, which
therefore sharpens \cite[Theorem 5.5]{At0} by removing it.
\end{rem}

\section{\textbf{Non-tempered GGP conjecture}}\label{sec:nontemp}
In this section, we
formulate the non-tempered GGP conjecture and prove one direction of it
for the members of the associated global \(L\)-packets of certain
non-tempered \(A\)-parameters. We first fix the relevant
\(L\)-functions, recalling from Subsection~\ref{dgp} the packet notation
that we shall use, and then treat the residual and the cuspidal
realizations of these associated global \(L\)-packets in turn.

\subsection{Non-tempered GGP conjecture} Let $\pi_1$ and $\pi_2$ be cuspidal automorphic representations of $\GL_n(\A)$ and $\GL_m(\A)$, respectively. Let $d_1,d_2$ be non-negative integers.

Define \begin{align*}
L(z, \pi_1 \boxtimes [d_1]) &\coloneqq \prod_{i=0}^{d_1} L\left(z + \frac{d_1}{2} - i, \pi_1\right) \\
L(z, (\pi_1 \times \pi_2) \boxtimes [d_1]) &\coloneqq \prod_{i=0}^{d_1} L\left(z + \frac{d_1}{2} - i, \pi_1 \times \pi_2\right) \\
L(z, (\pi_1 \boxtimes [d_1]) \times (\pi_2 \boxtimes [d_2])) &\coloneqq \prod_{k=0}^{\min(d_1, d_2)} L\bigl(z, (\pi_1 \times \pi_2) \boxtimes [d_1 + d_2 - 2k]\bigr)
\end{align*}

Here, $L(z,\pi_1)$ and $L(z,\pi_1 \times \pi_2)$ denote the completed standard automorphic and Rankin--Selberg $L$-functions, respectively.

For two discrete global $A$-parameters $M$ and $N$, decompose $M$ and $N$ as:
\( M = \sum_{\alpha} V_{\alpha}, \quad N = \sum_{\beta} W_{\beta} \)
where $V_{\alpha}$ and $W_{\beta}$'s are irreducible $A$-parameters.

Then we define:
\[ L(z,M) \coloneqq \prod_{\alpha} L(z,V_{\alpha}), \quad L(z,M \times N) \coloneqq \prod_{\alpha, \beta} L(z,V_{\alpha} \times W_{\beta}). \]

These $L$-functions of parameters agree with the analytic $L$-functions of
the automorphic representations that the parameters govern.  We shall use
the following form of the comparison, which is obtained by combining the
Lapid--Rallis theory of doubling $\gamma$-factors with the
multiplicativity of the doubling $L$-factors.

\begin{prop}[{\cite[Proposition~A.1]{HK}}]
\label{Leq}
Let $G_k$ be one of the groups $\SO_{2k}$, $\Sp_{2k}$, $\SO_{2k+1}$, or
$\Mp_{2k}$; in the metaplectic case, the packets and the local factors
are the ones attached to the fixed additive character $\psi$.  Let
\[
M=\bigoplus_{i=1}^{r}M_i\boxtimes[d_i],
\qquad d_i\ge0,
\]
be a discrete global $A$-parameter of $G_k$, and assume that $M_i$ is
tempered for every $i$ with $d_i\ge1$.  Then, for every $\pi\in\Pi_M$
and every automorphic character $\chi$ of $\GL_1(\A)$,
\[
L^{\mathrm{std}}(z,\pi\times\chi)=L(z,M\times\chi),
\]
where the left-hand side is the completed standard $L$-function of $\pi$
defined by the doubling method.
\end{prop}

% To state the non-tempered GGP conjecture, we should introduce the notion of 'relevant pair' of $A$-parameter.
\begin{defn}\label{drel}
Given two discrete global $A$-parameters $M$ and $N$, we can write $M$ and $N$ as follows:
\begin{align*}
M &= \sum_{i=1}^t M_i \boxtimes [n_i] \\
N &= \sum_{i=1}^t M_i \boxtimes [m_i],
\end{align*}
where $M_i$ is an irreducible unitary cuspidal automorphic representation of $\GL_{r_i}$ and $n_{i}, m_{i}$ are integers greater than or equal to $-1$. (Note that $[-1] = 0$.)

We say that $M$ and $N$ are relevant if there exists a permutation $p$ of $\{1,2,\ldots, t\}$ such that
\( M = \sum_{i=1}^t M_i \boxtimes [n_{p(i)}] \)
and $|n_{p(i)} - m_i| = 1$ for all $1 \leq i \leq t$. In this case, one of $M$ and $N$ is symplectic, and the other one is orthogonal.
\end{defn}

Suppose that $M$ and $N$ are relevant discrete $A$-parameters.

When $M$ is symplectic and $N$ is orthogonal, define $L(z,M,N)$ as
\[ L(z,M,N) \coloneqq \frac{L(z+\frac{1}{2},M \times N)}{L(z+1,\Sym^2(M)) \cdot L(z+1,\wedge^2(N))}.\]
When \(M\) is orthogonal and \(N\) is symplectic, we use the
opposite-orientation convention
\[
L(z,M,N)
\coloneqq
\frac{L(z+\frac12,M\times N)}
{L(z+1,\wedge^2(M))\cdot L(z+1,\Sym^2(N))}.
\]
In both cases, it is proved that $L(z,M,N)$ is holomorphic at \(z=0\) (see \cite[Theorem~9.7]{Gan1}).

\subsubsection{Component groups and the multiplicity formula}
\label{compgrp}

The arguments of Subsections~\ref{FJR} and~\ref{gen} use the internal
structure of the packets recalled in Subsection~\ref{locpack}.  We now
describe it, following \cite[Sect.~1.4]{Ae}, \cite[Sect.~5.3]{Li} and
\cite[Appendix~B]{At0}.

Let \(\theta\) be a local or global \(A\)- or \(L\)-parameter for a
classical group \(G\), let \(S_\theta\) be the centralizer of the image
of \(\theta\) in the dual group \(\widehat G\), and put
\[
 \mathcal S_\theta
 =
 \pi_0\bigl(S_\theta/Z(\widehat G)^{\Gamma}\bigr).
\]
Write \(\theta=\bigoplus_i\theta_i\) as a sum of pairwise inequivalent
irreducible summands.  A summand \(\theta_i\) that is self-dual of the
same type as the form defining \(\widehat G\) contributes to
\(S_\theta\) the element \(a_i\) acting by \(-1\) on \(\theta_i\) and
trivially on the remaining summands, whereas a pair of mutually dual
non-self-dual summands contributes a connected general linear factor and
hence nothing to \(\mathcal S_\theta\).  Thus \(\mathcal S_\theta\) is an
elementary abelian \(2\)-group generated by the classes of the \(a_i\),
subject to the single relation imposed by \(Z(\widehat G)^{\Gamma}\) and
by the determinant condition defining \(\widehat G\); we refer to
loc.\ cit.\ for the normalizations, which we shall never need in
explicit form.  In particular a discrete global \(A\)-parameter
\(M=\bigoplus_iM_i\boxtimes[d_i]\) has one generator \(a_i\) for each
summand \(M_i\boxtimes[d_i]\), and localization at \(v\) gives a
homomorphism
\[
 \mathcal S_M\longrightarrow\mathcal S_{M_v},
 \qquad
 s\longmapsto s_v,
\]
matching the generators attached to corresponding summands.

For a local \(A\)-parameter \(Y_v\), denote by
\[
q_{Y_v}^{A/L}:\mathcal S_{Y_v}\longrightarrow\mathcal S_{\phi_{Y_v}}
\]
the natural homomorphism induced by passage from \(Y_v\) to its
associated \(L\)-parameter.

Finally, the character \(\epsilon_M^{G}\) of \(\mathcal S_M\) occurring
in the multiplicity formula \eqref{globalApacket} is Arthur's sign
character \(\epsilon_M^{\mathrm{Art}}\) when \(G\) is symplectic or
special orthogonal, and
\(\epsilon_M^{\mathrm{Mp}}=\epsilon_M^{\mathrm{Art}}\nu_M\), with
\(\nu_M\) as in \cite[Definition~5.3.1]{Li}, when \(G\) is metaplectic.

With this notation in place, the non-tempered GGP conjecture for the
Fourier--Jacobi case is Conjecture~\ref{conn} of the introduction; see
\cite[Conjecture~9.1]{Gan1} for the Bessel case.

\begin{rem}For a symplectic tempered $A$-parameter $M$, it is easy to verify that $L(1,\Sym^2(M)) \ne 0$.
Similarly, one can check that $L(1,\wedge^2(N)) \ne 0$ for an orthogonal tempered $A$-parameter $N$.
Therefore, Conjecture~\ref{conn} implies the tempered GGP conjecture.
\end{rem}

To prove one direction of this conjecture for certain non-tempered cases, we need some analytic properties of automorphic $L$-functions.
We briefly review these properties (see \cite{JSS}, \cite{Sha1}, and \cite[Theorem~1.3]{Gr}).

Let \(\sigma_1\) and \(\sigma_2\) be irreducible cuspidal automorphic representations of \(\GL_n(\mathbb{A})\) and \(\GL_m(\mathbb{A})\), respectively. We denote by \(L(z,\gamma(\sigma_1))\) the complete symmetric or exterior square \(L\)-function according to  $\gamma=\Sym^2\text{ or }\wedge^2$. We review some analytic properties of \(L(z,\sigma_1 \times \sigma_2)\) and \(L(z,\gamma(\sigma_1))\).

\begin{itemize}
\item $L(z,\sigma_1 \times \sigma_2)$ and $L(z,\gamma(\sigma_1))$ have meromorphic continuations to the whole complex plane.
\item $L(z,\sigma_1 \times \sigma_2)$ and $L(z,\gamma(\sigma_1))$ satisfy functional equations relating the values at $s$ and $1-s$.
\item $L(z,\sigma_1 \times \sigma_2)$ and $L(z,\gamma(\sigma_1))$ are holomorphic for all $s \in \mathbb{C} \setminus \{0,1\}$.
\item For $\Re(s) \geq 1$ and $\Re(s) \leq 0$, $L(z,\sigma_1 \times \sigma_2)$ and $L(\gamma(\sigma_1),s)$ are non-zero.
\item $L(z,\sigma_1 \times \sigma_2)$ has a pole at $z=0,1$ if and only if $m=n$ and $\sigma_2=\sigma_1^{\vee}$ (the contragredient dual of $\sigma_1$). In this case, the pole is a simple pole.
\item If $\sigma_1$ is not self-dual, then $L(z,\gamma(\sigma_1))$ is an entire function.
\item If $\sigma_1$ is self-dual, exactly one of $L(z,\Sym^2(\sigma_1))$ and $L(z,\wedge^2(\sigma_1))$ has a simple pole at $z=0,1$.
\end{itemize}

On the other hand,  we freely use the following elementary facts.
\begin{fac} \label{fac}
Let $M$ and $N$ be a relevant pair of discrete $A$-parameters which has a decomposition
\( M = \bigoplus_{\alpha} V_{\alpha}, \quad N = \bigoplus_{\alpha} W_{\alpha} \)
so that each \((V_{\alpha},W_{\alpha})\) is a relevant pair of irreducible $A$-parameters. Assume that $M$ is symplectic and $N$ is orthogonal. Let $\rho$ and $V$ be representations of $\GL_k(\A)$ and $\SL_2(\CC)$, respectively.
Then:
\begin{itemize}
\item $M \otimes N = \sum_{\alpha} (V_{\alpha} \otimes W_{\alpha}) + \sum_{\alpha \neq \beta} (V_{\alpha} \otimes W_{\beta})$
\item $\Sym^2(M)=\bigoplus_\alpha\Sym^2(V_\alpha)\boxplus\bigoplus_{\alpha<\beta}(V_\alpha\otimes V_\beta)$
\item $\wedge^2(N)=\bigoplus_\alpha\wedge^2(W_\alpha)\boxplus\bigoplus_{\alpha<\beta}(W_\alpha\otimes W_\beta)$
\item $\Sym^2([i]) = [2i] + [2i-4] + \cdots$
\item $\wedge^2([i]) = [2i-2] + [2i-6] + \cdots$
\item $\Sym^2(\rho \boxtimes V) = \Sym^2(\rho) \boxtimes \Sym^2(V) + \wedge^2(\rho) \boxtimes \wedge^2(V)$

\item $\wedge^2(\rho \boxtimes V) = \Sym^2(\rho) \boxtimes \wedge^2(V) + \wedge^2(\rho) \boxtimes \Sym^2(V)$
\item $[i] \otimes [j] = \bigoplus_{k=0}^{\min(i,j)}[i+j-2k]$
\end{itemize}
\end{fac}

With the conjecture and the required analytic facts in place, we now
turn to the implication from period nonvanishing to the nonvanishing
of the extended central \(L\)-value. The argument separates naturally
according to whether the associated global \(L\)-packet is realized
in the residual spectrum or, in the cuspidal spectrum.

\subsection{The Fourier--Jacobi period for associated global \(L\)-packets}
\label{FJR}

In this subsection, we prove the direction
\((1)\Rightarrow(2)\) in Conjecture~\ref{conn}~(ii) for the
automorphic members of certain associated global
\(L\)-packets attached to relevant non-tempered \(A\)-parameters.
The argument has two parts. We first show that, once the relevant
residual representation exists, every automorphic member of the
associated global \(L\)-packet is residual. We then specialize this
packet-theoretic result to a relevant pair and prove the GGP
implication under the required central nonvanishing hypothesis.

\subsubsection{Residual realization of associated global
\(L\)-packets}
\label{residual-associated-packets}

We begin with local and global realization results that apply
uniformly to the two types of non-tempered parameters occurring
below. Let \(d\in\{1,2\}\), let \(\sigma\) be an irreducible unitary
cuspidal automorphic representation of \(\GL_a(\A)\), and let \(X\)
be a discrete tempered global \(A\)-parameter of the appropriate
type. Put
\begin{equation}
\label{general-parameters-Xd}
X_d=\sigma\boxtimes[d]\boxplus X,
\qquad
X_d^{-}=\sigma\boxtimes[d-2]\boxplus X.
\end{equation}
As in Subsection~\ref{dgp}, we understand \([-1]=0\). Thus
\(X_1^{-}=X\), whereas \(X_2^{-}=\sigma\boxplus X\).

Let \(G_k\) denote either \(\Sp_{2k}\) or \(\Mp_{2k}\). We assume
that \(X_d^{-}\) is a discrete tempered global \(A\)-parameter for
\(G_k\), and that \(X_d\) is a discrete global \(A\)-parameter for
the group \(G_{k+a}\) of the same type. In particular, when \(d=2\),
the parameter \(X_2^{-}=\sigma\boxplus X\) is assumed to be
discrete.

By the definition of the associated global \(L\)-parameter recalled
above, the summand \([d]\) contributes the characters
\(|\cdot|^{d/2},|\cdot|^{d/2-1},\ldots,|\cdot|^{-d/2}\), whose
middle terms are precisely those contributed by \([d-2]\). Hence
\begin{equation}
\label{associated-parameter-Xd}
\phi_{X_d}
=
\sigma|\cdot|^{d/2}
\boxplus\phi_{X_d^{-}}
\boxplus\sigma^\vee|\cdot|^{-d/2}.
\end{equation}
Since \(X_d\) is a discrete parameter of the required type,
\(\sigma\) is self-dual. We retain \(\sigma^\vee\) in
\eqref{associated-parameter-Xd} to display the structure of the
associated \(L\)-parameter.

Lemma~\ref{automatic-almost-tempered} supplies the local bounds needed
in every global application below. We now identify the members of the
associated local \(L\)-packet in terms of their middle tempered data and
set up the standard modules that will realize them.

For \(\rho_v\in\Pi_{X_{d,v}^{-}}\), denote the corresponding member
of the associated local \(L\)-packet by
\[
\pi_{G,\psi_v}(\phi_{X_d,v},\rho_v)
=
\begin{cases}
\pi(\phi_{X_d,v},\rho_v),&G_k=\Sp_{2k},\\
\pi_{\psi_v}(\phi_{X_d,v},\rho_v),&G_k=\Mp_{2k}.
\end{cases}
\]
For a representation \(\tau_v\) of a general linear group over
\(F_v\), put
\[
(\tau_v)_{\psi_v,G}
=
\begin{cases}
\tau_v,&G_k=\Sp_{2k},\\
(\tau_v)_{\psi_v},&G_k=\Mp_{2k},
\end{cases}
\qquad
\P_{k+a,a}^{G}
=
\begin{cases}
\P_{k+a,a},&G_k=\Sp_{2k},\\
\widetilde{\P}_{k+a,a},&G_k=\Mp_{2k}.
\end{cases}
\]
Thus \(\P_{k+a,a}^{G}(F_v)\) denotes the corresponding local
parabolic subgroup, and \(\P_{k+a,a}^{G}(\A)\) denotes its global
counterpart.
For \(z\in\CC\), define
\[
I_{d,v}(z,\rho_v)
=
(\sigma_v|\det|_v^z)_{\psi_v,G}\rtimes\rho_v,
\qquad
z_d=\frac d2.
\]
When \(G_{k+a}=\Mp_{2k+2a}\), the notation \(G_{k+a}(F_v)\)
refers to the metaplectic covering group. We write
\[
N_{d,v}(z,\rho_v):
I_{d,v}(z,\rho_v)\longrightarrow I_{d,v}(-z,\rho_v)
\]
for the normalized local intertwining operator associated with the
nontrivial Weyl element of \(\P_{k+a,a}^{G}\). These operators are
normalized compatibly with induction in stages and with the
restricted tensor-product construction described below. Using
compatible local isomorphisms
\(\sigma_v\simeq\sigma_v^\vee\), we identify the target of the
standard intertwining operator with \(I_{d,v}(-z,\rho_v)\). When
the inducing representation is clear, we write \(N_{d,v}(z)\).
In the covering case, we choose at every place a family of
normalizing factors satisfying the conditions of
\cite[Definition~3.1.1 and Theorems~3.2.1, 3.3.1]{LiTrace},
compatibly with induction in stages. At the unramified places we
impose the spherical normalization and use the resulting restricted
tensor product globally. The comparison with the normalization used
in Section~4 is absorbed into the scalar meromorphic function
\(r_d(z,\rho)\) below.

For a standard module \(S\), write \(\operatorname{LQ}(S)\) for its
unique irreducible Langlands quotient, \(S^{\mathrm{opp}}\) for its
opposite induced module.

\begin{lem}
\label{associated-packet-middle-datum}
For every place \(v\), the two outer summands in
\eqref{associated-parameter-Xd} contribute only connected general
linear factors to the centralizer. Consequently, projection to the
middle summand induces a canonical isomorphism
\(
\iota_v:\mathcal S_{\phi_{X_d,v}}
\overset{\sim}{\longrightarrow}\mathcal S_{X_{d,v}^{-}}.
\)
Abbreviate \(q_{X_{d,v}}^{A/L}\) to \(q_v^{A/L}\), and put
\(
q_v=\iota_v\circ q_v^{A/L}:
\mathcal S_{X_{d,v}}
\longrightarrow\mathcal S_{X_{d,v}^{-}}.
\)

Moreover, the map
\(
\rho_v\longmapsto
\pi_{G,\psi_v}(\phi_{X_d,v},\rho_v)
\)
is a bijection from \(\Pi_{X_{d,v}^{-}}\) onto
\(\Pi_{\phi_{X_d,v}}\). Then
\begin{equation}
\label{eq:local-internal-compatibility-Xd}
\left\langle q_v^{A/L}(s),\,
\pi_{G,\psi_v}(\phi_{X_d,v},\rho_v)
\right\rangle_{\phi_{X_d,v},G}
=
\left\langle q_v(s),\rho_v
\right\rangle_{X_{d,v}^{-},G}
\end{equation}
for every \(s\in\mathcal S_{X_{d,v}}\).
\end{lem}

\begin{proof}
The standard-module classification applied to
\eqref{associated-parameter-Xd} gives the asserted bijection. The
outer pair
\(\sigma_v|\cdot|^{d/2}\oplus
\sigma_v^\vee|\cdot|^{-d/2}\)
is a non-self-dual hyperbolic pair. It therefore contributes a
general linear factor to the centralizer, which is connected. If
an isotypic constituent is already present in the middle datum, the
effect is only to enlarge the same connected general linear factor.
This proves the component-group isomorphism. Compatibility with the
internal parametrizations is
\cite[Proposition~7.4.1]{Ae} in the symplectic case. In the
metaplectic case it follows from the standard-module construction
in \cite[Definition~6.2.1]{Li} and the unique compatible inclusion
of the associated \(L\)-packet into the multiplicity-free part of
the \(A\)-packet in \cite[Proposition~6.3.4]{Li}.
\end{proof}

Thus the middle datum parametrizes the entire associated local
\(L\)-packet. To realize these members in residual representations,
we next control the normalized intertwining operators at the
relevant residual points; the only additional input is the following
rank-one regularity statement.

If \(\delta\) and \(\tau\) are irreducible tempered representations
of \(\GL_p(F_v)\) and \(\GL_q(F_v)\), respectively, denote by
\[
R_{\delta,\tau}(u):
\delta|\det|^{u/2}\times\tau|\det|^{-u/2}
\longrightarrow
\tau|\det|^{-u/2}\times\delta|\det|^{u/2}
\]
the normalized rank-one intertwining operator. All normalized
intertwining operators below are chosen compatibly with induction
in stages and with the restricted tensor-product construction.

\begin{lem}
\label{rankoneregularity}
The operator \(R_{\delta,\tau}(u)\) is holomorphic and nonzero for
every real \(u>-1\).
\end{lem}

\begin{proof}
For \(u>0\), the unnormalized standard intertwining integral for
tempered inducing data is absolutely convergent in the open
positive chamber. The normalizing scalar has neither a zero nor a
pole there by \cite[Theorem~2.1, condition {\rm (R7)}]{Arth}.
Thus the normalized operator is holomorphic and nonzero. At
\(u=0\), the same conclusion follows from the unitary-axis
normalization and the functional equation.

Suppose that \(-1<u\leq0\). Normalized parabolic induction from
tempered representations of general linear groups is irreducible
throughout the strip \(|u|<1\). In the non-archimedean case, this
follows from the Zelevinsky classification and the criterion for
linked segments; see \cite{Z}. In the archimedean case, it follows
from the classification of complementary series for general linear
groups; see \cite{Vogan}. The reverse operator
\(R_{\tau,\delta}(-u)\) is holomorphic on the closed positive
chamber, including the unitary point \(u=0\), and is a nonzero map
between irreducible representations; it is therefore an
isomorphism. The functional equation
\(R_{\tau,\delta}(-u)\circ R_{\delta,\tau}(u)=\operatorname{id}\)
then gives
\(R_{\delta,\tau}(u)=R_{\tau,\delta}(-u)^{-1}\).
\end{proof}

Fix a place \(v\), and assume that \(\sigma_v\) is generic and
almost tempered. Write
\[
\sigma_v\simeq
\delta_{1,v}|\det|^{x_{1,v}}
\times\cdots\times
\delta_{r_\sigma,v}|\det|^{x_{r_\sigma,v}},
\qquad
\frac12>x_{1,v}\geq\cdots\geq x_{r_\sigma,v}>-\frac12,
\]
where the \(\delta_{i,v}\) are irreducible tempered
representations. Since \(\sigma_v\) is generic, this normalized
induced representation is irreducible.

Let \(\rho_v\in\Pi_{X_{d,v}^{-}}\) be almost tempered, and write it
as the Langlands quotient of
\[
S_{\rho_v}
=
(\tau_{1,v}|\det|^{y_{1,v}})_{\psi_v,G}
\times\cdots\times
(\tau_{s,v}|\det|^{y_{s,v}})_{\psi_v,G}
\rtimes\rho_{0,v},
\qquad
\frac12>y_{1,v}\geq\cdots\geq y_{s,v}>0,
\]
where the \(\tau_{j,v}\) and \(\rho_{0,v}\) are tempered. In the
metaplectic case, \(\rho_{0,v}\) is genuine.
Whenever some of the exponents \(y_{j,v}\) are equal, we first
group all blocks with the same exponent into the corresponding
tempered representation of a larger Levi factor. In this way
\(S_{\rho_v}\) is understood as arising from the canonical
Langlands datum
\begin{equation}
\label{eq:canonical-Langlands-datum-rhov}
\bigl(\P_{\rho_v},\tau_v^{\mathrm{temp}},\lambda_v\bigr),
\qquad
\lambda_v\in
\bigl(\mathfrak a_{\P_{\rho_v}}^\ast\bigr)^+,
\end{equation}
where the superscript \(+\) denotes the open positive chamber.
The displayed block realization is recovered from
\eqref{eq:canonical-Langlands-datum-rhov} by induction in stages.

Put \(a_{i,v}^{(d)}=d/2+x_{i,v}\). Let
\(\{(\lambda_{\ell,v}^{(d)},c_{\ell,v}^{(d)})\}_{\ell}\) be an
ordering of
\(
\{(\delta_{i,v},a_{i,v}^{(d)})\}_{i}
\cup
\{(\tau_{j,v},y_{j,v})\}_{j}
\)
such that
\(c_{1,v}^{(d)}\geq\cdots\geq
c_{r_\sigma+s,v}^{(d)}>0\). Define
\[
\mathcal I_{X_d,v}^{\mathrm{std}}(\rho_v)
=
(\lambda_{1,v}^{(d)}|\det|^{c_{1,v}^{(d)}})_{\psi_v,G}
\times\cdots\times
(\lambda_{r_\sigma+s,v}^{(d)}
|\det|^{c_{r_\sigma+s,v}^{(d)}})_{\psi_v,G}
\rtimes\rho_{0,v}.
\]
This is the standard module attached to
\((\phi_{X_d,v},\rho_v)\), and
\begin{equation}
\label{local-standard-LQ-Xd}
\operatorname{LQ}
\left(\mathcal I_{X_d,v}^{\mathrm{std}}(\rho_v)\right)
\simeq
\pi_{G,\psi_v}(\phi_{X_d,v},\rho_v).
\end{equation}

When \(d=1\), the exponent
\(a_{i,v}^{(1)}=\frac12+x_{i,v}\) may be smaller than some
\(y_{j,v}\), so a nontrivial shuffle may be required. When \(d=2\),
however,
\begin{equation}
\label{separated-exponents-d2}
a_{i,v}^{(2)}=1+x_{i,v}>\frac12>y_{j,v}
\end{equation}
for every \(i,j\); hence the blocks arising from
\(\sigma_v|\det|\) already precede all positive-exponent blocks of
\(S_{\rho_v}\).

Let \(w_{X_d,v}\) be the long Weyl element associated with the Levi
subgroup of \(\mathcal I_{X_d,v}^{\mathrm{std}}(\rho_v)\), and put
\[
\mathcal I_{X_d,v}^{\mathrm{opp}}(\rho_v)
=
w_{X_d,v}\cdot\mathcal I_{X_d,v}^{\mathrm{std}}(\rho_v).
\]
With our convention, this opposite module is
\[
\big((\lambda_{1,v}^{(d)})^\vee
|\det|^{-c_{1,v}^{(d)}}\big)_{\psi_v,G}
\times\cdots\times
\big((\lambda_{r_\sigma+s,v}^{(d)})^\vee
|\det|^{-c_{r_\sigma+s,v}^{(d)}}\big)_{\psi_v,G}
\rtimes\rho_{0,v}.
\]
Denote the normalized long intertwining operator between these two
modules by
\[
R_{X_d,v}^{\mathrm{long}}:
\mathcal I_{X_d,v}^{\mathrm{std}}(\rho_v)
\longrightarrow
\mathcal I_{X_d,v}^{\mathrm{opp}}(\rho_v).
\]

\begin{prop}
\label{localintertwiningrealization}
Let \(d\in\{1,2\}\), and let \(X_d,X_d^{-}\) be as in
\eqref{general-parameters-Xd}. Fix a place \(v\). Assume that
\(\sigma_v\) is generic and almost tempered and that
\(\rho_v\in\Pi_{X_{d,v}^{-}}\) is almost tempered. Then
\(N_{d,v}(z,\rho_v)\) is holomorphic at \(z=z_d=d/2\). Moreover,
there exist \(G_{k+a}(F_v)\)-homomorphisms
\[
A_{d,v}:
\mathcal I_{X_d,v}^{\mathrm{std}}(\rho_v)
\longrightarrow I_{d,v}\left(\frac d2,\rho_v\right),
\qquad
B_{d,v}:
I_{d,v}\left(-\frac d2,\rho_v\right)
\longrightarrow\mathcal I_{X_d,v}^{\mathrm{opp}}(\rho_v)
\]
such that
\begin{equation}
\label{local-intertwining-realization-Xd}
B_{d,v}\circ N_{d,v}\left(\frac d2,\rho_v\right)\circ A_{d,v}
=
c_{d,v}R_{X_d,v}^{\mathrm{long}},
\qquad c_{d,v}\in\CC^\times.
\end{equation}
Consequently,
\[
\operatorname{Im}
\left(
B_{d,v}\circ N_{d,v}\left(\frac d2,\rho_v\right)\circ A_{d,v}
\right)
\simeq
\pi_{G,\psi_v}(\phi_{X_d,v},\rho_v).
\]
In particular,
\(\pi_{G,\psi_v}(\phi_{X_d,v},\rho_v)\) is an irreducible
subquotient of
\(\operatorname{Im}N_{d,v}(\frac d2,\rho_v)\).
\end{prop}

\begin{proof}
The proof of Proposition~\ref{localintertwiningrealization} is
inspired by the intertwining-operator method of
\cite[Sects.~3.3--3.4 and 6.4--6.5]{AGIKMS}, especially the passage
through an expanded standard module and the identification of the
relevant Langlands quotient by a normalized long intertwining
operator.  Since our residual points, the \(d=1\) shuffle, and the
metaplectic case are not covered there, we give the full argument.

We divide the proof into four steps.

\medskip
\noindent
\textbf{Step 1: Expanded standard-module realizations.}

Let \(q_{\rho_v}:S_{\rho_v}\twoheadrightarrow\rho_v\) be the
canonical Langlands quotient map, and let
\(S_{\rho_v}^{\mathrm{opp}}\) be the opposite standard module. Let
\(
\mathcal J_{\rho_v}:
S_{\rho_v}\longrightarrow S_{\rho_v}^{\mathrm{opp}}
\)
be the standard unnormalized long intertwining operator attached to
the canonical datum
\eqref{eq:canonical-Langlands-datum-rhov}.

Let \(r_{\rho_v}(\lambda)\) denote the meromorphic scalar normalizing
factor for the corresponding family of long intertwining operators,
with respect to the local normalization fixed above. By
\cite[Theorem~2.1, condition {\rm (R7)}]{Arth} in the linear case,
and by \cite[Definition~3.1.1, condition {\rm (R7)},
Sect.~3.2, and Theorem~3.3.1]{LiTrace} in the covering case,
\(r_{\rho_v}(\lambda)\) has neither a zero nor a pole at
\(\lambda=\lambda_v\). The use of the canonical datum is essential
here: it places \(\lambda_v\) in the open positive chamber even when
equal exponents occur in the original block presentation. Therefore,
\(
r_{\rho_v}(\lambda_v)\in\CC^\times.
\)

Accordingly, the normalized long intertwining operator at
\(\lambda=\lambda_v\) is well defined by
\(
R_{\rho_v}
\coloneqq
r_{\rho_v}(\lambda_v)^{-1}\mathcal J_{\rho_v}.
\)
Since \(r_{\rho_v}(\lambda_v)\) is nonzero, the operators
\(\mathcal J_{\rho_v}\) and \(R_{\rho_v}\) have the same kernel and
the same image.

At a non-archimedean place,
\cite[Theorem~3.2]{BJ}, applied to the trivial central extension
in the linear case and to the two-fold central extension
\[
1\longrightarrow\mu_2\longrightarrow\Mp_{2k}(F_v)
\longrightarrow\Sp_{2k}(F_v)\longrightarrow1
\]
in the metaplectic case, shows that the image of
\(\mathcal J_{\rho_v}\) is isomorphic to the Langlands quotient of
\(S_{\rho_v}\). Hence
\[
\operatorname{Im}\mathcal J_{\rho_v}
\simeq
\operatorname{LQ}(S_{\rho_v})
=
\rho_v.
\]
At an archimedean place the analogous assertion is the standard
intertwining realization in the Langlands classification, applied
to the underlying Harish--Chandra modules; see
\cite[Chapter~10]{Wa}. Consequently, there is an embedding
\(j_{\rho_v}:\rho_v\hookrightarrow S_{\rho_v}^{\mathrm{opp}}\),
which we normalize so that
\begin{equation}
\label{eq:rho-long-factorization}
R_{\rho_v}
=
j_{\rho_v}\circ q_{\rho_v}:
S_{\rho_v}\longrightarrow S_{\rho_v}^{\mathrm{opp}}.
\end{equation}

Put
\[
\widetilde I_{d,v}(z)
=
(\sigma_v|\det|^z)_{\psi_v,G}\rtimes S_{\rho_v},
\qquad
\widetilde I_{d,v}^{\mathrm{opp}}(-z)
=
(\sigma_v^\vee|\det|^{-z})_{\psi_v,G}
\rtimes S_{\rho_v}^{\mathrm{opp}}.
\]
By induction in stages, \(\widetilde I_{d,v}(z)\) is identified
with
\[
(\delta_{1,v}|\det|^{z+x_{1,v}})_{\psi_v,G}
\times\cdots\times
(\delta_{r_\sigma,v}|\det|^{z+x_{r_\sigma,v}})_{\psi_v,G}
\times
(\tau_{1,v}|\det|^{y_{1,v}})_{\psi_v,G}
\times\cdots\times
(\tau_{s,v}|\det|^{y_{s,v}})_{\psi_v,G}
\rtimes\rho_{0,v}.
\]
Exactness of normalized parabolic induction gives a surjection
\[
Q_{d,v}(z)=\operatorname{id}\rtimes q_{\rho_v}:
\widetilde I_{d,v}(z)\twoheadrightarrow I_{d,v}(z,\rho_v)
\]
and an injection
\[
J_{d,v}(z)=\operatorname{id}\rtimes j_{\rho_v}:
I_{d,v}(-z,\rho_v)
\hookrightarrow\widetilde I_{d,v}^{\mathrm{opp}}(-z).
\]

\medskip
\noindent
\textbf{Step 2: Regularity of the expanded operator.}

Let
\[
N_{d,v}(z,S_{\rho_v}):
\widetilde I_{d,v}(z)
\longrightarrow
(\sigma_v^\vee|\det|^{-z})_{\psi_v,G}\rtimes S_{\rho_v}
\]
be the normalized operator before passing to the quotient
\(\rho_v\), and put
\[
\widetilde N_{d,v}(z)
=
(\operatorname{id}\rtimes R_{\rho_v})
\circ N_{d,v}(z,S_{\rho_v}).
\]
Thus \(\widetilde N_{d,v}(z)\) maps
\(\widetilde I_{d,v}(z)\) to
\(\widetilde I_{d,v}^{\mathrm{opp}}(-z)\).

All factorizations in this step are taken relative to the common
refined block presentation obtained by expanding the
\(\sigma_v\)- and \(S_{\rho_v}\)-data and grouping equal exponents
within each datum separately; blocks from the two data are not
merged. Put
\(
\lambda_{d,v}
=
\bigl(
a_{1,v}^{(d)},\ldots,a_{r_\sigma,v}^{(d)},
y_{1,v},\ldots,y_{s,v}
\bigr),
\)
and write \(e_i\) and \(f_j\) for the block coordinates
corresponding to \(\delta_{i,v}\) and \(\tau_{j,v}\), respectively.
The Weyl element \(w\) underlying
\(\widetilde N_{d,v}(z_d)\) changes the signs of all these
coordinates: its \(e_i\)-part is contributed by
\(N_{d,v}(z_d,S_{\rho_v})\), and its \(f_j\)-part by
\(R_{\rho_v}\).

A direct calculation of
\(
\operatorname{Inv}(w)=\{\beta>0:w\beta<0\}
\)
shows that its part involving at least one \(e_i\) consists of the
block-root directions
\[
e_i-e_k,\quad e_i+e_k\quad(i<k),\qquad
e_i-f_j,\quad e_i+f_j,
\]
together with the terminal root directions supported on the
\(e_i\). The remaining inversion roots involve only the
\(f_j\)-blocks and the tempered classical factor, and their
rank-one factors are precisely those occurring in \(R_{\rho_v}\).
Hence a reduced expression of \(w\) crosses each of these root
directions exactly once.

For a crossed block root \(\beta\), the corresponding rank-one
factor depends only on the restriction of \(\lambda_{d,v}\) to
the coroot direction \(\beta^\vee\). In the present block
normalization, the four types displayed above give, respectively,
\[
a_{i,v}^{(d)}-a_{k,v}^{(d)},\qquad
a_{i,v}^{(d)}+a_{k,v}^{(d)},\qquad
a_{i,v}^{(d)}-y_{j,v},\qquad
a_{i,v}^{(d)}+y_{j,v}.
\]
The terminal rank-one normalizing factor has the remaining
arguments
\(a_{i,v}^{(d)}\) and \(2a_{i,v}^{(d)}\),
arising from the \(e_i\)- and \(2e_i\)-weight spaces in the
relevant unipotent radical. Finally, the ordering of the internal
blocks gives
\(
a_{i,v}^{(d)}-a_{k,v}^{(d)}\geq0
\qquad(i<k),
\)
so the corresponding rank-one factors lie in the closed positive
chamber.

If \(d=1\), all these parameters are nonnegative except possibly
\(a_{i,v}^{(1)}-y_{j,v}\), and
\(
a_{i,v}^{(1)}-y_{j,v}
=
\frac12+x_{i,v}-y_{j,v}>-\frac12.
\)
After removing the common twist
\(|\det|^{(a_{i,v}^{(1)}+y_{j,v})/2}\), the corresponding factor
is
\(R_{\delta_{i,v},\tau_{j,v}}
(a_{i,v}^{(1)}-y_{j,v})\), which is holomorphic and nonzero by
Lemma~\ref{rankoneregularity}. All remaining factors lie in the
closed positive chamber. For a strictly positive parameter, the
unnormalized standard integral is absolutely convergent and the
normalizing scalar is nonzero by condition {\rm (R7)}. On a
boundary wall one uses the unitary-axis normalization and the
functional equation. These assertions follow from
\cite[Theorem~2.1]{Arth} in the linear case and from
\cite[Definition~3.1.1 and Theorems~3.2.1 and~3.3.1]{LiTrace} in
the archimedean and non-archimedean covering cases, respectively.
This is the general tempered-inducing-data argument and does not
require the classical tempered factor \(\rho_{0,v}\) to be
generic; in the covering case the standard integral is defined
using the canonical splitting over the relevant unipotent
subgroups.
If \(d=2\),
\eqref{separated-exponents-d2} shows that every rank-one parameter
lies in the same region. Hence
\begin{equation}
\label{eq:expanded-holomorphic-Xd}
\widetilde N_{d,v}(z)
\text{ is holomorphic at }z=\frac d2.
\end{equation}

\medskip
\noindent
\textbf{Step 3: Descent of regularity to \(\rho_v\).}

There is a scalar \(b_{d,v}\in\CC^\times\), independent of \(z\),
such that
\begin{equation}
\label{eq:naturality-Ndv}
J_{d,v}(z)\circ N_{d,v}(z,\rho_v)\circ Q_{d,v}(z)
=
b_{d,v}\widetilde N_{d,v}(z)
\end{equation}
as meromorphic families. Indeed, in a chamber of absolute
convergence the unnormalized intertwining integral is natural with
respect to \(q_{\rho_v}\) and \(j_{\rho_v}\). Together with
\eqref{eq:rho-long-factorization}, this identifies the two sides
up to a fixed nonzero scalar. More explicitly, every
\(z\)-dependent rank-one normalizing factor on either side is
attached to a root meeting the same outer datum
\(\sigma_v|\det|^z\). Compatibility with induction in stages
identifies these factors for
\(N_{d,v}(z,\rho_v)\) and
\(N_{d,v}(z,S_{\rho_v})\). The only additional factor, the one
normalizing \(R_{\rho_v}\), is evaluated at the fixed point
\(\lambda_v\) and is therefore independent of \(z\). Hence the
quotient of the two normalizations is the constant \(b_{d,v}\);
it records only the fixed choices of Weyl representatives,
self-duality isomorphisms, and the normalization of
\eqref{eq:rho-long-factorization}. In the linear case this is
\cite[Theorem~2.1, properties {\rm (R2)} and {\rm (R3)}]{Arth};
in the covering case it is
\cite[Definition~3.1.1, properties {\rm (R1)}, {\rm (R3)}, and
{\rm (R5)}, Theorems~3.2.1 and~3.3.1]{LiTrace}.
Meromorphic continuation proves
\eqref{eq:naturality-Ndv} for all \(z\).

Suppose that \(N_{d,v}(z,\rho_v)\) has a pole at \(z_d=d/2\), and
write
\[
N_{d,v}(z,\rho_v)
=
\sum_{\ell=-q}^{\infty}(z-z_d)^\ell N_{d,v,\ell},
\qquad q>0,\quad N_{d,v,-q}\ne0.
\]
By \eqref{eq:expanded-holomorphic-Xd}, the right-hand side of
\eqref{eq:naturality-Ndv} is holomorphic at \(z_d\), so
\(
J_{d,v}(z_d)\circ N_{d,v,-q}\circ Q_{d,v}(z_d)=0.
\)
Since \(Q_{d,v}(z_d)\) is surjective and \(J_{d,v}(z_d)\) is
injective, this forces \(N_{d,v,-q}=0\), a contradiction. Thus
\(N_{d,v}(z,\rho_v)\) is holomorphic at \(z=d/2\). At an
archimedean place, the argument is carried out on the underlying
Harish--Chandra modules, where normalized parabolic induction is
exact.

\medskip
\noindent
\textbf{Step 4: Construction of \(A_{d,v}\) and \(B_{d,v}\).}

When \(d=1\), let
\[
R_{A,1,v}:
\mathcal I_{X_1,v}^{\mathrm{std}}(\rho_v)
\longrightarrow\widetilde I_{1,v}\left(\frac12\right)
\]
be the normalized reordering operator from the combined standard
order to the order in which the blocks from
\(\sigma_v|\det|^{1/2}\) precede those from \(S_{\rho_v}\), and let
\[
R_{B,1,v}:
\widetilde I_{1,v}^{\mathrm{opp}}\left(-\frac12\right)
\longrightarrow\mathcal I_{X_1,v}^{\mathrm{opp}}(\rho_v)
\]
be the corresponding operator on the opposite side. We choose the
order of blocks with equal exponents compatibly on the two sides.
Every rank-one factor in these reordering operators has nonnegative
relative parameter. The positive-parameter factors are holomorphic
and nonzero by the positive-chamber argument, while the
zero-parameter factors are invertible by unitary-axis regularity and
the functional equation. Thus both reordering operators are
holomorphic and nonzero. When
\(d=2\), \eqref{separated-exponents-d2} shows that the expanded
modules are already in the required orders; put
\(R_{A,2,v}=R_{B,2,v}=\operatorname{id}\).

Define
\[
A_{d,v}=Q_{d,v}(z_d)\circ R_{A,d,v},
\qquad
B_{d,v}=R_{B,d,v}\circ J_{d,v}(z_d).
\]
Together with \eqref{eq:naturality-Ndv}, these definitions are
summarized by the following commutative diagram; the top middle
arrow has been multiplied by \(b_{d,v}\) so that the middle square
commutes exactly.
\[
\begin{tikzcd}[column sep=2.7em,row sep=3.2em]
\mathcal I_{X_d,v}^{\mathrm{std}}(\rho_v)
  \arrow[r,"{R_{A,d,v}}"]
  \arrow[dr,"{A_{d,v}}"']
&
\widetilde I_{d,v}(z_d)
  \arrow[r,"{b_{d,v}\widetilde N_{d,v}(z_d)}"]
  \arrow[d,two heads,"{Q_{d,v}(z_d)}"']
&
\widetilde I_{d,v}^{\mathrm{opp}}(-z_d)
  \arrow[r,"{R_{B,d,v}}"]
&
\mathcal I_{X_d,v}^{\mathrm{opp}}(\rho_v)
\\
&
I_{d,v}(z_d,\rho_v)
  \arrow[r,"{N_{d,v}(z_d,\rho_v)}"']
&
I_{d,v}(-z_d,\rho_v)
  \arrow[u,hook,"{J_{d,v}(z_d)}"']
  \arrow[ur,"{B_{d,v}}"']
&
\end{tikzcd}
\]

After passing to a common block refinement of the relevant Levi
data, all operators in the top row are normalized
parabolic-to-parabolic operators. For \(d=1\), the two outside Weyl
elements are the mutually inverse stable shuffles, whereas the
middle Weyl element changes the signs of all block coordinates.
Hence their product is the sign-changing long Weyl element attached
to the refined standard datum; for \(d=2\), both shuffles are the
identity. No length-additivity assertion is needed, since normalized
operators satisfy transitivity for arbitrary intermediate
parabolics. Moreover, \(R_{A,d,v}\) and \(R_{B,d,v}\) are
holomorphic at the indicated parameters by the preceding paragraph,
and \(\widetilde N_{d,v}\) is holomorphic at \(z_d\) by Step~2.
The cocycle relation may therefore be evaluated at \(z_d\). Thus
\cite[Theorem~2.1, properties {\rm (R2)} and {\rm (R3)}]{Arth} in
the linear case and
\cite[Definition~3.1.1, properties {\rm (R1)}, {\rm (R3)}, and
{\rm (R5)}]{LiTrace} in the covering case give
\(e_{d,v}\in\CC^\times\) such that
\[
R_{B,d,v}\circ\widetilde N_{d,v}(z_d)\circ R_{A,d,v}
=
e_{d,v}R_{X_d,v}^{\mathrm{long}}.
\]
The scalar records only the fixed Weyl-conjugation and self-duality
identifications. Taking \(c_{d,v}=b_{d,v}e_{d,v}\) proves
\eqref{local-intertwining-realization-Xd}.

It remains to identify the image of the long operator. Let
\(
\bigl(
\P_{X_d,v},
\tau_{X_d,v}^{\mathrm{temp}},
\lambda_{X_d,v}
\bigr)
\)
be the canonical Langlands datum underlying
\(\mathcal I_{X_d,v}^{\mathrm{std}}(\rho_v)\), obtained by
grouping all blocks with equal exponents into tempered
representations of the corresponding larger Levi factors. Thus
\(
\lambda_{X_d,v}\in
\bigl(\mathfrak a_{\P_{X_d,v}}^\ast\bigr)^+.
\)
Denote by \(R_{X_d,v}^{\mathrm{grp}}\) the normalized long
intertwining operator attached to this grouped datum. Its
normalizing factor satisfies
\(
r_{X_d,v}(\lambda_{X_d,v})\in\CC^\times
\)
by \cite[Theorem~2.1, condition {\rm (R7)}]{Arth} in the linear
case and by
\cite[Definition~3.1.1, condition {\rm (R7)}, Sect.~3.2, and
Theorem~3.3.1]{LiTrace} in the covering case. Hence the normalized
and unnormalized grouped long operators have the same kernel and
image.

The refined block Levi used above and the canonical grouped Levi
need not coincide: for \(d=1\), an exponent
\(a_{i,v}^{(1)}\) may equal some \(y_{j,v}\). Passing from the
refined presentation to the grouped datum only inserts normalized
rank-one operators with relative parameter zero, together with
induction in stages. These zero-parameter operators are
isomorphisms by unitary-axis regularity and the functional equation.
Consequently, transitivity gives, under the resulting
identifications of the standard and opposite modules,
\[
U_{d,v}^{\mathrm{opp}}\circ R_{X_d,v}^{\mathrm{long}}
=
e'_{d,v}R_{X_d,v}^{\mathrm{grp}}\circ U_{d,v}^{\mathrm{std}},
\qquad e'_{d,v}\in\CC^\times,
\]
where \(U_{d,v}^{\mathrm{std}}\) and
\(U_{d,v}^{\mathrm{opp}}\) are isomorphisms. Thus the refined and
grouped long operators have corresponding kernels and images.

At a non-archimedean place,
\cite[Theorem~3.2]{BJ}, applied to the grouped canonical datum,
shows that the image of its unnormalized long operator is the
Langlands quotient. The same conclusion at an archimedean place is
the standard intertwining realization in the Langlands
classification; see \cite[Chapter~10]{Wa}. Therefore,
using \eqref{local-standard-LQ-Xd},
\[
\operatorname{Im}R_{X_d,v}^{\mathrm{long}}
\simeq
\operatorname{LQ}
\left(\mathcal I_{X_d,v}^{\mathrm{std}}(\rho_v)\right)
=
\pi_{G,\psi_v}(\phi_{X_d,v},\rho_v).
\]
This proves the asserted image isomorphism.

Finally,
\[
\operatorname{Im}
\bigl(B_{d,v}\circ N_{d,v}(z_d,\rho_v)\circ A_{d,v}\bigr)
=
B_{d,v}\bigl(\operatorname{Im}(N_{d,v}(z_d,\rho_v)\circ
A_{d,v})\bigr),
\]
and
\[
\operatorname{Im}(N_{d,v}(z_d,\rho_v)\circ A_{d,v})
\subseteq
\operatorname{Im}N_{d,v}(z_d,\rho_v).
\]
Hence \(\pi_{G,\psi_v}(\phi_{X_d,v},\rho_v)\) is a quotient of a
subrepresentation of
\(\operatorname{Im}N_{d,v}(d/2,\rho_v)\), and therefore an
irreducible subquotient of that image.
\end{proof}

Proposition~\ref{localintertwiningrealization} realizes the desired
local packet member as an irreducible subquotient of the image of
the normalized local operator. We now pass from these local
realizations to the global residual representation.

\begin{lem}
\label{localrealization}
Let \(d\in\{1,2\}\), and let \(X_d,X_d^{-}\) be as in
\eqref{general-parameters-Xd}. Let
\(\rho=\bigotimes_v'\rho_v\) be an irreducible cuspidal automorphic
representation of \(G_k(\A)\) with global parameter \(X_d^{-}\),
genuine when \(G_k=\Mp_{2k}\). Assume that
\(\mathcal E^{d-1}(\sigma,\rho)\ne0\). Then
\(
\pi_{X_d}(\rho)
:=
\bigotimes_v'
\pi_{G,\psi_v}(\phi_{X_d,v},\rho_v)
\)
occurs as an irreducible direct summand of
\(\mathcal E^{d-1}(\sigma,\rho)\).
\end{lem}

\begin{proof}
By Lemma~\ref{automatic-almost-tempered}, every \(\sigma_v\) is
generic and almost tempered and every \(\rho_v\) is almost
tempered. Thus
Proposition~\ref{localintertwiningrealization} gives, for every
\(v\), homomorphisms \(A_{d,v}\) and \(B_{d,v}\) such that
\[
\operatorname{Im}
\left(
B_{d,v}\circ N_{d,v}(z_d,\rho_v)\circ A_{d,v}
\right)
\simeq
\pi_{G,\psi_v}(\phi_{X_d,v},\rho_v).
\]
At almost all finite places, normalize all maps so that normalized
spherical vectors are carried to normalized spherical vectors.
Then \(c_{d,v}=1\) for almost all \(v\), and the restricted tensor
products \(A_d=\bigotimes_v'A_{d,v}\) and
\(B_d=\bigotimes_v'B_{d,v}\) are well-defined.

Put
\(I_d(z,\rho)=\bigotimes_v'I_{d,v}(z,\rho_v)\) and
\(N_d(z,\rho)=\bigotimes_v'N_{d,v}(z,\rho_v)\). The normalized
global operator acts on decomposable vectors as this restricted
tensor product, and hence
\[
\operatorname{Im}
\left(
B_d\circ N_d(z_d,\rho)\circ A_d
\right)
\simeq
\pi_{X_d}(\rho).
\]

Let \(M_d(z,\rho)\) be the global standard intertwining operator.
For a scalar meromorphic function \(r_d(z,\rho)\), one has
\(
M_d(z,\rho)=r_d(z,\rho)N_d(z,\rho).
\)
By Proposition~\ref{localintertwiningrealization} and the
restricted tensor-product construction,
\(
N_d(z_d,\rho)\ne0.
\)
Moreover, Propositions~\ref{p2}--\ref{p5} show that
\(M_d(z,\rho)\) has at most a simple pole at \(z=z_d\). Since
\(N_d(z,\rho)\) is holomorphic and nonzero at \(z_d\), the scalar
\(r_d(z,\rho)\) has at most a simple pole there. The nonvanishing
of \(\mathcal E^{d-1}(\sigma,\rho)\) forces this pole to be simple.
Hence
\[
\operatorname*{Res}_{z=z_d}M_d(z,\rho)
=
c_dN_d(z_d,\rho)
\qquad(c_d\in\CC^\times).
\]
Let
\(
\operatorname{res}_d:
I_d(z_d,\rho)\longrightarrow
\mathcal E^{d-1}(\sigma,\rho)
\)
be the residual Eisenstein map. By definition it is surjective.
The identity term in the constant term of the Eisenstein series is
holomorphic at \(z_d\). Hence taking the constant term along
\(\P_{k+a,a}^{G}\) gives
\begin{equation}
\label{eq:constant-term-residual-map-Xd}
\bigl(\operatorname{res}_d(f)\bigr)_{\P_{k+a,a}^{G}}
=
c_dN_d(z_d,\rho)f.
\end{equation}
In particular,
\(\ker(\operatorname{res}_d)\subset
\ker N_d(z_d,\rho)\). Thus \(N_d(z_d,\rho)\) descends to a
surjection
\[
\mathcal E^{d-1}(\sigma,\rho)
\twoheadrightarrow
\operatorname{Im}N_d(z_d,\rho).
\]
Since
\(
\operatorname{Im}(B_d\circ N_d(z_d,\rho)\circ A_d)
\simeq\pi_{X_d}(\rho),
\)
it follows that \(\pi_{X_d}(\rho)\) is an irreducible subquotient
of \(\mathcal E^{d-1}(\sigma,\rho)\).

The residues under consideration are square integrable by the
Langlands square-integrability criterion; compare
\cite[Lemma~I.4.11]{Mo}. Moreover, the residual space attached to
the fixed cuspidal datum has finite length; see the residual
spectral construction in \cite[Chapters~VI--VII]{Mo}. Therefore
\(\mathcal E^{d-1}(\sigma,\rho)\) is a finite-length unitary
representation and is semisimple. Every irreducible subquotient is
accordingly an irreducible direct summand. This proves the lemma.
\end{proof}

To apply this global realization uniformly to an automorphic member of
the associated packet, we must know that the recovered middle datum is
cuspidal; this is exactly Lemma~\ref{tempered-parameter-cuspidal}. We
next show that every automorphic member of the associated global
\(L\)-packet is realized in the corresponding residual representation.

\begin{prop}
\label{voganres}
Let \(d\in\{1,2\}\), and let \(X_d,X_d^{-}\) be as in
\eqref{general-parameters-Xd}. Assume that \(X_d^{-}\) is a
discrete tempered global \(A\)-parameter for \(G_k\). If \(d=1\),
assume in
addition that
\begin{equation}
\label{eq:general-central-assumption}
L\left(\frac12,\sigma\times X\right)\ne0.
\end{equation}

Let \(\Pi=\bigotimes_v'\Pi_v\) be an irreducible discrete
automorphic representation of \(G_{k+a}(\A)\) such that
\(\Pi_v\in\Pi_{\phi_{X_d,v}}\) for every \(v\). Then there exists
an irreducible cuspidal
automorphic representation
\(\rho=\bigotimes_v'\rho_v\) of \(G_k(\A)\), genuine in the
metaplectic case, with global parameter \(X_d^{-}\), such that
\(\Pi\) occurs as an irreducible direct summand of
\(\mathcal E^{d-1}(\sigma,\rho)\).
\end{prop}

\begin{proof}
By Lemma~\ref{associated-packet-middle-datum}, for every place
\(v\) there is a unique \(\rho_v\in\Pi_{X_{d,v}^{-}}\) such that
\begin{equation}
\label{eq:local-associated-member-voganres}
\Pi_v
\simeq
\pi_{G,\psi_v}(\phi_{X_d,v},\rho_v).
\end{equation}
Here and below, the subscript \(\psi_v\) is omitted in the
symplectic case.

The representation \(\rho_v\) is unramified for almost all \(v\),
so \(\rho=\bigotimes_v'\rho_v\) is a well-defined irreducible
admissible representation.

We first locate \(\Pi\) in the global discrete decomposition. By
Lemma~\ref{associated-packet-middle-datum}, every \(\Pi_v\) lies in
the multiplicity-free part of the local \(A\)-packet attached to
\(X_{d,v}\). Since \(\Pi\) is discrete automorphic, uniqueness of
the global Arthur parameter determined by the unramified Satake
parameters shows that \(\Pi\) has global \(A\)-parameter \(X_d\).
In the symplectic case this is part of Arthur's global decomposition
\cite[Theorem~1.5.2]{Ae}; in the metaplectic case it is precisely
\cite[Corollary~5.4.4(i)]{Li}, together with
\cite[Proposition~6.3.4]{Li}.

We now prove that \(\rho\) is automorphic. Retain the maps
\(q_v^{A/L}\), \(\iota_v\), and
\(q_v=\iota_v\circ q_v^{A/L}\) from
Lemma~\ref{associated-packet-middle-datum}.

Let
\(q:\mathcal S_{X_d}\to\mathcal S_{X_d^{-}}\) be the corresponding
global map. It is compatible with localization:
\begin{equation}
\label{eq:localization-component-map-Xd}
(q(s))_v=q_v(s_v),
\qquad s\in\mathcal S_{X_d}.
\end{equation}
The map \(q\) is surjective. If \(d=1\), it maps each generator
coming from \(X=X_1^{-}\) to the corresponding generator and kills
the new generator associated with \(\sigma\boxtimes[1]\). If
\(d=2\), it identifies the generator associated with
\(\sigma\boxtimes[2]\) with that associated with \(\sigma\) in
\(X_2^{-}=\sigma\boxplus X\), and is the identity on the
generators coming from \(X\); hence it is an isomorphism.

Define
\[
\eta_\Pi(s)
=
\prod_v
\left\langle
q_v^{A/L}(s_v),\Pi_v
\right\rangle_{\phi_{X_d,v},G},
\qquad
\eta_\rho(t)
=
\prod_v
\left\langle
t_v,\rho_v
\right\rangle_{X_{d,v}^{-},G}.
\]
Equations \eqref{eq:local-internal-compatibility-Xd} and
\eqref{eq:localization-component-map-Xd} imply
\begin{equation}
\label{eq:eta-compatibility-Xd}
\eta_\Pi=\eta_\rho\circ q.
\end{equation}
By the compatibility of the inclusion of the associated local
\(L\)-packet into the local \(A\)-packet, \(\eta_\Pi\) is precisely
the global packet character of \(\Pi\), viewed as a member of the
global \(A\)-packet attached to \(X_d\).

Recall from Subsection~\ref{compgrp} the character \(\epsilon_Y^{G}\) of
\(\mathcal S_Y\) occurring in the multiplicity formula
\eqref{globalApacket}. We claim that
\begin{equation}
\label{eq:multiplicity-character-compatibility-Xd}
\epsilon_{X_d}^G
=
\epsilon_{X_d^{-}}^G\circ q.
\end{equation}

Write \(X=\boxplus_{i=1}^{t}\tau_i\). Suppose first that \(d=1\).
Let \(a_\sigma\) be the generator associated with
\(\sigma\boxtimes[1]\), and \(a_i\) the generator associated with
\(\tau_i\). Arthur's character satisfies
\[
\epsilon_{X_1}^{\mathrm{Art}}(a_i)
=
\epsilon_X^{\mathrm{Art}}(a_i)
\epsilon\left(\frac12,\sigma\times\tau_i\right),
\qquad
\epsilon_{X_1}^{\mathrm{Art}}(a_\sigma)
=
\prod_i\epsilon\left(\frac12,\sigma\times\tau_i\right).
\]
The assumption \eqref{eq:general-central-assumption} implies
\(L(\frac12,\sigma\times\tau_i)\ne0\) for every \(i\). Since
\(\sigma\) and \(\tau_i\) are self-dual, the functional equation
gives
\(\epsilon(\frac12,\sigma\times\tau_i)=1\). Hence
\(\epsilon_{X_1}^{\mathrm{Art}}
=\epsilon_X^{\mathrm{Art}}\circ q\). In the metaplectic case,
\(\nu_{X_1}\) and \(\nu_X\circ q\) agree on the old generators and
\(
\nu_{X_1}(a_\sigma)
=
\epsilon\left(\frac12,\sigma\right)^2
=
1
=
(\nu_X\circ q)(a_\sigma).
\)
This proves \eqref{eq:multiplicity-character-compatibility-Xd} for
\(d=1\).

Suppose now that \(d=2\). Under the natural identification
\(\mathcal S_{X_2}\simeq\mathcal S_{X_2^{-}}\), the Arthur
characters agree because all summands of \(X\) have
\(SL_2(\CC)\)-dimension \(1\) and
\(\min(3,1)=\min(1,1)=1\). Equivalently, the only additional
Rankin--Selberg root numbers compare self-dual cuspidal
constituents of the same symplectic type and are equal to \(1\);
see \cite[Theorem~1.5.3(b)]{Ae}. In the metaplectic case, the factors
\(\nu_{X_2}\) and \(\nu_{X_2^{-}}\) agree as well, since on the
generator associated with \(\sigma\),
\(
\epsilon\left(\frac12,\sigma\right)^3
=
\epsilon\left(\frac12,\sigma\right),
\)
and they agree trivially on the generators coming from \(X\).
Thus \eqref{eq:multiplicity-character-compatibility-Xd} also holds
for \(d=2\).

Since \(\Pi\) is automorphic, the global multiplicity formula gives
\(\eta_\Pi=\epsilon_{X_d}^G\); see
\cite[Theorem~1.5.2]{Ae} in the symplectic case
(resp. \cite[Theorem~5.4.1]{Li} in the metaplectic case).
Together with \eqref{eq:eta-compatibility-Xd} and
\eqref{eq:multiplicity-character-compatibility-Xd}, this yields
\(\eta_\rho\circ q=\epsilon_{X_d^{-}}^G\circ q\). Since \(q\) is
surjective,
\(
\eta_\rho=\epsilon_{X_d^{-}}^G.
\)
The multiplicity formula for \(X_d^{-}\) shows that \(\rho\) is
automorphic with global \(A\)-parameter \(X_d^{-}\). Since
\(X_d^{-}\) is tempered,
Lemma~\ref{tempered-parameter-cuspidal} shows that \(\rho\) is
cuspidal.

It remains to prove that the relevant residual representation is
nonzero. Suppose first that \(d=1\). The Rankin--Selberg
\(L\)-function attached to \(\sigma\) and \(\rho\) depends only on
\(X_1^{-}=X\). In the symplectic case, the fact that
\(X_1=\sigma\boxtimes[1]\boxplus X\) is a parameter for a
symplectic group implies that \(\sigma\) is of symplectic type.
Thus \(L(s,\sigma,\wedge^2)\) has a pole at \(s=1\), and
\eqref{eq:general-central-assumption} together with
Proposition~\ref{p2} gives
\(\mathcal E^0(\sigma,\rho)\ne0\). In the metaplectic case,
\(\sigma\) is of orthogonal type, so
\(L(s,\sigma,\Sym^2)\) has a pole at \(s=1\). Since
\(
L_{\psi}\left(\frac12,\sigma\times\rho\right)
=
L\left(\frac12,\sigma\times X\right),
\)
the same conclusion follows from Proposition~\ref{p4}.

Suppose now that \(d=2\). Since
\(X_2^{-}=\sigma\boxplus X\), the representation \(\sigma\) is an
isobaric summand of the global parameter of \(\rho\). Hence
Proposition~\ref{p3} in the symplectic case
(resp. Proposition~\ref{p5} in the metaplectic case) gives
\(\mathcal E^1(\sigma,\rho)\ne0\). Thus
\(\mathcal E^{d-1}(\sigma,\rho)\ne0\) in either case.

Since \(\sigma\) is cuspidal on \(\GL_a\), all \(\sigma_v\) are
generic. Lemma~\ref{localrealization} and
\eqref{eq:local-associated-member-voganres} now show that \(\Pi\)
occurs as an irreducible direct summand of
\(\mathcal E^{d-1}(\sigma,\rho)\).
\end{proof}

To obtain nonemptiness, choose an automorphic representation
\(\rho\) with global \(A\)-parameter \(X_d^{-}\), as supplied by
the global multiplicity formula; see \cite[Theorem~1.5.2]{Ae} in
the symplectic case and \cite[Theorem~5.4.1]{Li} in the
metaplectic case. It is cuspidal by
Lemma~\ref{tempered-parameter-cuspidal}. Under the hypotheses below,
Propositions~\ref{p2}--\ref{p5} give
\(\mathcal E^{d-1}(\sigma,\rho)\ne0\), and
Lemma~\ref{localrealization} therefore supplies an automorphic
member of the associated global \(L\)-packet. Combining this with
Proposition~\ref{voganres}, we obtain the following corollary.

\begin{cor}
\label{voganrescor}
Let \(d\in\{1,2\}\), and let \(X_d,X_d^{-}\) be as in
\eqref{general-parameters-Xd}. Assume that \(X_d^{-}\) is a
discrete tempered global \(A\)-parameter for \(G_k\). If \(d=1\),
assume in addition that
\(
L\left(\frac12,\sigma\times X\right)\ne0.
\)
Then
\[
\varnothing\ne
\Pi_{\phi_{X_d}}^{\mathrm{aut}}
\subset
\Irr_{\mathrm{res}}(G_{k+a}).
\]
In particular,
\[
\Pi_{\phi_{X_d}}^{\mathrm{aut}}
\cap
\Irr_{\mathrm{cusp}}(G_{k+a})
=
\varnothing.
\]
Thus no additional central-value hypothesis is required when
\(d=2\), whereas the above nonvanishing condition is required by
the present argument when \(d=1\).
\end{cor}

\begin{rem}
\label{skcounter}
The central nonvanishing hypothesis in the case \(d=1\) cannot in
general be omitted. This is already visible in the
Saito--Kurokawa case. Let \(\pi\) be a cuspidal automorphic
representation of \(\mathrm{PGL}_2(\A)\) such that
\(
\epsilon\left(\frac12,\pi\right)=1,\ L\left(\frac12,\pi\right)=0,
\)
and consider the Saito--Kurokawa parameter
\(
M_\pi
=
\pi\boxtimes[1]\boxplus\mathbf 1.
\)
For example, one may take \(\pi\) associated with a weight-two
newform of analytic rank two. The local Saito--Kurokawa
representation at every place is the nontempered Langlands
quotient determined by \(\phi_{M_\pi,v}\); see
\cite[Lemma~2.2]{SchmidtSK}. Nevertheless,
\cite[Theorem~3.1]{SchmidtSK}, applied with \(S=\varnothing\),
shows that the resulting global representation is cuspidal when
\(L(\frac12,\pi)=0\). After passage from
\(\mathrm{PGSp}_4\) to the corresponding symplectic packet, this
gives a cuspidal automorphic member of the associated global
\(L\)-packet of \(M_\pi\).

By contrast, when \(L(\frac12,\pi)\ne0\), the member with
\(S=\varnothing\) belongs to the residual spectrum. Thus the
vanishing of the central value need not make the associated
automorphic \(L\)-packet empty: it may replace its residual
realization by a cuspidal CAP realization. Consequently, the
conclusion of Corollary~\ref{voganrescor} may fail if its
\(d=1\) central nonvanishing hypothesis is removed.
\end{rem}

The preceding corollary suggests the following multiblock
generalization.

\begin{con}
\label{multiblockres}
Let \(M_{\mathrm{core}}\) be a discrete tempered global
\(A\)-parameter for \(\Sp_{2n_0}\), and let
\(\sigma_i\) be pairwise inequivalent irreducible cuspidal
automorphic representations of
\(\GL_{a_i}(\A)\) of symplectic type, for \(1\le i\le r\), where
\(r\ge1\). Put
\(
n=n_0+\sum_{i=1}^{r}a_i
\)
and assume that
\(
M
=
\left(
\boxplus_{i=1}^{r}\sigma_i\boxtimes[1]
\right)
\boxplus M_{\mathrm{core}}
\)
is a discrete global \(A\)-parameter for \(\Sp_{2n}\). If
\(
\prod_{i=1}^{r}
L\left(
\frac12,\sigma_i\times M_{\mathrm{core}}
\right)
\ne0,
\)
then every automorphic member of the associated global
\(L\)-packet is residual; equivalently,
\[
\Pi_{\phi_M}^{\mathrm{aut}}
\subset
\Irr_{\mathrm{res}}(\Sp_{2n}).
\]
In particular,
\[
\Pi_{\phi_M}^{\mathrm{aut}}
\cap
\Irr_{\mathrm{cusp}}(\Sp_{2n})
=
\varnothing.
\]
\end{con}

\subsubsection{The GGP implication under central nonvanishing}
\label{residual-ggp-central-nonvanishing}

We now specialize the preceding results to the relevant pair of
parameters in Conjecture~\ref{conn}~(ii). Let \(n\geq m\), and let
\[
M_1=\boxplus_{i=1}^{t_1}\sigma_i,
\qquad
N_1=\boxplus_{j=1}^{t_2}\sigma_j'
\]
be discrete tempered global \(A\)-parameters for \(\Sp_{2n}\) and
\(\Mp_{2m}\), respectively. Thus each \(\sigma_i\) is of
orthogonal type, whereas each \(\sigma_j'\) is of symplectic type.
Lemma~\ref{automatic-almost-tempered} supplies all local
almost-temperedness conditions needed below; no local Ramanujan
hypothesis is imposed.

Fix an irreducible cuspidal automorphic representation \(\sigma\)
of \(\GL_a(\A)\) such that \(\sigma\simeq\sigma_{j_0}'\) for some
\(j_0\), and put
\[
N_0
=
\boxplus_{\substack{1\leq j\leq t_2\\j\ne j_0}}\sigma_j',
\qquad
N_1=\sigma\boxplus N_0.
\]
Define
\[
M=\sigma\boxtimes[1]\boxplus M_1,
\qquad
N=\sigma\boxtimes[2]\boxplus N_0.
\]
Then \(M\) is a discrete global \(A\)-parameter for
\(\Sp_{2n+2a}\), and \(N\) is one for \(\Mp_{2m+2a}\). Moreover,
\((M,N)\) is a relevant pair in the sense of
Definition~\ref{drel}. Their associated global \(L\)-parameters are
\[
\phi_M
=
\sigma|\cdot|^{1/2}\boxplus M_1
\boxplus\sigma^\vee|\cdot|^{-1/2},
\qquad
\phi_N
=
\sigma|\cdot|\boxplus N_1
\boxplus\sigma^\vee|\cdot|^{-1}.
\]

Assume that
\begin{equation}
\label{centralass}
L\left(\frac12,\sigma\times M_1\right)
=
\prod_{i=1}^{t_1}
L\left(\frac12,\sigma\times\sigma_i\right)
\ne0.
\end{equation}

We are now ready to prove the direction
\((1)\Rightarrow(2)\) in Conjecture~\ref{conn}~(ii).

\begin{thm}
\label{non}
Let
\(\pi_1\in\Pi_{\phi_M}^{\mathrm{aut}}\) and
\(\pi_2\in\Pi_{\phi_N}^{\mathrm{aut}}\). If
\(
\mathcal FJ_{\psi}
\left(
\pi_1,\pi_2,\nu_{\psi^{-1},W_{m+a}}
\right)
\ne0,
\)
then \(\left.L(z,M,N)\right|_{z=0}\ne0\).
\end{thm}

\begin{proof}
Apply Proposition~\ref{voganres} with
\((d,X,G)=(1,M_1,\Sp)\) and
\((d,X,G)=(2,N_0,\Mp)\). We obtain irreducible cuspidal
automorphic representations
\(
\rho\in\Pi_{\phi_{M_1}}^{\mathrm{aut}},\ \widetilde{\rho}\in\Pi_{\phi_{N_1}}^{\mathrm{aut}}
\)
such that \(\pi_1\) and \(\pi_2\) occur as irreducible direct
summands of \(\mathcal E^0(\sigma,\rho)\) and
\(\mathcal E^1(\sigma,\widetilde{\rho})\), respectively.

At this point it matters that \(\mathcal{FJ}_\psi\) is a functional on
spaces of automorphic forms and not on abstract representations.  The
hypothesis provides the nonvanishing of \(\mathcal{FJ}_\psi\) on the
automorphic realization of \(\pi_1\boxtimes\pi_2\), whereas
Proposition~\ref{voganres} produces a copy of \(\pi_1\boxtimes\pi_2\)
inside the two residual representations; for the hypothesis to be of any
use, these two spaces of automorphic forms must coincide.  This is
exactly what multiplicity one supplies: each of \(\pi_1\) and \(\pi_2\)
occurs exactly once in the discrete automorphic spectrum, by Arthur's
multiplicity formula in the symplectic case and by
\cite[Corollary~5.4.4(ii)]{Li} in the metaplectic case, so its
automorphic realization is unique.  We may therefore identify \(\pi_1\)
and \(\pi_2\) with the corresponding direct summands of the two residual
representations. Under these
identifications, the regularized
Fourier--Jacobi period on the two ambient residual representations
restricts to the given Fourier--Jacobi period on
\(\pi_1\boxtimes\pi_2\). Since this restriction is nonzero, we
obtain
\[
\mathcal FJ_{\psi}
\left(
\mathcal E^0(\sigma,\rho),
\mathcal E^1(\sigma,\widetilde{\rho}),
\nu_{\psi^{-1},W_{m+a}}
\right)
\ne0.
\]

Using the Clebsch--Gordan and symmetric/exterior-square identities
in Fact~\ref{fac}, with the roles of the orthogonal and symplectic
types interchanged, we obtain
\begin{align*}
M\otimes N
&=
(\sigma\otimes\sigma)
\boxtimes\bigl([3]\boxplus[1]\bigr)
\boxplus(\sigma\otimes N_0)\boxtimes[1]\\
&\quad\boxplus(M_1\otimes\sigma)\boxtimes[2]
\boxplus(M_1\otimes N_0),\\
\wedge^2(M)
&=
\Sym^2(\sigma)
\boxplus\wedge^2(\sigma)\boxtimes[2]
\boxplus(\sigma\otimes M_1)\boxtimes[1]
\boxplus\wedge^2(M_1),\\
\Sym^2(N)
&=
\Sym^2(\sigma)\boxtimes\bigl([4]\boxplus[0]\bigr)
\boxplus\wedge^2(\sigma)\boxtimes[2]\\
&\quad\boxplus(\sigma\otimes N_0)\boxtimes[2]
\boxplus\Sym^2(N_0).
\end{align*}
Here \(M\) is of orthogonal type and \(N\) is of symplectic type.
Thus the opposite-orientation convention gives
\begin{equation}
\label{eq:opposite-orientation-LMN}
L(z,M,N)
=
\frac{L(z+\frac12,M\times N)}
{L(z+1,\wedge^2(M))\,L(z+1,\Sym^2(N))}.
\end{equation}
Put
\(A=\ord_{s=1/2}L(s,M_1\times\sigma)\) and
\(B=\ord_{s=1/2}L(s,M_1\times N_0)\), where a positive order
denotes a zero and a negative order a pole. The standard analytic
properties of Rankin--Selberg, exterior-square, and
symmetric-square \(L\)-functions give
\begin{align*}
\ord_{z=0}L\left(z+\frac12,M\times N\right)&=-4+A+B,\\
\ord_{z=0}L\left(z+1,\wedge^2(M)\right)&=-2+A,\\
\ord_{z=0}L\left(z+1,\Sym^2(N)\right)&=-2.
\end{align*}
Therefore
\begin{equation}
\label{ordratio}
\ord_{z=0}L(z,M,N)=B.
\end{equation}

By Theorem~\ref{rthm}, the nonvanishing of the residual
Fourier--Jacobi period implies
\(
\mathcal FJ_{\psi}
\left(
\rho,\widetilde{\rho},\nu_{\psi^{-1},W_m}
\right)
\ne0.
\)
Theorem~\ref{e} now gives
\[
L_{\psi}\left(\frac12,\rho\times\widetilde{\rho}\right)
=
L\left(\frac12,M_1\times N_1\right)
\ne0.
\]
Since \(N_1=\sigma\boxplus N_0\), this value equals
\[
L\left(\frac12,M_1\times\sigma\right)
L\left(\frac12,M_1\times N_0\right).
\]
Hence \(L(\frac12,M_1\times N_0)\ne0\), so \(B=0\).
Equation~\eqref{ordratio} yields
\(\ord_{z=0}L(z,M,N)=0\). Thus \(L(z,M,N)\) is holomorphic and
nonzero at \(z=0\).
\end{proof}

\begin{rem}\label{conb}
We have treated the case in which \(M\) is a parameter for the
symplectic group and \(N\) is a parameter for the metaplectic
group. The opposite case is proved in the same manner, with the
roles of the orthogonal and symplectic parameters interchanged,
and hence with the roles of the exterior-square and
symmetric-square factors interchanged.
\end{rem}

Theorem~\ref{non} treats the case
\(
L\left(\frac12,\sigma\times M_1\right)\ne0.
\)
This hypothesis is exactly what produces the residual realization of
the associated global \(L\)-packet used in its proof. When the central
value vanishes, that realization is no longer available, and the
associated automorphic \(L\)-packet has a different, cuspidal
realization. It is therefore natural to ask whether
the implication
\((1)\Rightarrow(2)\) in Conjecture~\ref{conn}~(ii) remains valid for
such a cuspidal member. In the next subsection we answer this question
for the basic metaplectic parameter
\(M=[1]\boxplus M'\) satisfying \(L\left(\frac12,M'\right)=0\).

\subsection{The Fourier--Jacobi period for central-zero associated
global \(L\)-packets}
\label{gen}

We now turn to the complementary central-vanishing situation
described at the end of the preceding subsection. In this case the
associated automorphic \(L\)-packet is nonempty and all of its members
are cuspidal. Accordingly, the result below is formulated for an
arbitrary automorphic member of the packet.

\begin{thm}
\label{non0}
Let \(n\geq1\). Let \(M'\) be a discrete tempered symplectic global
\(A\)-parameter of dimension \(2n-2\), let \(N\) be a discrete
tempered global \(A\)-parameter for \(\Sp_{2n}\), and put
\(
M=[1]\boxplus M'.
\)
Assume that
\(
L\left(\frac12,M'\right)=0.
\)
Let
\(
\widetilde\pi
\in
\Pi_{\phi_M}^{\mathrm{aut}},  \pi\in\Pi_{\phi_N}^{\mathrm{aut}}.
\)

If
\(
\mathcal FJ_\psi
\left(
\widetilde\pi,\pi,\omega_{\psi^{-1},W_n}
\right)\ne0,
\)
then \((M,N)\) is a relevant pair and
\[
\left.L(z,M,N)\right|_{z=0}\ne0.
\]
\end{thm}

The proof uses the relation between Bessel and Fourier--Jacobi
periods provided by the global theta correspondence. We first fix the
notation and isolate the theta-lift realization of
\(\widetilde\pi\).

Let \((V,(\ ,\ )_V)\) be a quadratic space and let
\((W,\langle\ ,\ \rangle_W)\) be a symplectic space over \(F\). We
equip \(V\otimes W\) with the symplectic form
\(
(v_1\otimes w_1,v_2\otimes w_2)
=
(v_1,v_2)_V\langle w_1,w_2\rangle_W.
\)
The restriction of the Weil representation of
\(\Mp(V\otimes W)(\A)\) gives the usual theta correspondence for the
dual pair \(\mathrm O(V)\times\Mp(W)\). When \(\dim V\) is even, the
metaplectic cover splits over \(\Sp(W)(\A)\). For any nontrivial
additive character \(\eta\) of \(F\backslash\A\), we write
\(\Theta_{\eta,V,W}\) for the resulting global theta lift.

Strictly speaking, the orthogonal member of the dual pair is
\(\mathrm O(V)\). Whenever a lift is viewed as a representation of
\(\SO(V)\), we choose an irreducible constituent of its restriction.
Conversely, we choose an extension to \(\mathrm O(V)\) that
participates in the nonzero theta correspondence. This
convention will also be used in the seesaw identity below.

Let \(\mc X\) be a Lagrangian subspace of \(V\otimes W\). For
\(f\in\mc S(\mc X(\A))\), let
\(\Theta_{\eta,V,W}(f)\) be the corresponding global theta
function. For a dual pair \((G,H)\), if  \(\tau\) is a cuspidal automorphic representation of \(G\),  its global theta lift is generated by the
absolutely convergent integrals
\begin{equation}
\label{strange}
\theta(f,\varphi)(h)
=
\int_{[\mathrm G]}
\Theta_{\eta,V,W}(f)(g,h)\varphi(g)\,dg.
\end{equation}

\begin{prop}[{\cite[Proposition~3.1]{Gan}}]
\label{Gan}
If a global theta lift of a cuspidal representation is cuspidal, then
it is irreducible and isomorphic to the restricted tensor product of
the corresponding local theta lifts.
\end{prop}

We first determine the spectral nature of the associated packet and
show that it is nonempty. This is the point at which the vanishing of
the central \(L\)-value replaces the residual realization of the
preceding subsection by a cuspidal theta-lift realization.

\begin{prop}
\label{central-zero-associated-packet}
Let \(M'\) be a discrete tempered symplectic global \(A\)-parameter
of dimension \(2n-2\), and put
\(
M=[1]\boxplus M'.
\)
If
\(
L\left(\frac12,M'\right)=0,
\)
then
\[
\varnothing\ne
\Pi_{\phi_M}^{\mathrm{aut}}
\subset
\Irr_{\mathrm{cusp}}(\Mp_{2n}).
\]
Moreover, \(\Pi_{\phi_M}^{\mathrm{aut}}\) contains a globally
generic member.
\end{prop}

\begin{proof}
We first prove cuspidality. Let
\(\widetilde\pi\in\Pi_{\phi_M}^{\mathrm{aut}}\), and suppose that
\(\widetilde\pi\) is not cuspidal. Being discrete automorphic, it is
then residual: by Langlands' spectral decomposition
\cite[Chapter~VII]{la2}, \cite[Chapters~VI--VII]{Mo}, there are a
proper parabolic subgroup \(\wt\P=\wt\M\U\) of \(\Mp_{2n}\) with
\[
 \wt\M\simeq\wt\GL_{a_1}\times\cdots\times\wt\GL_{a_t}
 \times_{\{\pm1\}}\Mp_{2n_0},
 \qquad t\ge1,
\]
an irreducible \emph{cuspidal} automorphic representation
\[
 \varrho=(\varrho_1)_\psi\boxtimes\cdots\boxtimes(\varrho_t)_\psi
 \boxtimes\varrho_0
\]
of \(\wt\M(\A)\), and a point
\(\lambda_0=(c_1,\ldots,c_t)\) with \(c_1>\cdots>c_t>0\), such that
\(\widetilde\pi\) is an iterated residue at \(\lambda_0\) of the
associated Eisenstein series. The positivity of the \(c_i\) is
Langlands' square-integrability criterion
\cite[Lemma~I.4.11]{Mo}; the covering group causes no difficulty,
since it splits canonically over unipotent radicals (compare
\cite[Sects.~3.1--3.3]{LiTrace}).

In particular \(\widetilde\pi\) is a subquotient of the
representation induced from \(\varrho\otimes\lambda_0\), so at every
place where all data are unramified the Satake parameters of
\(\widetilde\pi_v\) are those of that induced representation. By the
decomposition of the genuine discrete spectrum by global Arthur
parameters \cite[Theorems~5.1.1 and~5.4.1]{Li},
\(\varrho_0\) has a discrete global \(A\)-parameter \(Y_0\). Write
\(\operatorname{Std}(Y_0)\) for the standard isobaric automorphic
representation attached to the associated global \(L\)-parameter
\(\phi_{Y_0}\); at almost all unramified places its Satake parameter
is the standard Satake parameter of \(\varrho_{0,v}\). Comparing the
standard parameters at almost all places and invoking strong
multiplicity one together with the uniqueness of isobaric
decompositions \cite{JS0}, we obtain
\[
 \BIGboxplus_{i=1}^{t}
 \left(\varrho_i|\cdot|^{c_i}\boxplus
 \varrho_i^\vee|\cdot|^{-c_i}\right)
 \boxplus \operatorname{Std}(Y_0)
 =
 \tbf1|\cdot|^{1/2}\boxplus M'\boxplus\tbf1|\cdot|^{-1/2}.
\]
Since \(M'\) is tempered, the only non-unitary summands on the right
are the two one-dimensional summands
\(\tbf1|\cdot|^{\pm1/2}\). The displayed \(2t\) summands contributed
by the proper Levi on the left are all non-unitary because
\(c_i>0\). Uniqueness of the isobaric decomposition therefore gives
\(2t\le2\). Since \(t\ge1\), it follows that \(t=1\), and comparison
of the two non-unitary summands gives
\(
a_1=1,\ \varrho_1=\tbf1,\ c_1=\frac12.
\)
The remaining summands then give
\(\operatorname{Std}(Y_0)=M'\). In particular \(Y_0\) is tempered,
and uniqueness of the global Arthur decomposition gives
\(Y_0=M'\) and \(n_0=n-1\).

Thus \(\widetilde\pi\) occurs in the residual representation
\(\mathcal E^0(\tbf1,\varrho_0)\) attached to the cuspidal datum
\((\tbf1)_\psi\boxtimes\varrho_0\) on \(\wt\P_{n,1}\), where
\(\varrho_0\) is cuspidal on \(\Mp_{2n-2}(\A)\) with tempered global
\(A\)-parameter \(M'\). Proposition~\ref{p4} applies with
\(\sigma=\tbf1\) and \(\pi'=\varrho_0\); as
\(L(z,\tbf1,\Sym^2)=\zeta_F(z)\) does have a pole at \(z=1\), it
shows that \(\mathcal E^0(\tbf1,\varrho_0)\ne0\) forces
\[
 L_\psi\left(\tfrac12,\tbf1\times\varrho_0\right)
 =L\left(\tfrac12,M'\right)\ne0,
\]
contradicting the hypothesis. Hence every automorphic member of the
associated global \(L\)-packet is cuspidal.

It remains to prove that the packet is nonempty. By automorphic
descent \cite[Theorem~10.1]{GRS}, there is a globally generic
irreducible cuspidal automorphic representation
\(
\tau\subset
\mathcal A_{\mathrm{cusp}}\left(\SO(V_{2n-1})\right)
\)
of the split odd special orthogonal group whose standard global
parameter is \(M'\). Apply \cite[Theorem~5.2]{HK} with
\(\epsilon=1\). Since
\(
L\left(\frac12,\tau\right)
=
L\left(\frac12,M'\right)
=0,
\)
the global theta lift of \(\tau\) to \(\Mp(W_{n-1})\) vanishes.
On the other hand, \cite[Corollary~4.3]{HK} shows that
\(
\widetilde\pi_0
\coloneqq
\Theta_{\psi,V_{2n-1},W_n}(\tau)
\)
is nonzero and globally generic; because the preceding lift in the
tower vanishes, it is also cuspidal.

We verify that \(\widetilde\pi_0\) belongs to the associated global
\(L\)-packet. Proposition~\ref{Gan} identifies it with the restricted
tensor product of its local theta lifts. At every non-archimedean
place \(v\), write the extension of \(\tau_v\) used in the theta
correspondence as the Langlands quotient
\[
\operatorname{LQ}\bigl(
\delta_{r,v}|\det|^{s_{r,v}}\times\cdots\times
\delta_{1,v}|\det|^{s_{1,v}}\rtimes\tau_{0,v}\bigr),
\qquad \tau_{0,v}\ \text{tempered}.
\]
By Lemma~\ref{automatic-almost-tempered}, \(0<s_{i,v}<\frac12\).
The explicit higher-lift formula
\cite[Theorem~6.1]{BHtc}, applied with the roles interchanged, shows
that
\[
\widetilde\pi_{0,v}
\simeq
\operatorname{LQ}\bigl(
|\cdot|_v^{1/2}\times
\delta_{r,v}|\det|^{s_{r,v}}\times\cdots\times
\delta_{1,v}|\det|^{s_{1,v}}
\rtimes\theta_0(\tau_{0,v})\bigr).
\]
Here the quadratic twists in the notation of \cite{BHtc} are trivial
for the present split pair.
Since \(\tau_{0,v}\) is tempered,
\cite[Theorem~4.3(3),(4)]{AG} identifies the parameter of
\(\theta_0(\tau_{0,v})\) with that of \(\tau_{0,v}\). Hence
\[
\phi_{\widetilde\pi_{0,v}}
=
|\cdot|_v^{1/2}
\boxplus\phi_{M',v}
\boxplus|\cdot|_v^{-1/2}
=
\phi_{M,v}.
\]
At a real place, the equal-rank metaplectic
Langlands--Vogan parametrization is the one defined by theta
correspondence in \cite[Theorem~11.1]{Gan2}, based on
\cite{AB}; compatibility with Langlands classification and the
archimedean induction principle gives the same higher-lift formula.
At a complex place the corresponding assertion follows from
\cite{ABc}. Consequently,
\(
\widetilde\pi_{0,v}
\in
\Pi_{\phi_{M,v}}
\text{ for every place }v.
\)
Since \(\widetilde\pi_0\) is automorphic, this proves
\(\widetilde\pi_0\in
\Pi_{\phi_M}^{\mathrm{aut}}\), and hence proves both the
nonemptiness and the final assertion.
\end{proof}

The preceding proposition constructs one globally generic member by
theta lifting. For Theorem~\ref{non0}, however, the given
\(\widetilde\pi\) is an arbitrary automorphic member of the
associated packet, so we need the converse realization furnished by
the next lemma.

\begin{lem}
\label{theta-realization-associated}
Let \(M'\), \(M\), and \(\widetilde\pi\) be as in
Theorem~\ref{non0}. There exist a quadratic space \(V_{2n-1}\) of
dimension \(2n-1\) and an irreducible cuspidal automorphic
representation
\(
\tau\subset
\mathcal A_{\mathrm{cusp}}
\left(\SO(V_{2n-1})\right)
\)
with global \(A\)-parameter \(M'\) such that
\(
\widetilde\pi
\simeq
\Theta_{\psi,V_{2n-1},W_n}(\tau).
\)
Here the theta lift is understood using the extension of \(\tau\) to
\(\mathrm O(V_{2n-1})(\A)\) selected by the theta correspondence.
\end{lem}

\begin{proof}
By Proposition~\ref{central-zero-associated-packet},
\(\widetilde\pi\) is cuspidal. The associated global
\(L\)-parameter of \(M=[1]\boxplus M'\) is
\(
\phi_M
=
|\cdot|^{1/2}\boxplus\phi_{M'}
\boxplus|\cdot|^{-1/2}.
\)
Since \(\widetilde\pi\in\Pi_{\phi_M}^{\mathrm{aut}}\subset\Pi_M\) and the
only summand of \(M\) carrying a nontrivial \(\SL_2(\CC)\)-factor is
\(\tbf1\boxtimes[1]\), Proposition~\ref{Leq} applies and shows that the
completed standard \(L\)-function of \(\widetilde\pi\) occurring in
Yamana's theta-lifting criterion is
\begin{equation}
\label{eq:standard-L-associated-cusp}
L^{\mathrm{std}}(s,\widetilde\pi)
=
L(s,M)
=
\zeta_F\left(s+\frac12\right)
L(s,M')
\zeta_F\left(s-\frac12\right).
\end{equation}
The right-hand side has a pole at \(s=-\frac12\) and has no pole at
the points farther to the left that occur in the same parity tower.
Indeed, the pole at \(-\frac12\) is contributed by
\(\zeta_F(s+\frac12)\), while the remaining completed factors are
nonzero there.

We apply Yamana's theorem with the trivial quadratic character. In the
notation of \cite[Theorem~10.1]{Y0}, the symplectic space has dimension
\(2n\), so that \(\rho_{2n}=2n+1\), and an orthogonal space of dimension
\(2n-1\) corresponds to the point
\(
\frac12\bigl((2n-1)-\rho_{2n}+1\bigr)=-\frac12,
\)
which is exactly the pole located above; moreover the trivial quadratic
character fixes the discriminant tower \(\disc(V_{2n-1})=1\).  Thus
\cite[Theorem~10.1]{Y0}, together with the tower property
\cite[Proposition~10.1]{Y0}, gives a quadratic space \(V_{2n-1}\)
for which
\(
\widehat\tau
\coloneqq
\Theta_{\psi,W_n,V_{2n-1}}(\widetilde\pi)
\)
is nonzero. At almost all places, the unramified theta
correspondence removes the two outer characters
\(|\cdot|^{1/2}\) and \(|\cdot|^{-1/2}\) from \(\phi_{M,v}\).
Equivalently, this follows from Kudla's supercuspidal-support theorem
\cite{Ku1,Ku2}. Hence every irreducible constituent \(\tau\) of the
restriction of \(\widehat\tau\) to
\(\SO(V_{2n-1})(\A)\) has global \(A\)-parameter \(M'\), and is
therefore cuspidal by Lemma~\ref{tempered-parameter-cuspidal}. Thus
\(\widehat\tau\) is a nonzero cuspidal global theta lift of the
cuspidal representation \(\widetilde\pi\), and is irreducible by
Proposition~\ref{Gan}. The involutive property of the global theta
correspondence for the odd orthogonal--metaplectic dual pair
\cite[Theorem~1.3]{JiangSoudryTheta} now gives
\[
\Theta_{\psi,V_{2n-1},W_n}(\widehat\tau)
\simeq \widetilde\pi.
\]
Taking \(\tau\) to be an irreducible constituent of
\(\widehat\tau|_{\SO(V_{2n-1})(\A)}\), and \(\widehat\tau\) as the
extension selected by the theta correspondence, proves the lemma.
\end{proof}

We can now transfer the Fourier--Jacobi period to an orthogonal
Bessel period by means of the theta realization and a seesaw
identity.

\begin{proof}[Proof of Theorem~\ref{non0}]
Choose \(V_{2n-1}\) and \(\tau\) as in
Lemma~\ref{theta-realization-associated}. Put
\(
V_1=\langle-1\rangle,
\qquad
V_{2n}=V_{2n-1}\oplus V_1.
\)
For a one-dimensional quadratic space \(\langle a\rangle\), our
discriminant convention gives
\(
\disc\bigl(V_{2n-1}\oplus\langle a\rangle\bigr)
=-a\,\disc(V_{2n-1}).
\)
Hence \(\disc(V_{2n})=1\), and consequently
\(\chi_{V_{2n}}=\mathbf1\). Moreover,
\(\psi_{-1}(x)=\psi(-x)=\psi(x)^{-1}\), and therefore
\(
\omega_{\psi,V_1,W_n}
\simeq
\omega_{\psi^{-1},W_n}.
\)
Thus the choice \(V_1=\langle-1\rangle\) simultaneously matches the
Weil factor in the Fourier--Jacobi period and the trivial
discriminant character needed below. In particular,
\(\SO(V_{2n})\times\SO(V_{2n-1})\) is the corresponding relevant pair
of pure inner forms.

The seesaw identity is taken for the full orthogonal groups:
\[
\xymatrix{
  \mathrm O(V_{2n}) \ar @{-}[d]\ar @{-}[dr] &
  \Mp(W_n)\times_{\{\pm1\}}\Mp(W_n) \ar @{-}[dl]\ar @{-}[d] \\
  \mathrm O(V_{2n-1})\times \mathrm O(V_1) & \Sp(W_n).
}
\]
Here \(\mathrm O(V_1)=\{\pm1\}\). We take the
\(\mathrm O(V_1)\)-isotypic component that realizes the automorphic
Weil factor \(\omega_{\psi^{-1},W_n}\), choose the participating
extensions to the two full orthogonal groups, and then restrict to
the selected \(\SO\)-constituents. The resulting identity below is
unchanged up to a nonzero scalar, which is immaterial for
nonvanishing.

Since the theta lifts of vectors in \(\tau\) span
\(\widetilde\pi\), the assumed nonvanishing of the Fourier--Jacobi
period allows us to choose
\[
f'\in\omega_{\psi,V_{2n-1},W_n},
\quad
f\in\omega_{\psi^{-1},W_n}
\simeq\omega_{\psi,V_1,W_n},
\quad
\varphi_\tau\in\tau,
\quad
\varphi_\pi\in\pi
\]
such that
\[
\mathcal FJ_\psi
\left(
\Theta_{\psi,V_{2n-1},W_n}(f',\varphi_\tau),
\varphi_\pi,f
\right)\ne0.
\]
Set
\(
\widetilde f=f'\otimes f
\in
\omega_{\psi,V_{2n},W_n}.
\)
The seesaw identity gives
\begin{align}
&\mathcal FJ_\psi
\left(
\Theta_{\psi,V_{2n-1},W_n}(f',\varphi_\tau),
\varphi_\pi,f
\right)
\nonumber\\
&\quad=
\int_{[\SO(V_{2n-1})]}
\left(
\int_{[\Sp(W_n)]}
\Theta_{\psi,V_{2n},W_n}
(\widetilde f)(\widetilde g,h)
\varphi_\pi(\widetilde g)\,dg
\right)
\varphi_\tau(h)\,dh.
\label{be}
\end{align}
All integrals are absolutely convergent because both
\(\widetilde\pi\) and \(\pi\) are cuspidal. Here the cuspidality of
\(\widetilde\pi\) follows from
Proposition~\ref{central-zero-associated-packet}, whereas that of
\(\pi\) follows from Lemma~\ref{tempered-parameter-cuspidal}, since
\(N\) is a discrete tempered global \(A\)-parameter.

Let
\(
\tau'
\subset
\Theta_{\psi,W_n,V_{2n}}(\pi)
\)
be an irreducible constituent whose restriction to
\(\SO(V_{2n})(\A)\) occurs in the nonzero period in
\eqref{be}. The lift is nonzero. Moreover, it is the first occurrence
of \(\pi\) in the even orthogonal Witt tower containing \(V_{2n}\).
Indeed, choose an unramified finite place \(v\). If
\(\theta_{\psi_v,W_n,V_{2r}}(\pi_v)\neq0\) for some \(r<n\), then
Kudla's supercuspidal-support theorem forces the cuspidal support of
\(\pi_v\) to contain a $\GL_1$-component whose real
exponent is \(n-r\ge1\). This contradicts the almost
temperedness of \(\pi_v\). Hence the local theta lift to \(V_{2r}\)
vanishes for every \(r<n\), and consequently the global lifts to those smaller spaces vanish. By the
tower property, \(\tau'\) is cuspidal, and it is irreducible by
Proposition~\ref{Gan}.

Let \(N'\) be the global \(A\)-parameter of \(\tau'\). Kudla's
supercuspidal-support theorem gives, in general,
\(
N=\chi_{V_{2n}}\boxplus
\bigl(\chi_{V_{2n}}\otimes N'\bigr).
\)
Since \(\chi_{V_{2n}}=\mathbf1\), this becomes
\begin{equation}
\label{eq:N-one-plus-Nprime}
N=\mathbf1\boxplus N'.
\end{equation}
In particular, \(N'\) is a discrete tempered orthogonal global
\(A\)-parameter of dimension \(2n\), and it does not contain
\(\mathbf1\). The summands of \(M'\) are of symplectic type, whereas
those of \(N'\) are of orthogonal type, so no summand can occur in
both. It follows immediately from Definition~\ref{drel}, using
\([-1]=0\), that
\(
M=[1]\boxplus M'\) and
\(N=\mathbf1\boxplus N'
\)
form a relevant pair.

The right-hand side of \eqref{be} is a nonzero Bessel period for an
irreducible cuspidal constituent of
\(\tau'\boxtimes\tau\) on
\(\SO(V_{2n})(\A)\times\SO(V_{2n-1})(\A)\). Theorem~5.7 of
\cite{JZ}, which applies to these relevant pure inner forms and to
their generic global \(A\)-parameters, therefore yields
\begin{equation}
\label{eq:central-Nprime-Mprime}
L\left(\frac12,N'\times M'\right)\ne0.
\end{equation}

It remains to compute \(L(z,M,N)\) at \(z=0\). By
Fact~\ref{fac} and \eqref{eq:N-one-plus-Nprime},
\begin{align*}
M\otimes N
&=
(M'\otimes N')\boxplus M'\boxplus[1]
\boxplus(N'\otimes[1]),\\
\Sym^2(M)
&=
[2]\boxplus\Sym^2(M')\boxplus(M'\otimes[1]),\\
\wedge^2(N)
&=
\wedge^2(N')\boxplus N'.
\end{align*}
Consequently,
\begin{align*}
\ord_{z=0}L\left(z+\frac12,M\times N\right)
&=
-2+\ord_{z=0}\Bigl(
L\left(z+\frac12,M'\times N'\right)
L\left(z+\frac12,M'\right)\\
&\hspace{42mm}\cdot
L(z+1,N')L(z,N')
\Bigr),\\
\ord_{z=0}L\left(z+1,\Sym^2(M)\right)
&=
-2+\ord_{z=0}\Bigl(
L\left(z+1,\Sym^2(M')\right)
L\left(z+\frac32,M'\right)\\
&\hspace{42mm}\cdot
L\left(z+\frac12,M'\right)
\Bigr),\\
\ord_{z=0}L\left(z+1,\wedge^2(N)\right)
&=
\ord_{z=0}\Bigl(
L\left(z+1,\wedge^2(N')\right)L(z+1,N')
\Bigr).
\end{align*}
The factors \(L(z+\frac12,M')\), including their possible zeros,
cancel between the numerator and the denominator. We obtain
\begin{align}
\ord_{z=0}L(z,M,N)
&=
\ord_{z=0}\Bigl(
L\left(z+\frac12,M'\times N'\right)L(z,N')
\Bigr)
\nonumber\\
&\quad-
\ord_{z=0}\Bigl(
L\left(z+1,\Sym^2(M')\right)
L\left(z+1,\wedge^2(N')\right)
L\left(z+\frac32,M'\right)
\Bigr).
\label{eq:order-LMN-cuspidal-associated}
\end{align}

Since \(M'\) is symplectic tempered and \(N'\) is orthogonal
tempered,
\[
L\left(1,\Sym^2(M')\right)\ne0,
\qquad
L\left(1,\wedge^2(N')\right)\ne0.
\]
The discreteness of \(N=\mathbf1\boxplus N'\) implies that \(N'\)
does not contain \(\mathbf1\), and hence \(L(0,N')\ne0\).
Furthermore, \(L(\frac32,M')\ne0\) by absolute convergence.
Combining these facts with
\eqref{eq:central-Nprime-Mprime} in
\eqref{eq:order-LMN-cuspidal-associated}, we obtain
\(
\ord_{z=0}L(z,M,N)=0.
\)
Since \(L(z,M,N)\) is holomorphic at \(z=0\) by
\cite[Theorem~9.7]{Gan1}, it follows that
\[
\left.L(z,M,N)\right|_{z=0}\ne0.
\]
\end{proof}

For a globally generic cuspidal representation of a metaplectic group, its global $A$-parameter is identified with its functorial lift to the corresponding general linear group; see \cite[Theorem~B.1]{AHKO}. If this parameter is non-tempered, then \cite[Theorem~11.2]{GRS} shows that it is necessarily of the form $M = [1] \boxplus M'$, where $M'$ is a discrete tempered symplectic global $A$-parameter. Corollary~\ref{voganrescor}, together with cuspidality, then forces $L\left(\frac{1}{2}, M'\right) = 0$, so Theorem~\ref{non0} applies. Combining this with the tempered case \cite[Theorem~1.1]{Y}, we obtain the following consequence.

\begin{cor}
\label{non1}
Let \(\wt{\pi}\) and \(\pi\) be irreducible generic cuspidal
automorphic representations of \(\Mp_{2n}(\A)\) and
\(\Sp_{2n}(\A)\), respectively. Let \(M\) and \(N\) be the
\(A\)-parameters of \(\wt{\pi}\) and \(\pi\), respectively. If
\(
\mathcal FJ_{\psi}
\left(
\wt{\pi},\pi,\omega_{\psi^{-1},W_n}
\right)
\ne0,
\)
then \((M,N)\) is a relevant pair and
\[
\left.L(z,M,N)\right|_{z=0}\ne0.
\]
\end{cor}

Combining Corollary~\ref{voganrescor},
Proposition~\ref{central-zero-associated-packet}, and the preceding
discussion, we obtain the following dichotomy.

\begin{cor}
\label{central-value-packet-dichotomy}
Let \(n\geq1\), let \(M'\) be a discrete tempered symplectic global
\(A\)-parameter of dimension \(2n-2\), and put
\(
M=[1]\boxplus M'.
\)
Then:
\begin{enumerate}
\item The associated automorphic \(L\)-packet is always nonempty:
\(
\Pi_{\phi_M}^{\mathrm{aut}}\ne\varnothing.
\)
\item If \(L(\frac12,M')=0\), then
\(
\Pi_{\phi_M}^{\mathrm{aut}}
\subset
\Irr_{\mathrm{cusp}}(\Mp_{2n}).
\)
\item If \(L(\frac12,M')\ne0\), then
\(
\Pi_{\phi_M}^{\mathrm{aut}}
\subset
\Irr_{\mathrm{res}}(\Mp_{2n}).
\)
\end{enumerate}
\end{cor}

\subsection*{Acknowledgements} The author expresses heartfelt gratitude to Shunsuke Yamana for many insightful discussions and to Hiraku Atobe for illuminating discussions on Miyawaki lifting and its connection with the GGP conjecture. The author would also like to thank Dongho Byeon, Youngju Choie, Wee Teck Gan, Haseo Ki, Zhengyu Mao, Dipendra Prasad, Sug Woo Shin, and Shunsuke Yamana for their interest in this work, warm encouragement, and generous support. Special thanks are due to the anonymous referees for their valuable comments and suggestions, which significantly improved the clarity of the paper.

This work was undertaken while the author held research fellowships at KAIST, Yonsei University, and Catholic Kwandong University. The author is grateful to these institutions for providing a conducive research environment and generous support. Financial support was provided by the POSCO Science Fellowship and the National Research Foundation of Korea (NRF) grants (nos. 2020R1F1A1A01048645 and RS-2023-00237811). During the preparation of this manuscript, the author used Google's Gemini to improve English language clarity and assist with LaTeX formatting.

\end{document}